\documentclass[11pt, oneside, article]{memoir}

\settrims{0pt}{0pt}
\settypeblocksize{*}{35pc}{*}
\setlrmargins{*}{*}{1}
\setulmarginsandblock{1.1in}{1.4in}{*}
\setheadfoot{\onelineskip}{2\onelineskip}
\setheaderspaces{*}{1.5\onelineskip}{*}
\checkandfixthelayout

\usepackage[centertags,sumlimits,intlimits,namelimits,reqno]{amsmath}
\usepackage{latexsym}
\usepackage{amsfonts,amsthm,amssymb}
\usepackage{mathtools}
\usepackage[cal=euler,scr=rsfso]{mathalfa}
\usepackage[boxslash]{stmaryrd}
\usepackage{newpxtext}
\usepackage{enumitem}
\usepackage[T1]{fontenc}
\usepackage[colorlinks,linkcolor=darkblue,citecolor=darkblue,urlcolor=darkblue,breaklinks=true]{hyperref}
\usepackage{tikz}
\usepackage[capitalize]{cleveref}
\usepackage[backend=biber, maxbibnames=10, style=alphabetic, sortcites=true, bibencoding=utf8, backref, backrefstyle=none]{biblatex}
\DeclareMathAlphabet{\mathpzc}{OT1}{pzc}{m}{it}
\usepackage{xstring}

\newcommand{\bigexists}{%
  \mathop{\lower.9ex\hbox{%
      \scalebox{1.9}{\ensuremath{\exists}}}}\limits}

\allowdisplaybreaks

\usetikzlibrary{
  cd,
  math,
  decorations.markings,
  decorations.pathreplacing,
  positioning,
  circuits.logic.US,
  arrows.meta,
  shapes,
  shadows,
  shadings,
  calc,
  fit,
  quotes,
  intersections,
}

\newif\ifpgfshaperectangleroundnortheast
\newif\ifpgfshaperectangleroundnorthwest
\newif\ifpgfshaperectangleroundsoutheast
\newif\ifpgfshaperectangleroundsouthwest
\pgfkeys{/pgf/.cd,
  rectangle round north east/.is if=pgfshaperectangleroundnortheast,
  rectangle round north west/.is if=pgfshaperectangleroundnorthwest,
  rectangle round south east/.is if=pgfshaperectangleroundsoutheast,
  rectangle round south west/.is if=pgfshaperectangleroundsouthwest,
  rectangle round north east, rectangle round north west,
  rectangle round south east, rectangle round south west,
}
\makeatletter
\def\pgf@sh@bg@rectangle{%
  \pgfkeysgetvalue{/pgf/outer xsep}{\outerxsep}%
  \pgfkeysgetvalue{/pgf/outer ysep}{\outerysep}%
  \pgfpathmoveto{\pgfpointadd{\southwest}{\pgfpoint{\outerxsep}{\outerysep}}}%
  {\ifpgfshaperectangleroundnorthwest\else\pgfsetcornersarced{\pgfpointorigin}\fi%
    \pgfpathlineto{\pgfpointadd{\southwest\pgf@xa=\pgf@x\northeast\pgf@x=\pgf@xa}{\pgfpoint{\outerxsep}{-\outerysep}}}}%
  {\ifpgfshaperectangleroundnortheast\else\pgfsetcornersarced{\pgfpointorigin}\fi%
    \pgfpathlineto{\pgfpointadd{\northeast}{\pgfpoint{-\outerxsep}{-\outerysep}}}}%
  {\ifpgfshaperectangleroundsoutheast\else\pgfsetcornersarced{\pgfpointorigin}\fi%
    \pgfpathlineto{\pgfpointadd{\southwest\pgf@ya=\pgf@y\northeast\pgf@y=\pgf@ya}{\pgfpoint{-\outerxsep}{\outerysep}}}}%
  {\ifpgfshaperectangleroundsouthwest\else\pgfsetcornersarced{\pgfpointorigin}\fi%
    \pgfpathclose}}

\tikzset{
    WD/.style={
      label/.style={
        font=\everymath\expandafter{\the\everymath\scriptstyle},
        inner sep=0pt,
        node distance=2pt and -2pt},
      semithick,
      node distance=\bbx and \bby,
      decoration={markings, mark=at position \stringdecpos with \stringdec},
      bb port length=0,
      bb port sep=.5,
      bbx = .4cm,
      bb min width=.4cm,
      bby = 2ex,
      bb rounded corners=2pt,
      dot size=3pt,
      pack size = 16pt,
      pack psize = 16pt,
      penetration = 0pt,
      link size = 2pt,
      pack color = blue,
      surround sep=2pt,
      ar/.style={postaction={decorate}},
      execute at begin picture={\tikzset{
        x=\bbx, y=\bby,
        circuit logic US, tiny circuit symbols
        }
      }
    },
    bbx/.store in=\bbx,
    bby/.store in=\bby,
    bb port sep/.store in=\bbportsep,
    bb port length/.store in=\bbportlen,
    bb min width/.store in=\bbminwidth,
    bb rounded corners/.store in=\bbcorners,
    bb/.code 2 args={
      \pgfmathsetlengthmacro{\bbheight}{\bbportsep * (max(#1,#2)+1) * \bby}
      \pgfkeysalso{draw,minimum height=\bbheight,minimum
       width=\bbminwidth,outer sep=0pt,
         rounded corners=\bbcorners,thick,
         prefix after command={\pgfextra{\let\fixname\tikzlastnode}},
         append after command={\pgfextra{\draw
            \ifnum #1=0{} \else foreach \i in {1,...,#1} {
              ($(\fixname.north west)!{(2*\i-1)/(2*#1)}!(\fixname.south west)$) +(-\bbportlen,0) coordinate (\fixname_in\i) -- +(\bbportlen,0) coordinate (\fixname_in\i')}\fi
            \ifnum #2=0{} \else foreach \i in {1,...,#2} {
              ($(\fixname.north east)!{(2*\i-1)/(2*#2)}!(\fixname.south east)$) +(-
\bbportlen,0) coordinate (\fixname_out\i') -- +(\bbportlen,0) coordinate (\fixname_out\i)}\fi;
           }}}
    },
  dot size/.store in=\dotsize,
  dot/.style={
    circle, draw, thick, inner sep=0, fill=black, minimum width=\dotsize
  },
  pack size/.store in=\psize,
  pack psize/.store in=\ppsize,
  penetration/.store in=\penetration,
  spacing/.store in=\spacing,
  link size/.store in=\lsize,
  pack color/.store in=\pcolor,
  pack inside color/.store in=\picolor,
  pack inside color=blue!20,
  pack outside color/.store in=\pocolor,
  pack outside color=blue!50!black,
  surround sep/.store in=\ssep,
  link/.style={
    circle,
    draw=black,
    fill=black,
    inner sep=0pt,
    minimum size=\lsize
  },
  pack/.style={
    circle,
    draw = \pocolor,
    fill = \picolor,
    minimum size = \ppsize,
  },
  packs/.style={
    pack,
    pack size=6pt,
    pack psize=6pt,
  },
  func/.style={
    pack,
    rectangle,
    shape border uses incircle,
    rounded corners=.5*\psize,
    inner ysep=.125*\psize,
    minimum width=1.125*\psize,
    inner xsep=.25*\psize,
  },
  funcr/.style={
    func,
    rectangle round north west=false,
    rectangle round south west=false,
  },
  funcl/.style={
    func,
    rectangle round north east=false,
    rectangle round south east=false,
  },
  funcu/.style={
    func,
    rectangle round south east=false,
    rectangle round south west=false,
  },
  funcd/.style={
    func,
    rectangle round north east=false,
    rectangle round north west=false,
  },
  oshell/.style={
    shape border uses incircle,
    kite,
    inner sep=0pt,
    draw = \pocolor,
    fill = \picolor,
    inner ysep=.125*\psize,
    minimum width=.5*\psize,
    inner xsep=.075*\psize,
  },
  oshellr/.style={
    oshell,
    shape border rotate=90,
    },
  oshelll/.style={
    oshell,
    shape border rotate=270,
    },
  oshellu/.style={
    oshell,
    shape border rotate=180,
    },
  oshelld/.style={
    oshell,
    shape border rotate=0,
    },
  outer pack/.style={
    ellipse,
    draw,
    inner sep=\ssep,
    color=gray,
  },
  intermediate pack/.style={
    ellipse,
    dashed,
    draw,
    inner sep=\ssep,
    color=\pocolor,
  },
 }

\tikzset{light gray nodes/.style={every node/.style={fill=gray!40}}}

\tikzset{
  oriented WD/.style={
    every to/.style={out=0,in=180,draw},
    label/.style={
      font=\everymath\expandafter{\the\everymath\scriptstyle},
      inner sep=0pt,
      node distance=2pt and -2pt},
    semithick,
    node distance=1 and 1,
    decoration={markings, mark=at position \stringdecpos with \stringdec},
    ar/.style={postaction={decorate}},
    execute at begin picture={\tikzset{
      x=\bbx, y=\bby,
      every fit/.style={inner xsep=\bbx, inner ysep=\bby}}}
    },
    string decoration/.store in=\stringdec,
    string decoration={\arrow{stealth};},
    string decoration pos/.store in=\stringdecpos,
    string decoration pos=.7,
    bbx/.store in=\bbx,
    bbx = 1.5cm,
    bby/.store in=\bby,
    bby = 1.5ex,
    bb port sep/.store in=\bbportsep,
    bb port sep=1.5,
    bb port length/.store in=\bbportlen,
    bb port length=4pt,
    bb penetrate/.store in=\bbpenetrate,
    bb penetrate=0,
    bb min width/.store in=\bbminwidth,
    bb min width=1cm,
    bb rounded corners/.store in=\bbcorners,
    bb rounded corners=2pt,
    bb spider/.style={
      bb port sep=1, bb port length=10pt, bbx=.4cm, bb min width=.4cm, bby=.8ex},
    bb small/.style={
      bb port sep=1, bb port length=2.5pt, bbx=.4cm, bb min width=.4cm, bby=.7ex},
    bb medium/.style={
      bb port sep=1, bb port length=2.5pt, bbx=.4cm, bb min width=.4cm, bby=.9ex},
    bb/.code 2 args={
      \pgfmathsetlengthmacro{\bbheight}{\bbportsep * (max(#1,#2)+1) * \bby}
      \pgfkeysalso{draw,minimum height=\bbheight,minimum
       width=\bbminwidth,outer sep=0pt,
         rounded corners=\bbcorners,thick,
         prefix after command={\pgfextra{\let\fixname\tikzlastnode}},
         append after command={\pgfextra{\draw
            \ifnum #1=0{} \else foreach \i in {1,...,#1} {
              ($(\fixname.north west)!{(2*\i-1)/(2*#1)}!(\fixname.south west)$) +(-\bbportlen,0) coordinate (\fixname_in\i) -- +(\bbpenetrate,0) coordinate (\fixname_in\i')}\fi
            \ifnum #2=0{} \else foreach \i in {1,...,#2} {
              ($(\fixname.north east)!{(2*\i-1)/(2*#2)}!(\fixname.south east)$) +(-
\bbpenetrate,0) coordinate (\fixname_out\i') -- +(\bbportlen,0) coordinate (\fixname_out\i)}\fi;
           }}}
    },
    bb name/.style={
      append after command={
        \pgfextra{\node[anchor=north] at (\fixname.north) {#1};}
      }
    },
  }

\tikzset{
  unoriented WD/.style={
    every to/.style={draw},
    shorten <=-\penetration, shorten >=-\penetration,
    label distance=-2pt,
    thick,
    node distance=\spacing,
    execute at begin picture={\tikzset{
      x=\spacing, y=\spacing, circuit logic US, tiny circuit symbols}
    }
  },
  pack size/.store in=\psize,
  pack psize/.store in=\ppsize,
  pack size = 8pt,
  pack psize = 8pt,
  penetration/.store in=\penetration,
  penetration = 0pt,
  spacing/.store in=\spacing,
  spacing = 8pt,
  link size/.store in=\lsize,
  link size = 2pt,
  pack color/.store in=\pcolor,
  pack color = blue,
  pack inside color/.store in=\picolor,
  pack inside color=blue!20,
  pack outside color/.store in=\pocolor,
  pack outside color=blue!50!black,
  surround sep/.store in=\ssep,
  surround sep=4pt,
  link/.style={
    circle,
    draw=black,
    fill=black,
    inner sep=0pt,
    minimum size=\lsize
  },
  pack/.style={
    circle,
    draw = \pocolor,
    fill = \picolor,
    inner sep = \ppsize/2,
    minimum size = \ppsize
  },
  func/.style={
    pack,
    rectangle,
    rounded corners=.5*\psize,
    inner ysep=0.3*\psize,
    minimum width=1.125*\psize,
    inner xsep=0.5*\psize,
  },
  funcr/.style={
    func,
    rectangle round north west=false,
    rectangle round south west=false,
  },
  funcl/.style={
    func,
    rectangle round north east=false,
    rectangle round south east=false,
  },
  funcu/.style={
    func,
    rectangle round south east=false,
    rectangle round south west=false,
  },
  funcd/.style={
    func,
    rectangle round north east=false,
    rectangle round north west=false,
  },
  outer pack/.style={
    ellipse,
    draw,
    inner sep=\ssep,
    color=gray,
  },
  to_out/.style={shorten >= -2pt},
  fr_out/.style={shorten <= -2pt},
  intermediate pack/.style={
    ellipse,
    dashed,
    draw,
    inner sep=\ssep,
    color=\pocolor,
   },
  syntax/.style={
    pack color = violet,
    pack inside color=violet!20,
    pack outside color=violet!50!black,
  },
}

\tikzset{
  spider diagram/.style={
    every to/.style={out=0, in=180, draw, thick},
    thick,
    dot size = 5pt,
    execute at begin picture={\tikzset{
      x=\leglen, y=\leglen/3}}
  },
  dot size/.store in=\dotsize,
  dot fill/.store in=\dotfill,
  dot fill = black,
  leg length/.store in=\leglen,
  leg length = 15pt,
  baby/.style={dot size = 2pt, leg length = 6pt},
  young/.style={dot size = 3pt, leg length = 10pt},
  adolescent/.style={dot size = 4pt, leg length = 12pt},
  special spider/.code n args={4}{
    \pgfkeysalso{circle, draw, thick, inner sep=0, fill=\dotfill, minimum width=\dotsize,
      prefix after command={\pgfextra{\let\fixname\tikzlastnode}},
      append after command={\pgfextra{
        \ifnum #1=0{} \else {\foreach \i in {1,...,#1} {
          \tikzmath{\anglei={-90*(#1+1-2*\i)/#1};}
          \draw [thick]
            (\fixname) .. controls
            ($(\fixname.center)-(\anglei:#3/3)$) and ($(\fixname.center)-(\anglei:#3*2/3)$) ..
            ({$(\fixname.center)-(\anglei:#3*2/3)$}-|{$(\fixname.center)-(#3,0)$}) coordinate (\fixname_in\i);
        }}\fi
        \ifnum #2=0{} \else {\foreach \i in {1,...,#2} {
          \tikzmath{\anglei={90*(#2+1-2*\i)/#2};}
          \draw [thick]
            (\fixname.center) .. controls
            ($(\fixname.center)+(\anglei:#4/3)$) and ($(\fixname.center)+(\anglei:#4*2/3)$) ..
            ({$(\fixname.center)+(\anglei:#4*2/3)$}-|{$(\fixname.center)+(#4,0)$}) coordinate (\fixname_out\i);
        }}\fi
      }}
    }
  },
  spider/.code 2 args={
    \pgfkeysalso{special spider={#1}{#2}{\leglen}{\leglen}}
  }
}

\tikzset{
  inner WD/.style={
    every to/.style={out=0, in=180, draw, thick},
    unoriented WD,
    surround sep=0pt,
    pack size=10pt,
    pack psize=3pt,
    to_out/.style={shorten >= -1pt},
    fr_out/.style={shorten <= -1pt},
    pack size/.store in=\psize,
    pack psize/.store in=\ppsize,
    font=\tiny,
    anchor=center
  }
}

\tikzset{
  function/.style={->, thin, shorten <=4pt, shorten >=4pt}
}

\tikzset{
  tick/.style={
    postaction={
      decorate,
      decoration={
        markings, mark=at position 0.5 with {
          \draw[-] (0,.4ex) -- (0,-.4ex);
        }
      }
    }
  }
}

\tikzcdset{
  twocell/.style={Rightarrow, shorten <=5pt, shorten >=5pt}
}

\usetikzlibrary{arrows.meta,shapes,arrows}
\tikzset{shrink/.style={outer sep=0,inner sep=0, minimum size=0},
  squeeze/.style={outer sep=4pt,inner sep=0, minimum size=0},
  rc/.style={rounded corners=0.5cm},
  la/.style={scale=0.8},
  a/.style={->}, e/.style={->>},
  m/.style={right hook->},
  d/.style={double, double equal sign distance,-},
  u/.style={dashed,->},
  n/.style={double equal sign distance, -implies},
  ne/.style={double equal sign distance, -},
  t/.style={double distance=2.5pt, -implies, postaction={draw,-}},
  te/.style={double distance=2.5pt, -, postaction={draw,-}},
  cover/.style={preaction={draw=white, -,line width=6pt}},
  over/.style={auto=false,fill=white,inner sep=1pt}}

\definecolor{darkblue}{rgb}{0,0,0.7}
\setlist{noitemsep, nolistsep}

\theoremstyle{plain}
\newtheorem{theorem}{Theorem}[chapter]
\newtheorem{proposition}[theorem]{Proposition}
\newtheorem{corollary}[theorem]{Corollary}
\newtheorem{lemma}[theorem]{Lemma}

\theoremstyle{definition}
\newtheorem*{question}{Question}
\newtheorem{definition}[theorem]{Definition}

\newtheorem{notation}[theorem]{Notation}

\theoremstyle{remark}
\newtheorem{example}[theorem]{Example}

\newtheorem{remark}[theorem]{Remark}
\newtheorem{warning}[theorem]{Warning}

\theoremstyle{plain}
\newtheorem{altheorem}{Theorem}

\crefalias{chapter}{section}

\newcommand{\ord}[1]{\underline{#1}}
\newcommand{\const}[1]{\operatorname{\mathtt{#1}}}
\newcommand{\cat}[1]{\mathcal{#1}}
\newcommand{\ccat}[1]{\mathbb{#1}}
\newcommand{\Cat}[1]{{\mathsf{#1}}}
\newcommand{\CCat}[1]{\operatorname{\mathcal{\StrLeft{#1}{1}}\Cat{\StrGobbleLeft{#1}{1}}}}
\newcommand{\poCat}[1]{\operatorname{\mathbb{\StrLeft{#1}{1}}\Cat{\StrGobbleLeft{#1}{1}}}}
\newcommand{\Funr}[1]{\operatorname{\mathsf{#1}}}

\DeclareMathOperator{\model}{-Mod}
\DeclareMathOperator{\ob}{Ob}
\DeclareMathOperator{\dom}{dom}
\DeclareMathOperator{\comp}{comp}
\DeclareMathOperator{\cod}{cod}
\DeclareMathOperator{\id}{id}
\DeclareMathOperator{\er}{ER}
\DeclarePairedDelimiter{\pair}{\langle}{\rangle}
\DeclarePairedDelimiter{\copair}{[}{]}
\DeclarePairedDelimiter{\church}{\llbracket}{\rrbracket}

\newcommand{\tn}[1]{\textnormal{#1}}
\newcommand{\op}{^{\tn{op}}}
\newcommand{\tp}{^{\dagger}}
\newcommand{\co}{^{\tn{co}}}

\newcommand{\inv}{^{\text{-}1}}
\newcommand{\tpow}[1]{^{\otimes #1}}

\newcommand{\name}[1]{{#1}^{\sqsubset}}
\newcommand{\unname}[1]{{#1}^{\sqsupseteq}}
\newcommand{\tab}[1]{|{#1}|}

\newcommand{\rrgcat}{\CCat{RgCat}}
\newcommand{\fflcat}{\CCat{FLCat}}
\newcommand{\rrgcalc}{\CCat{RgCalc}}
\newcommand{\rrlpocat}{\CCat{RlPoCat}}
\newcommand{\pprlpocat}{\CCat{PrlPoCat}}
\newcommand{\rrel}{\poCat{Rel}}
\newcommand{\ccospan}{\poCat{Cospan}}
\newcommand{\sspan}{\poCat{Span}}
\newcommand{\pposet}{\poCat{Poset}}

\newcommand{\cont}[1]{\ccat{C}_{#1}}
\newcommand{\rc}[1]{(\cont{#1},{#1})}

\newcommand{\fsyn}{\poCat{Syn}}
\newcommand{\fprd}{\poCat{Prd}}
\newcommand{\frel}{\poCat{Rel}}
\newcommand{\fladj}{\Funr{LAdj}}
\newcommand{\fspanpo}{\Funr{Span^{po}}}

\newcommand{\finset}{\Cat{FinSet}}
\newcommand{\lax}{\Cat{Lax}}

\newcommand{\cc}{\ccat{C}}
\newcommand{\dd}{\ccat{D}}
\newcommand{\rr}{\ccat{R}}
\newcommand{\nn}{\mathbb{N}}
\newcommand{\pp}{\mathbb{P}}
\newcommand{\qq}{\mathbb{Q}}
\newcommand{\ww}{\mathbb{W}}

\newcommand{\true}{\const{true}}
\newcommand{\thecup}{\const{cup}}
\newcommand{\thecap}{\const{cap}}

\newcommand{\cp}{\mathbin{\fatsemi}}
\newcommand{\cocolon}{\ \,\llap{\normalfont:}}
\DeclareMathOperator{\iso}{\cong}
\DeclareMathOperator{\eqv}{\simeq}
\newcommand{\aeqv}{\ensuremath{\mathrel{\overset{ae}\eqv}}}
\newcommand{\To}[1]{\xrightarrow{#1}}
\newcommand{\Too}[1]{\To{\;\;#1\;\;}}
\newcommand{\from}{\leftarrow}
\newcommand{\From}[1]{\xleftarrow{#1}}
\newcommand{\surj}{\twoheadrightarrow}
\newcommand{\tto}{\Rightarrow}

\newcommand{\qqand}{\qquad\text{and}\qquad}
\newcommand{\qand}{\quad\text{and}\quad}

\newcommand{\adj}[5][30pt]{
  \begin{tikzpicture}[baseline=(1.base)]
    \node(1)[]{\ensuremath{#2}};
    \node(2)[right={#1} of 1]{\ensuremath{#5}};
    \draw[a]($(1.east)+(0,5pt)$)to node(left)[la,above]{\ensuremath{#3}}($(2.west)+(0,5pt)$);
    \draw[a]($(2.west)-(0,5pt)$)to node(right)[la,below]{\ensuremath{#4}}($(1.east)-(0,5pt)$);
    \path (1) -- node[midway,rotate=-90]{$\dashv$} (2);
  \end{tikzpicture}
}

\newcommand{\adjr}[5][30pt]{
  \begin{tikzpicture}[baseline=(1.base)]
    \node(1)[]{\ensuremath{#2}};
    \node(2)[right={#1} of 1]{\ensuremath{#5}};
    \draw[a]($(1.east)+(0,5pt)$)to node(left)[la,above]{\ensuremath{#3}}($(2.west)+(0,5pt)$);
    \draw[a]($(2.west)-(0,5pt)$)to node(right)[la,below]{\ensuremath{#4}}($(1.east)-(0,5pt)$);
    \path (1) -- node[midway,rotate=90]{$\dashv$} (2);
  \end{tikzpicture}
}

\newcommand{\sub}{\Cat{Sub}}
\newcommand{\out}{\mathrm{out}}
\newcommand{\define}[1]{\textbf{#1}}

\newlist{propenum}{enumerate}{1}
\setlist[propenum]{label=\roman*., ref=\theproposition~(\roman*)}

\newlist{lemenum}{enumerate}{1}
\setlist[lemenum]{label=\roman*., ref=\thelemma~(\roman*)}

\crefalias{propenumi}{proposition}
\crefalias{lemenumi}{lemma}

\newcommand{\otimesp}{\mathbin{\otimes'}}

\newcommand{\biadj}{\dashv_{\mathrm{bi}}}

\newlength{\arrowlength}
\newcommand{\threecellarr}{\mathrel{\begin{tikzpicture}[baseline=(A.base)]
      \node(A)[inner sep=0,outer sep=0,minimum size=0] at (0,0) {\vphantom{a}};
      \draw[t](A) -- ++(\arrowlength,0);
    \end{tikzpicture}}}

\def\cellslide{0.5}
\def\celllength{0.2cm}

\usepackage{xparse}
\NewDocumentCommand{\cell}{ O{} O{n} O{\cellslide} O{\celllength} m m m }{
  \coordinate (mid) at ($#5!{#3}!#6$);
  \coordinate (start) at ($(mid)!{#4}!#5$);
  \coordinate (end) at ($(mid)!{#4}!#6$);
  \draw[#2] (start) to node(secret)
  [inner sep=6pt,outer sep=0,minimum size=0,#1]{{#7}} (end);
}

\NewDocumentCommand{\celli}{ O{} O{n} O{\cellslide} O{\celllength} m m m }{
  \coordinate (mid) at ($#5!{#3}!#6$);
  \coordinate (start) at ($(mid)!{#4}!#5$);
  \coordinate (end) at ($(mid)!{#4}!#6$);
  \draw[#2] (start) to node(label)[inner sep=6pt,outer sep=0,minimum size=0,#1]{{#7}} (end);
  \coordinate (far) at ($(end)+(mid)-(label)$);
  \node[] at ($(end)!6pt!(far)$) {$\scriptscriptstyle\iso$} ;
}

\usepackage{xparse}
\usepackage{ifmtarg}
\newlength{\diagrampunctdist}
\newsavebox{\diagrampunct}
\makeatletter
\DeclareDocumentEnvironment{diagram*}{ O{} O{} }
{\xdef\basenode{\@ifmtarg{#2}{current bounding box.south}{#2.base}}\savebox{\diagrampunct}{#1}\begin{center}\begin{tikzpicture}}
    {\@ifnotmtarg{#1}{\node at ($(\basenode-|current bounding box.east)+(\diagrampunctdist,0)$) {\usebox{\diagrampunct}};\node at ($(\basenode-|current bounding box.west)-(\diagrampunctdist,0)$) {\hphantom{\usebox{\diagrampunct}}};}\end{tikzpicture}\end{center}}
\makeatother

\newcounter{diagram}
\makeatletter
\DeclareDocumentEnvironment{diagram}{ O{} O{} }
{\setcounter{diagram}{\value{equation}}\refstepcounter{equation}\refstepcounter{diagram}
  \xdef\basenode{\@ifmtarg{#2}{current bounding box.south}{#2.base}}\savebox{\diagrampunct}{#1}\begin{center}\hfill\begin{tikzpicture}}
    {\@ifnotmtarg{#1}{\node at ($(\basenode-|current bounding box.east)+(\diagrampunctdist,0)$) {\usebox{\diagrampunct}};\node at ($(\basenode-|current bounding box.west)-(\diagrampunctdist,0)$) {\hphantom{\usebox{\diagrampunct}}};}\end{tikzpicture}\hfill\llap{\normalfont(\thediagram)}\end{center}}
\makeatother

\crefname{diagram}{}{}
\Crefname{diagram}{Diagram}{Diagram}
\creflabelformat{diagram}{\textnormal{(#2#1#3)}}

\newcommand{\funcrinl}[1]{\begin{tikzpicture}[inner WD,baseline=(f.-25)]
    \node(f)[syntax,oshellr]{#1};
    \draw(f.west) -- +(-0.5,0);
    \draw(f.east) -- +(+0.5,0);
  \end{tikzpicture}
}

\newcommand{\ladjinl}[1]{\begin{tikzpicture}[inner WD,baseline=(f.-25)]
    \node(f)[syntax,funcr]{#1};
    \draw(f.west) -- +(-0.5,0);
    \draw(f.east) -- +(+0.5,0);
  \end{tikzpicture}
}

\newcommand{\radjinl}[1]{\begin{tikzpicture}[inner WD,baseline=(f.-25)]
    \node(f)[syntax,funcl]{#1};
    \draw(f.west) -- +(-0.5,0);
    \draw(f.east) -- +(+0.5,0);
  \end{tikzpicture}
}

\newcommand{\examplewiringdiagram}{  \begin{tikzpicture}[inner WD]
    \node[pack, syntax, oshell, shape border rotate=230] (f1) {$f_1$};
    \node[packs, below = 1.5 of f1.south] (p1) {};
    \node[packs, right= 2.5 of p1] (p2) {};
    \node[pack, syntax, oshell, shape border rotate=135, right=3 of f1] (f2) {$f_2$};
    \node[link, "$c_3$" left] at ($(p1)!.5!(p2)+(0,-2)$) (dot) {};
    \node[link, "$c_{6}$" right, below= 1.2 of p2] (c6) {};
    \node[outer pack, inner xsep=4pt, inner ysep=2pt, fit={($(f1)+(-2,2)$) (f2) (p1) (p2) (dot)}] (outer) {};
    \draw (p1) -- (dot);
    \draw (p2) -- (dot);
    \draw[to_out] (dot) -- (outer.270-|dot);
    \draw (p2) -- (c6);
    \draw (p1) to node[above] {$c_2$} (p2);
    \draw (p1) to node[left] {$c_1$}  (f1);
    \draw (p2) to node[right] {$c_5$}  (f1);
    \draw[to_out] (f1.140) to node[above,pos=0.01] {$c_5$}  (outer.140);
    \draw[to_out] (f2.80) to node[left] {$c_4$} (outer.50);
    \draw[to_out] (f2.0) to node[below] {$c_5$} (outer.20);
  \end{tikzpicture}}

\DeclareFontFamily{U}{rcjhbltx}{}
\DeclareFontShape{U}{rcjhbltx}{m}{n}{<->s*[1.2]rcjhbltx}{}
\DeclareSymbolFont{hebrewletters}{U}{rcjhbltx}{m}{n}
\let\aleph\relax\let\beth\relax
\DeclareMathSymbol{\aleph}{\mathord}{hebrewletters}{39}
\DeclareMathSymbol{\beth}{\mathord}{hebrewletters}{98}

\title{Regular Calculi I: Graphical Regular Logic}
\author{tslil clingman \and Brendan Fong \and David I.\ Spivak}
\date{\vspace{-.3in}}

\begin{document}

\maketitle

\begin{abstract}
  What is ergonomic syntax for relations? In this first paper in a series of two, to answer the question we define regular calculi: a suitably structured functor from a category representing the syntax of regular logic to the category of posets, that takes each object to the poset of relations on that type. We introduce two major classes of examples, regular calculi corresponding to regular theories, and regular calculi corresponding to regular categories. For working in regular calculi, we present a graphical framework which takes as primitive the various moves of regular logic. Our main theorem for regular calculi is a syntax-semantics $2$-dimensional adjunction to regular categories.
\end{abstract}

\newpage

\tableofcontents*

\newpage

\chapter{Introduction}\label{chap.intro}

Classically, a set relation $R$ on finitely many sets $\{A_{i}\}_{I}$ is a subset $R\subseteq\prod A_{i}$ of the product of those sets. A salient feature of set relations is that they may be compared: given $R, R'$ set relations on $\{A_{i}\}_{I}$, it is reasonable to ask whether $R\subseteq R'$. More still, we may note that this partial ordering underlies in particular a meet-semilattice structure on set relations on $\{A_{i}\}_{I}$; there is a maximal set relation $\prod A_{i}$ and given two set relations $R,R'$ we may compute their meet as $R\cap R'$.

All of the structure we have highlighted so far is ``local'', it is particular to a fixed family of sets $\{A_{i}\}_{I}$ and set relations thereupon. Should we allow ourselves to involve multiple such families, we see that set relations support further structure still. For example, given set relations $R \subseteq X \times Y$ and $S \subseteq Y \times Z$, their composite set relation $R\cp_{Y} S\subseteq X\times Z$ may be described by the formula \[ R \cp_{Y} S = \big\{(x,z) \mid \exists y [R(x,y)\wedge S(y,z)]\big\}\ . \] This form of composition interacts with all of the ``local'' structure we listed above, and is in fact well defined for arbitrarily finitely many set relations $R_{k}$ over arbitrarily finitely many shared indices $F\subseteq I$---although we would quickly exhaust the utility of our notation $R\cp_{Y} S$ above in an attempt to write this precisely. Further examination of such operations reveals that we wish to deal precisely with those operations and structures permitted by and expressible in \emph{regular logic}: that fragment of first order logic generated by equality ($=$), true ($\true$), conjunction ($\wedge$), and existential quantification ($\exists$).

Having thus noted key aspects of the calculus of set relations, and their connection to regular logic, our motivating question is as follows.

\begin{question}\hypertarget{question}
  What is an \emph{ergonomic} syntax for \emph{relations}?
\end{question}\newcommand{\quest}{\hyperlink{question}{question}}

Let us now give meaning to the emphasised terms in the above question, beginning with the notion of \emph{relation}. Although we began our exploration by teasing apart the structure of set relations, we wish to find a suitably general syntax for \emph{any} class of objects supporting the structure of regular logic. Certainly at first we may straightforwardly generalise the notion of a set relation $R\subseteq\prod A_{i}$ to any sufficiently structured category by asking for a monomorphism $R\rightarrowtail\prod A_{i}$. However, we mean something broader even than this. Our interest is in allowing the type of things upon which relations are defined and the type of relations to be entirely different. For instance, we hope to also encompass conjunctive queries from database theory; essentially, those database queries that can be expressed using regular logic \cite{chandra1977optimal}. Another, abundant and indeed motivating source of examples where the type of relations differs from the type of objects is the class \emph{regular theories}.

A regular theory comprises a collection $\Sigma\text{-sort}$ of sorts, a collection $\Sigma\text{-rel}$ of relation symbols, and a collection $\mathbb T$ of axioms. From the first collection we build the notion of \emph{context}, a list of sorts $\sigma_{i}\in\Sigma\text{-sort}$ from which we are permitted to draw abstract variables. We then allow ourselves to form \emph{formulae} in contexts by inductively applying the structures of regular logic ($=$, $\true$, $\wedge$, $\exists$) to abstract variables and relation symbols on those variables. Finally we impose an ordering of \emph{provability} according to the laws of deduction of regular logic, as well as our stated axioms. As formulae host all of the same structures and properties we extracted from set relations above, we see that we are naturally led to consider formulae in a regular theory as relations. Note however that contexts are the type of things upon which formulae are defined and so we must allow for a difference between our relations and the things upon which they are defined.

Having attended to the word \emph{relation}, let us now give concrete meaning to our question by elaborating \emph{ergonomic}. By this term mean firstly the following rigorous notions. A relation $R$ defined on some finite collection of objects $\{A_{i}\}_{I}$ does not naturally admit notions of domain and codomain. While it may be possible in practice to choose somehow a division of the $A_{i}$, nevertheless on our view an ergonomic syntax for relations should not impose arbitrary measures such as a privileging of certain objects as domain or codomain. In a similar vein, while we only displayed explicitly a binary composition $R\cp_{Y}S$ for relations above, we noted that arbitrarily finitely many relations may be composed, over arbitrarily finitely shared indices, simultaneously. An ergonomic syntax then should directly support such unbiased, multi-ary operadic compositions.

However, there are also some non-rigorous criteria we wish to ascribe to our notion of ergonomics. The syntax, whatever its form, should intuitively encode the various deductions of regular logic. For instance, it should be evident in the syntax that we can eliminate any conjunct $\exists_{x'}[x=x']$ whenever $x$ is already present in the context. Furthermore, given the operadic nature of composition, our notion of ergonomic includes also the requirement of a graphical syntax.

Now that our \quest{} has been understood, we can turn our attention to the state of the art. The classical categorical syntax for relations is that of regular categories. These are categories which abstract enough of the structure of the category of sets to house regular logic, and thus to deal with relations. Unfortunately, by design regular categories do not capture the notion of relation we outlined above: a relation $R$ on objects $\{A_{i}\}_{I}$ in a regular category $\cat R$ is in particular an object $R\in\ob\cat R$ and so is necessarily of the same sort as the objects upon which it is defined. In order to remedy this, work was done on ``functionally complete `bicategories of relations'{}''\cite{carboni1987cartesian}\footnote{for another, ultimately equivalent approach in this general theme see \cite{freyd1990categories}}, or more recently and equivalently on \emph{relational po-categories} \cite{fong2019regular}---structures which form the progenitor of this paper. In common to both of these approaches is a privileging of relations to live between objects, instead of among them, thereby freeing the type of relations to differ from the type of objects. Moreover, these approaches benefit from sharing the same ``category theory'' as regular categories: \cite[Theorem 7.3]{fong2019regular} proves that $2$-category $\rrlpocat$ of relational po-categories is equivalent to the $2$-category $\rrgcat$ of regular categories. While this viewpoint does admit a pleasing graphical syntax for relations, nevertheless it suffers from the technical requirement of privileging a sense of ``binarity'' for relations: as morphisms, relations require a choice of domain and codomain, and composition is presented as a binary operation.

To organise more complicated multi-way composites of relations, many fields have developed some notion of wiring diagram. A good amount of recent work, including but not limited to control theory \cite{bonchi2014categorical,baez2015categories,fong2016categorical}, database theory and knowledge representation \cite{bonchi2018graphical,patterson2017knowledge}, electrical engineering \cite{baez2018compositional}, and chemistry \cite{baez2017compositional}, all serve to demonstrate the link between these languages and an ergonomic syntax for relations.

In a similar vein, we desire to use an ergonomic syntax to describe those relational structures arising from regular theories. Consider in particular the regular theory of a pre-order. This regular theory has a single sort $S$, and a single relation symbol $P$. Formally its axioms, reflexivity and transitivity, are implications between certain formulae in the theory. However, an ergonomic syntax would instead allow us to describe these axioms as follows.
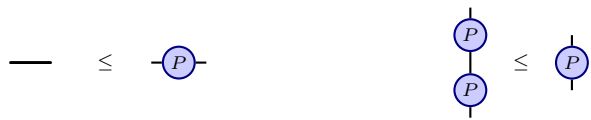
\begin{diagram*}
  \node(A)[]{
    \begin{tikzpicture}[inner WD]
      \node(1)[minimum size=0, inner sep=0]{};
      \draw(1) -- ++(-2,0) -- ++(2,0);
      \node(2)[pack] at ($(1)+(6,0)$) {$P$};
      \draw(2.west) -- ++(-0.5, 0);
      \draw(2.east) -- ++(+0.5, 0);
      \path(1) -- node(a)[midway]{$\leq$} ($(2)-(1,0)$);
    \end{tikzpicture}
  };
  \node(B)[right=3cm of A.east,anchor=west]{
    \begin{tikzpicture}[inner WD]
      \node(1)[pack]{$P$};
      \node(2)[below= of 1,pack]{$P$};
      \coordinate(l) at ($(1)!0.5!(2)$){};
      \draw(1) -- (2);
      \draw(1.north) -- ++(0, +0.5);
      \draw(2.south) -- ++(0, -0.5);

      \node(3)[right=4 of l,pack]{$P$};
      \draw(3.north) -- ++(0,+0.5);
      \draw(3.south) -- ++(0,-0.5);
      \path(l-|2.east) -- node[midway]{$\leq$} (3);
    \end{tikzpicture}
  };
\end{diagram*}

The above-left diagram is our desired syntactic representation of reflexivity: whenever we have a wire---which represents an abstract variable of the unique sort $S$---we may deduce the presence of a $P$ relation between that wire and itself. Similarly, in the above-right diagram we see our desired syntactic representation of $\exists_{y}[xPy\wedge yPz]\vdash xPz$, that is, transitivity: given two abstract variables and an existentially quantified variable shared between two copies of $P$, we may deduce a $P$ relation directly between the two abstract variables.

How then do we go about answering our \quest{} and rigorously develop an ergonomic syntax for relations which affords us the above pictures? This paper and its companion \cite{grl2} are intended to serve as just such a means. Our first step, and the main thrust of this paper, is to introduce the notion of \emph{regular calculi}.

A regular calculus $\rc P$ comprises first the data of a symmetric monoidal \emph{po-category} $\cont P$---a category whose homs are posets and whose compositional and monoidal structures are monotonic; see \cref{def.pocats_funs_laxnts}. This po-category additionally has a structured notion of \emph{wiring diagram}, a purely combinatorial way to describe how objects are connected to one-another in $\cont P$ governed by cospans of sets and for which we develop a graphical syntax in \cref{chap.wires}. A regular calculus then additionally comprises the data of a \emph{right ajoint lax} po-functor $P$ from $\cont P$ to posets (\cref{def.ajax}). These pieces we assemble as follows. We think of the objects of $\Gamma\in\ob\cont P$ as contexts for relations in some relational theory, each poset $P(\Gamma)$ as the poset of relations in the context $\Gamma$ ordered by implication, the right adjoint lax structure on $P$ gives meets of relations $R\wedge R'\in P(\Gamma)$ and $\true_{\Gamma}\in P(\Gamma)$, and each wiring diagram $\Gamma\to\Gamma'$ gives a method for converting relations in the context $\Gamma$ to relations in the context $\Gamma'$ by using equality ($=$), true ($\true$), conjunction ($\wedge$), and existential quantification ($\exists$).

The structured notion of wiring diagram in the symmetric monoidal po-category of contexts $\cont P$ of a regular calculus $\rc P$ automatically extends to a graphical language for describing relations which arise out of regular logic operations on others. For instance,
in a fixed context $\Gamma$ from which we draw variables, and from relations $R_{1}$, $R_{2}$, and $R_{3}$ of arity 3, 3, and 4 respectively, we might wish to construct a composite relation $S$ graphically by specifying how the various $R_{i}$ are connected and share variables. By using the graphical notation for regular calculi we develop here, we will be able to rigorously draw and interpret the following \emph{graphical term} of the regular calculus $\rc P$
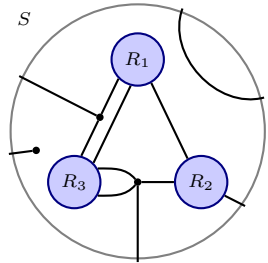
\begin{diagram*}[][][penetration=0, inner WD,  pack size=9pt, link size=2pt, scale=2, baseline=(out)]
  \node[packs] at (-1.5,-1) (f) {$R_{3}$};
  \node[packs] at (0,1.9) (g) {$R_{1}$};
  \node[packs] at (1.5,-1) (h) {$R_{2}$};
  \node[outer pack, inner sep=34pt] at (0,.2) (out) {};
  \node[above left=0.25 and 0.25 of out]{$S$};

  \node[link] at ($(f)!.5!(h)$) (link1) {};
  \node[link] at (-2.4,-.25) (link2) {};
  \node[link] at ($(f.75)!.5!(g.-135)$) (link3) {};

  \draw[fr_out] (out.270) to (link1);
  \draw[fr_out] (out.190)	to (link2);
  \draw[fr_out] (out.155) to (link3);
  \draw[fr_out] (out.-35)	to (h.-30);
  \draw[to_out,fr_out] (out.15)	to[out=-165,in=-110] (out.70);
  \draw (f.30) to[out=0,in=130] (link1);
  \draw (f.-30) to[out=0,in=-130] (link1);
  \draw (h.180) to (link1);
  \draw (g.-60) to (h.120);
  \draw (f.45) to (g.-105);
  \draw (f.75) to (link3);
  \draw (g.-135) to (link3);
\end{diagram*}
and recognise that it represents the following relation
\[
  S(y,z,z',x,x',z'') = \exists\, \tilde{x},\tilde{y}, \left[R_1(\tilde{x},\tilde{y},y) \wedge R_2(x',\tilde{x},x) \wedge R_3(y,\tilde{y},x',x') \wedge (z=z')\right]\ .
\] The structure governing these graphical terms, and the notion of wiring diagram, is here formalised as a \emph{supply for the po-prop} $\ww$. This object $\ww$ is morally the po-category $\ccospan\co$ of cospans of finite sets, and affords us intuitive encodings of the various theorems of regular logic. Indeed, that our wiring diagrams admit definition purely combinatorially in terms of four generators and some relations is a central aspect of our development; see \cref{sec.depict_ww}. For example, the relation $\exists_{x'}[x=x']$ is rendered as a graphical term as below left, and one of the equations in $\ww$ dictates precisely the equality below, therefore enabling us to eliminate any occurrences as indicated below right.
\begin{diagram*}
  \node(A){
    \begin{tikzpicture}[WD]
      \node[link] (mu) {};
      \draw (mu) -- +(-1,0);
      \draw (mu) to[out=60, in=180] +(.5,.5) node[link] {};
      \draw (mu) to[out=-60, in=180] +(1,-.5) -- +(.5,0);
    \end{tikzpicture}
  };
  \node(B)[right=of A.east,anchor=west]{
    \begin{tikzpicture}[WD]
      \draw (0,0) -- (2,0);
    \end{tikzpicture}
  };
  \path (A) -- node[midway]{$=$} (B);
\end{diagram*}

Let us now return to our motivating \quest{}. Note that regular calculi do not impose any constraints on the type of relations, and moreover do not enforce a notion of ``binarity'': relations have no chosen domain or codomain, and composition---as seen in the diagram above---is fully multi-ary and operadic. Thus, in making the notion of graphical term rigorous through the course of this work, we will see that regular calculi do present an ergonomic syntax for relations.

To claim that regular calculi are truly a syntax for relations, however, we must additionally establish that they may stand in wherever we would have otherwise used a different syntax to capture regular logic and relations.

The first such case we wish to draw attention to is the that of models for regular theories. Given a regular theory $(\Sigma,\mathbb T)$, one can describe what it means to model the theory in a regular category $\cat R$ and this gives a functor $\mathbb T\model(-)\colon \rrgcat\to\CCat{Cat}$. A major triumph of the classical work on categorical regular logic is to provide a representation\footnote{More correctly, a bi-representation for the $2$-functorial extension} for this functor: there is a \emph{syntactic regular category} $\cat C^{\mathrm{reg}}_{\mathbb T}$ associated to $(\Sigma,\mathbb T)$ and an appropriately natural family of equivalences $\rrgcat(\cat C^{\mathrm{reg}}_{\mathbb T},\cat R)\eqv\mathbb T\operatorname{-Mod}(\cat R)$---see, for instance, \cite{butz1998regular} for exposition in this vein.

We are able to give a complete analogy for regular calculi. Recall that we wish to understand regular calculi as housing a regular theory of relations. To this end, we expect that we should be able to extract from each regular calculus $\rc P$ a \emph{syntactic relational po-category} $\fsyn\rc P$ in analogy with the classical process $(\Sigma,\mathbb T)\mapsto\cat C^{\mathrm{reg}}_{\mathbb T}$ for regular theories to regular categories. Indeed, we provide such a construction which proceeds, in particular, by defining the objects of $\fsyn\rc P$ to be pairs $(\Gamma,\theta)$ of contexts $\Gamma\in\ob\cont P$ and relations $\theta$ in context $\Gamma$. In this way we see that $\fsyn$ models the syntax of the regular calculus closely. An immediate benefit of working categorically is that we are able to extend this construction to a $2$-functor $\fsyn\colon\rrgcalc\to\rrlpocat$ from regular calculi to relational po-categories.

To complete the analogy then we need two additional ingredients. Thinking once more of regular calculi as housing regular theories of relations, there is an appropriate notion of \emph{model for a regular calculi} and in \cref{def.models} we describe a $2$-functor $\rc P\model(\cdot)\colon\rrlpocat\to\CCat{Cat}$ sending a relational po-category $\rr$ to the $2$-category of models of $\rc P$ in $\rr$. Then, given a regular theory $(\Sigma,\mathbb T)$, in \cref{con.rgcalc_rgtheory} we prove the existence of an associated regular calculus $(\cont{\mathbb T},F_{\mathbb T})$.

With the ingredients prepared we are able to give the analogous theory of models. First, using tools developed in the companion, we prove that our notion of syntactic po-category suitably generalises the classical notion. This result appears as \cref{thm.syn_equiv} below.

\begin{altheorem}
  Given a regular theory $(\Sigma,\mathbb T)$, the syntactic po-category $\fsyn(\cc_{\mathbb T},F_{\mathbb T})$ of the regular calculus associated to the theory is equivalent to the po-category of relations $\frel(\cat C^{\mathrm{reg}}_{\mathbb T})$ of the regular category associated to the theory. That is, under the equivalence between relational po-categories and regular categories, the two constructions agree.
\end{altheorem}

Using this we prove additionally that our notion of model of a regular calculus suitably generalises the classical notion. This result appears as \cref{thm.model_equiv} below.

\begin{altheorem}
  There is an equivalence of categories $(\cc_{\mathbb T},F_{\mathbb T})\model(\frel\cat R)\eqv \mathbb T\model(\cat R)$ appropriately natural in regular categories $\cat R$, where $\frel\rr$ is the po-category of relations of $\cat R$.
\end{altheorem}

In fact, the above theorem is proven as a corollary of another result which shows that we may suitably substitute regular calculi as a syntax for relations in another fashion. Let us now turn our attention to the interaction between regular calculi and relational po-categories.

We have already demonstrated that regular calculi provide a home for regular theories, and so we would be remiss to omit addressing the question of regular categories---equivalently the relational po-categories. We show that from a relational po-category $\rr$ one may construct a regular calculus $\fprd\rr$ through a process we call \emph{taking predicates}. Given a regular category $\cat{R}$, viewed as its po-category of relations $\rrel\cat{R}$, our construction yields the regular calculus $\fprd\rrel\cat{R}$ whose contexts are the objects of $\cat{R}$ and whose relations $\theta$ in context $r$ are precisely the sub-objects $\theta\rightarrowtail r$, exactly as we might have hoped. This assignment of relational po-categories $\rr\mapsto\fprd\rr$ to regular calculi we extend to a $2$-functor $\fprd\colon\rrlpocat\to\rrgcalc$.

Recall however that we have the opposed $2$-functor $\fsyn\colon\rrgcalc\to\rrlpocat$ of the syntactic po-category construction, and so with $\fprd$ above we prove our first comparison theorem---appearing as part of \cref{thm.main} in this paper and whose proof is elaborated in the companion.

\begin{altheorem}
  The $2$-functors $\fsyn\colon\rrgcalc\to\rrlpocat$ and $\fprd\colon\rrlpocat\to\rrgcalc$ form a bi-adjunction $\fsyn\biadj\fprd$.
\end{altheorem}

By a \emph{bi-adjunction} here we mean the appropriate notion of $2$-dimensional adjunction where the equations on the unit and co-unit now hold only up to invertible $3$-dimensional morphisms, each of which satisfy some appropriate equation. As adjunctions transfer a wealth of category-theoretic aspects, so too do bi-adjunctions---this bi-adjunction affords us a rich comparison of the $2$-category theory of regular calculi and relational po-categories. However, from the point of view of providing an ergonomic syntax for relational po-categories, or equivalently regular categories, it is as yet unsatisfactory.

Without a stronger theorem we cannot be sure that by working syntactically in the regular theory carried by the regular calculus $\fprd\rr$ we are in fact working in the relational po-category $\rr$ itself. That is, we wish to know: is there an equivalence between the syntax $\fsyn\fprd\rr$ given by our regular calculus approach and the relational po-category $\rr$? To this end we prove as part of our main theorem the following general answer to this question.

\begin{altheorem}
  The co-unit of the bi-adjunction $\fsyn\biadj\fprd$ is an adjoint equivalence, so relational po-categories are pseudo-reflective in regular calculi.
\end{altheorem}

By \emph{pseudo-reflective} here we mean the appropriate $2$-dimensional version of the analogous ordinary category theoretic notion of a fully-faithful inclusion of sub-categories which admits a left adjoint. Among other things, this result tells us that we may freely embed relational po-categories and their morphisms into regular calculi by taking predicates, and that all graphical manipulations and syntactical operations hold in the original object: $\fsyn\fprd\rr\eqv\rr$.

As \cite{fong2019regular} proves that regular categories and relational po-categories have equivalent $2$-categories, we have in fact also obtained the following theorem.

\begin{altheorem}
  The $2$-category of regular categories is pseudo-reflective in regular calculi.
\end{altheorem}

Finally, the presence of this syntax-semantics pseudo-reflection allows us to deduce further interesting consequences: in \cref{cor.ladj_birep} we prove that taking the regular category of left adjoints in a relational po-category is suitably represented as a $2$-functor, and in \cref{thm.reg_lex} we use it to give an alternate proof of an adjunction between finite limit categories and regular categories.

\section{Outline}

We have striven, where reasonable, to render this paper as self-contained as possible. Where we make use of results from the body of work of \cite{fong2019hypergraph,fong2019regular,fong2019supplying} we are careful to cite them or reprove them in our context. With that in mind, this paper is organised as follows.

\cref{chap.background} presents the setting of symmetric monoidal po-categories and morphisms thereof in which we will be working. \cref{chap.wires} introduces the po-prop for wiring $\ww$, the notion of supply for a po-prop, and develops our graphical notation for $\ww$ as well as for symmetric monoidal po-categories which supply it. Then in \cref{chap.regular_calculi} we define the central structures of this paper, the regular calculi and their morphisms, by way of the notions of right adjoint monoid and right ajax po-functor. Additionally, among our examples in this section, we show that regular theories give rise to regular calculi. In \cref{chap.relational_pocats} we recall the axiomatisation of relational po-categories, and construct and study the $2$-functor $\fprd$ which takes a relational po-category to its regular calculus of predicates. In \cref{chap.graphical_reglog} we develop our graphical formalism for regular calculi by defining graphical terms and establishing key lemmas which afford us intuitive means of graphical reasoning and manipulations, with examples drawn from regular theories and relational po-categories. Additionally in that section we sketch the construction of the syntactic po-category of a regular calculus. Finally in \cref{chap.preview} we preview some of the results and details that appear in the companion, including the statement of our main theorem \cref{thm.main} and its various corollaries.

\section{Related work}
Graphical formalisms for fragments of first order logic are an old subject, dating back to the existential graphs of Pierce in the late 19th Century \cite[Volume 4, Book II]{peirce1974collected}. More recently, Brady and Trimble \cite{brady2000categorical,brady2000string} and Haydon and Sobocinski \cite{haydon2020compositional} have sought to reinterpret Pierce's work within a categorical setting, providing a foundation for Pierce's diagrams through the string diagrams for monoidal categories defined by Joyal and Street \cite{joyal1991geometry}.

In the context of regular logic, Bonchi, Seeber, and Sobocinski's work on graphical conjunctive queries \cite{bonchi2018graphical}, as well as the last two authors' precursor version of this work \cite{fong2020string}, give alternate categorical, diagrammatic approaches to regular logic. Like regular calculi, these categorical approaches to regular logic and relational structures feature Frobenius monoids as a central way of capturing equality of variables or, graphically, the lines that can be drawn between relation symbols. A difference in our emphasis is that our diagrams are not simply string diagrams in a monoidal category; they are distinct, combinatorial objects inspired by an operadic notion of composition.

Brady and Trimble investigate Pierce's diagrams via Lawvere's notion of hyperdoctrine \cite{lawvere1969adjointness}, a foundational tool in categorical logic. A hyperdoctrine studies a logical system as a fibration, presenting it as a collection of posets indexed by a suitable base category. The required structure of this base category, such as having finite products, allows interpretation of certain syntactic aspects, while interpretations of connectives and quantifiers are studied through adjunctions. In particular, regular logic is modelled through the notion of an elementary existential doctrine \cite{lawvere1970equality}.

An elementary existential doctrine is a functor from the opposite of a category with finite products to the category of meet-semilattices, such that certain assignments of the functor have adjoints, and such that these adjoints obey Beck-Chavelley and Frobenius reciprocity conditions. Similar to elementary existential doctrines, regular calculi index posets. The difference lies in the choice of base category: while the base category of a hyperdoctrine is modelled on aspects of the category of finite sets and functions---that is, the free finite product category---the base category of a regular calculus is modelled on the po-category of finite sets and co-relations (certain equivalence classes of cospans)---or, more precisely, the free relational po-category.

Our regular calculus approach then expresses the logical structure through adjunctions in the base category and the right ajax condition. This leads to a more compact definition than that of an elementary existential doctrine, without the Beck-Chavelley and Frobenius reciprocity axioms. Indeed, we believe the regular calculus definition provides some motivation for how these axioms arise.

Many of the aspects of our work with regular calculi here have precursors in work done with elementary existential doctrines. The syntactic category, also known as a Lindenbaum--Tarski category of a logical theory or hyperdoctrine is a well-known construction (see for example \cite{butz1998regular}); our syntactic po-category is a relational analogue. Our completion theorem lies within the tradition of the completion theorems of Maietti and Rosolini \cite{maietti2013quotient}, with Trotta's recent paper on existential completions of primary doctrines perhaps the closest analogue \cite{trotta2020existential}.

Finally, the adjunction between hyperdoctrine-style logical presentations and categorical structure within the context of regular logic is also explored by Bonchi, Santamaria, Seeber, and Sobocinski \cite{bonchi2021doctrines}, who detail an adjunction between elementary existential doctrines and cartesian bicategories.

\section{Acknowledgements}

The second- and third-named authors would like thank Paolo Perrone for comments that have improved this article and Christina Vasilakopoulou for finding an error in a previous version, which led us to this fully 2-categorical formulation. The first-named author would like to thank Emily Riehl for conversations which informed the present structure of this paper and the companion. We acknowledge support from AFOSR grants FA9550-17-1-0058 and FA9550-19-1-0113.

\chapter{Background on symmetric monoidal po-categories}\label{chap.background}

To develop the theory of regular calculi and to state and prove our main results we will make extensive use of the language of symmetric monoidal po-categories and various higher morphisms thereof. In \cref{sec.two_cats} below we recall briefly the needed notions of \emph{oplax-natural transformation}, \emph{modification}, and \emph{adjunction in a $2$-category}. Readers familiar with these notions are invited to omit this section. Following this, in \cref{sec.po_categories} we will observe the several significant specialisations of these notions to the \emph{po-categorical} setting and cement terminology therein. Finally, in \cref{sec.symm_mom_po_cat} we will recall the notion of \emph{symmetric monoidal po-category} and various morphisms thereof.

Before we proceed with this background, let us pause a moment to record the salient features of our notation in this paper. While we endeavour to be standard in most aspects, it may nevertheless be useful to note the following.

\begin{itemize}
\item By a poset we mean a partially ordered set, that is, $x\leq y\wedge y\leq x\tto x=y$.

\item We typically denote composition in diagrammatic order, so the composite of $f\colon A\to B$ and $g\colon B\to C$ is $f\cp g\colon A\to C$. We often denote the identity morphism $\id_c\colon c\to c$ on an object $c\in\cat{C}$ simply by the name of the object, $c$. Thus if $f\colon c\to d$, we have $(c\cp f)=f=(f\cp d)$.

\item We may denote the unique map from an object $c$ to a terminal $1$ as $!\colon c \to 1$, and we denote the top element of any poset $P$ by $\true\in P$.

\item We denote the universal map into a product by $\pair{f,g}$ and the universal map out of a coproduct by $\copair{f,g}$.

\item Given a natural number $n\in\nn$, we write $\ord n$ for the set $\{1,2,\ldots,n\}\in\finset$; in particular $\ord{0}=\varnothing$.

\item We will write $c\tpow n$ in a monoidal category to denote the left-associated $n$-fold iterated binary tensor product $(\cdots((c\otimes c)\otimes c)\cdots)\otimes c$.

\item Given a lax monoidal functor $F\colon\cat{C}\to\cat{D}$, we denote the \emph{laxators} by $\varphi\colon I\to F(I)$ and $\varphi_{c,c'}\colon F(c)\otimesp F(c')\to F(c\otimes c')$ for objects $c,c'\in\ob\cat{C}$. If $F$ is strong, then we will make use of the same notation, but refer to these maps as \emph{strongators} instead.

\item Where our arguments make use of more than one dimension, we will write the morphisms with Latin letters, the $2$-morphisms with Greek letters, and the $3$-morphisms with Hebrew letters. For instance, $2$-functors will be denoted by $F$, $G$, \ldots, oplax-natural transformations will be denoted by $\alpha$, $\beta$, \ldots, and modifications will be denoted by $\aleph$, $\beth$, \ldots
\end{itemize}

\section{Background on \texorpdfstring{$2$}{2}-categories}\label{sec.two_cats} We will take for granted the notion of $2$-category and $2$-functor, but briefly recall here the definitions of higher morphisms between these. The reader already comfortable with such notions is nevertheless encouraged to review the various specialisations obtained in the \emph{po-categorical} setting in \cref{sec.po_categories}, and the later background on \emph{symmetric monoidal po-categories} in \cref{sec.symm_mom_po_cat}.

\begin{definition}[Oplax-natural transformations \& modifications]\label{def.strong_twont}
  Given a pair of parallel $2$-functors $F,G\colon\cat{K}\to\cat{L}$, an \define{oplax-natural transformation} $\alpha\colon F\tto G$ comprises the data of object components $\alpha_c\colon Fc\to Gc$ for each object $c\in\ob\cat{K}$, and morphism components $\alpha_h\colon (\alpha_c\cp Gh)\tto(Fh\cp \alpha_{c'})$ for each morphism $h\colon c\to c'$ of $\cat{K}$. These morphism components are required to be natural with respect to $2$-morphisms of $\cat{K}$, and are required to be compatible with identity morphisms and composition in $\cat{K}$. For details see, for example, \cite[Definition 4.2.1]{JohYau}.

  An oplax-natural transformation $\alpha$ is \define{pseudo-natural} when each morphism component $\alpha_{h}$ is a $2$-isomorphism, and is \define{$2$-natural} when each morphism component $\alpha_{h}$ is an identity.

  A \define{modification} $\aleph\colon \alpha\threecellarr\beta$ between oplax-natural transformations $\alpha,\beta\colon F\tto G$ comprises the data of object components $\aleph_c\colon\alpha_c\tto\beta_c$ in $\cat{L}(Fc,Gc)$ for each object $c\in\ob\cat{K}$, which are required to be compatible with the morphism components of $\alpha$ and $\beta$. For details see, for example, \cite[Definition 4.4.1]{JohYau}.
\end{definition}

Recall that, given a $2$-category $\cat{K}$, an \define{adjunction in $\cat{K}$} consists of a pair of objects $c,d\in\ob\cat{K}$, a pair of morphisms $l\colon c\to d$ and $r\colon d\to c$, and a pair of 2-morphisms $\eta\colon d\tto (l\cp r)$ and $\epsilon\colon (l\cp r)\tto c$ such that the following pair of diagrams, the \define{triangle equalities}, are rendered commutative.
\begin{diagram}[][][baseline=(b.base)]\label{eqn.adjunction}
  \node(1)[]{$l$};
  \node(2)[right= of 1]{$lrl$};
  \node(3)[below= of 2]{$l$};
  \draw[n](1)to node[la,above]{$\eta\cp l$}(2);
  \draw[n](2)to node[la,right]{$l\cp\epsilon$}(3);
  \draw[d](1)to(3);
  \node(4)[right= 3cm of 2]{$r$};
  \node(5)[below= of 4]{$rlr$};
  \node(6)[right= of 5]{$r$};
  \draw[n](4)to node(b)[la,left]{$r\cp\eta$}(5);
  \draw[n](5)to node[la,below]{$\epsilon\cp r$}(6);
  \draw[d](4)to(6);
\end{diagram}
One may verify that adjunctions compose, and so by $\fladj\cat{K}$ we denote the $1$-category with the same objects as $\cat{K}$ and whose morphisms are the data of left adjoints $(l,r,\eta,\epsilon)$ in $\cat{K}$.

For given data $(l,r,\eta,\epsilon)$ as above, the property of being an adjunction is expressed equationally in the compositions of the ambient $2$-category. As such, we obtain the following lemma.

\begin{lemma}
  Let $F\colon\cat{K}\to\cat{L}$ be a $2$-functor. The assignment $(l,r,\eta,\epsilon)\mapsto (Fl,Fr,F\eta,F\epsilon)$ sends adjunctions in $\cat{K}$ to adjunctions in $\cat{L}$, and so gives rise to a functor between the categories of left adjoints, $\fladj F\colon\fladj\cat{K}\to\fladj\cat{L}$. Moreover, this assignment $F\mapsto\fladj F$ of $2$-functors itself functorial in $2$-functors and so extends to a functor $\fladj\colon2\!\CCat{Cat}\to\CCat{Cat}$.\hfill$\qed$
\end{lemma}

In fact more is true, $\fladj$ is a $2$-functor when the $2$-morphisms in $2\!\CCat{Cat}$ are themselves required to be left adjoints, but we will not need this fact in this generality.

\section{Po-categories}\label{sec.po_categories}
The theory of $2$-categories specialises significantly to the context of \emph{po-categories}, and so we recall briefly the appropriate definitions now.

\begin{definition}[Po-category]\label{def.pocats_funs_laxnts}
  A \define{po-category $\cc$} is a locally-posetal $2$-category, that is, it is an ordinary $2$-category $\cc$ for which the category $\cc(c,c')$ of morphisms between any two objects is a partially ordered set.

  A \define{po-functor} $F\colon\cc\to\dd$ between po-categories is an ordinary $2$-functor, but we may summarise this structure by requiring that $F$ is an ordinary functor of the underlying $1$-categories and that the functions $F_{c,c'}\colon\cc(c,c')\to\dd(Fc,Fc')$ are monotonic for all objects $c,c'\in\cc$.

  An \define{oplax-natural transformation} $\alpha\colon F \tto G$ between po-functors $F,G\colon\cc\to\dd$ is an ordinary oplax-natural transformation between the $2$-functors $F$ and $G$. However, all of the compatibility conditions are degenerate and so the data is merely a collection of morphisms $\alpha_c\colon Fc \to Gc$ which satisfy $F(f) \cp \alpha_{c'}\leq \alpha_c \cp G(f)$ for all morphisms $f\colon c \to c'$ of $\cc$. In particular, a \define{$2$-natural transformation} of po-functors is merely a natural transformation of the underlying functors.

  Modifications are especially degenerate. Given parallel oplax-natural transformations $\alpha,\beta\colon F\tto G$, we write $\alpha\leq\beta$ if for each $c\in\cc$ there is an inequality $\alpha_c\leq\beta_c$ in $\dd(Fc,Gc)$ between $c$-components. Thus we are motivated in writing $[\cc,\dd]$ for the po-category of po-functors, oplax-natural transformations, and modifications; we call it the \define{po-category of po-functors from $\cc$ to $\dd$}.
\end{definition}

\begin{notation}[Po-categories]
  To distinguish po-categories from $1$-categories we will write po-categories with double-struck letters, as in $\cc,\dd,\ldots$, and reserve script for $1$-categories, as in $\cat{C},\cat{D},\ldots$
\end{notation}

Note that in any $2$-category, any two right adjoints to a given morphism are isomorphic, so in a po-category, a given morphism has \emph{at most one} right adjoint.

\begin{definition}[Left adjoint oplax-natural transformation]\label{def.ladj_lax_nt}
  Let $\cc$ and $\dd$ be po-categories. A \define{left adjoint oplax-natural transformation} is a left adjoint in the po-category $[\cc,\dd]$ of \cref{def.pocats_funs_laxnts}.
\end{definition}

As a consequence of the posetal nature of the po-category of po-functors $[\cc,\dd]$ we may freely pass the structure of left adjointness between oplax-natural transformations and their components, in the following sense.

\begin{lemma}\label{lemma.adj_in_pocat}
  Let $F,G\colon\cc\to\dd$ be po-functors, and let $\lambda\colon F\tto G$ and $\rho\colon F\tto G$ be opposed oplax-natural transformations. Then $\lambda$ is left adjoint to $\rho$ if and only if for each $c\in\cc$ the component $\lambda_c\colon Fc\to Gc$ are left adjoint to $\rho_c\colon Gc\to Fc$ in $\dd$.
\end{lemma}

\begin{proof}
  The forward direction is true even for $2$-categories that aren't locally posetal; the backwards direction holds since the uniqueness of $2$-morphisms in a po-category implies that the triangle equalities \cref{eqn.adjunction} hold trivially.
\end{proof}

Observe that an invertible modification whose eventual codomain is a po-category is necessarily an equality. As such, between po-categories the notions of $2$-dimensional and $1$-dimensional equivalence coincide.

\begin{definition}\label{def.po_equiv}
  An \define{equivalence of po-categories} is an equivalence in the $1$-category of po-categories and po-functors. We say that a po-functor $F\colon\cc\to\dd$ is \define{fully-faithful} when the morphism $F_{c,c'}\colon\cc(c,c')\to\dd(Fc,Fc')$ is an isomorphism of posets for each $c,c'\in\ob\cc$. Furthermore, we say that \define{$F$ is essentially surjective} if one can construct a function sending $d\in\ob\dd$ to a pair $(c\in\ob\cc, Fc\iso d)$, and we will tend to leave the function implicit.
\end{definition}

We will take for granted the following extension of the classical result relating fully-faithful essentially surjective functors and equivalences.

\begin{lemma}
  Given a po-functor $F\colon\cc\to\dd$, the data of an equivalence on $F$ is equivalently the data of essential surjectivity of $F$ and the property of fully-faithfulness for $F$.\hfill$\qed$
\end{lemma}

\section{Symmetric monoidal po-categories}\label{sec.symm_mom_po_cat}

We will have a great deal of use for symmetric monoidal po-categories. These objects may be viewed as 3-categories which are ``petite'' in two dimensions: they are locally po-categorical, and they have only one object. While it is notoriously difficult to give a full definition of monoidal bi-category, all of the coherence issues degenerate for po-categories. As such it is conceptually simpler to think of symmetric monoidal po-categories instead as symmetric monoidal $1$-categories with extra structure: hom-sets are equipped with an order, and the monoidal operation is monotonic on morphisms.

\begin{definition}[Symmetric monoidal po-category]\label{def.symm_mon_po}
  A \define{symmetric monoidal structure} on a po-category $\cc$ consists of a symmetric monoidal structure $(\otimes, I, \lambda, \rho)$ on its underlying $1$-category, such that $\otimes$ is a additionally a po-functor. That is, $(f_1\otimes g_1)\leq(f_2\otimes g_2)$ whenever $f_1\leq f_2$ and $g_1\leq g_2$. Recall that this means $\lambda$ and $\rho$ are automatically $2$-natural (\cref{def.pocats_funs_laxnts}).

  A \define{strong symmetric monoidal po-functor}  $(F,\varphi)\colon(\cc,\otimes, I)\to(\dd,\otimesp,I')$ is a po-functor $F\colon\cc\to\dd$ whose underlying functor is strong symmetric monoidal. Recall that this means that the \define{strongators} $\varphi_{c,c'}\colon Fc\otimesp Fc'\to F(c\otimes c')$ are automatically $2$-natural.

  We will frequently wish to apply the qualifier ``monoidal'' to various forms of natural transformations; by a \define{monoidal `\emph{adjective}' natural transformation} we will always mean an `\emph{adjective}' natural transformation whose components additionally obey the monoidal natural transformation conditions strictly. For example, a \define{monoidal left adjoint oplax-natural transformation} $\alpha\colon(F,\varphi)\tto(G,\psi)$ is a left adjoint oplax-natural transformation $\alpha\colon F\tto G$ whose components $\alpha_{c}\colon Fc\to Gc$ additionally obey the monoidal natural transformation conditions strictly---for instance $\varphi_{I}\cp\alpha_{I}=\psi_{I}$.
\end{definition}

\begin{notation}[Symmetry isomorphisms]
  If $(\cc,\otimes,I)$ is a symmetric monoidal po-category, $m,n\in\nn$ are natural numbers, and $c\colon \ord{m}\times \ord{n}\to\cc$ is a family of objects in $\cc$, then there is a canonical natural isomorphism
  \begin{equation}\label{eqn.symmetry}
    \sigma\colon
    \bigotimes_{i\in\ord{m}}\bigotimes_{j\in\ord{n}}c(i,j)\Too{\cong}
    \bigotimes_{j\in\ord{n}}\bigotimes_{i\in\ord{m}}c(i,j)\ .
  \end{equation}
  We will refer to these $\sigma$ as the \define{symmetry} isomorphisms, though note that generally such isomorphisms involve associators and unitors too.
\end{notation}

\chapter{Supplying wires}\label{chap.wires}

Our work toward understanding graphical regular logic begins with the establishment of the supporting machinery which was developed in \cite{fong2019regular,fong2019supplying} and whose salient details we recall here. In order to render our graphical terms, we will need already the more primitive notion of \emph{wiring diagram}. The somehow prototypical case of these is the generic structure supporting a basic graphical calculus, viz., the \emph{po-prop} for wiring $\ww$. Once we have gained some proficiency in this context, we will see how the notion of \emph{supply} for a po-prop allows us to understand wiring diagrams in any po-category which supplies $\ww$.

\section{The po-prop \texorpdfstring{$\ww$}{W} for wiring}\label{sec.prop_ww}

\begin{definition}[Po-prop]\label{def.props}
  A \define{po-prop} is a symmetric strict monoidal po-category $\pp$ whose monoid of objects is isomorphic to $(\nn,0,+)$. A \define{po-prop functor} $F\colon\pp\to\qq$ is a bijective-on-objects symmetric strict monoidal po-functor.
\end{definition}

The first, and indeed most important example we intend to consider is the po-prop for wiring. Consider the symmetric monoidal $2$-category $(\ccospan\co,\varnothing,+)$, i.e.\ the $2$-dual of cospans between finite sets.

\begin{definition}[$\ww$]\label{def.wwiring}
  The \define{po-prop for wiring}, $\ww$, is the local poset reflection of the full and locally full sub-$2$-category of $(\ccospan\co,\varnothing,+)$ spanned by the finite ordinals $\ord{n}$.
\end{definition}

Even though we have presented $\ww$ in a way that emphasises ``binarity''---its morphisms have domains and codomains and its composition is that of an ordinary category---nevertheless as we shall see in the next section, and in particular in \cref{ex.wiring_diagram}, the combinatorial nature of $\ww$ affords us a completely unbiased way to depict morphisms.

In \cref{prop.ww_explicit} we shall give a more explicit description of the hom posets of $\ww$, after which we will exhibit the generating morphisms and relations for $\ww$ graphically. Note that $\ww$ is almost the prop of equivalence relations, also known as corelations \cite{coya2017corelations}, but without the ``extra'' law, which would equate the following cospans.
\begin{equation}
  \label{eqn.extra_law}
  \begin{gathered}
    \ord0\to \ord0\from \ord0\\
    \text{``truth''}
  \end{gathered}\qqand
  \begin{gathered}
    \ord0\to \ord1\from \ord0\\
    \text{``inhabitedness''}
  \end{gathered}
\end{equation}
We'll see later that equating these two morphisms bears the consequence of forcing every sort in a regular theory to be inhabited.

\begin{proposition}\label{prop.ww_explicit}
  The hom-posets of $\ww$ admit the following explicit description:
  \[
    \ww(\ord m,\ord n)\cong
    \begin{cases}
      \{0\leq 1\}\op&\tn{ if }m=n=0\\
      \er\op(m+n)&\tn{ if }m+n\geq 1
    \end{cases}
  \]
  where $\{0\leq 1\}$ is the poset of booleans, and $\er(p)$ is the poset of equivalence relations on the set $\ord{p}$, ordered by inclusion.
\end{proposition}
\begin{proof}
  For any $m,n$, may identify $\ww(\ord m,\ord n)\op$ with the poset reflection of $\ccospan(\ord{m},\ord{n})$. But $\ccospan(\ord{m},\ord{n})$ is the coslice category $\ord{m+n}/\finset$, consisting of finite sets $S$ equipped with functions $\ord{m+n}\to S$. If $m+n=0$, then the poset reflection is that of $\finset$, namely $\{0\leq1\}$; otherwise it may be identified with the poset of equivalence relations on $\ord{m+n}$. Indeed, every function $\ord{m+n}\to S$ factors as an epi followed by a mono, and every mono out of a nonempty set has a retraction.
\end{proof}

The proof just given suggests two further results on the structure of $\ww$ which we record here.

\begin{corollary}\label{cor.jointly_surj_w}
  Every morphism $\omega\colon\ord n\to\ord m$ in $\ww$ with $m+n>0$ admits a unique representation as a jointly surjective cospan of finite sets $\ord n\to\ord n_{w}\from \ord m$.\hfill$\qed$
\end{corollary}

\begin{remark}\label{rem.canonical_decomp}
  At this point we have succeeded in showing that there are in fact \emph{canonical choices} of representative cospans for every morphism of $\ww$. When the domain or codomain are inhabited, then the above corollary uniquely determines a jointly epimorphic cospan, and in the remaining case we choose $\ord 0\to \ord 0\from\ord 0$ and $\ord 0\to \ord 1\from\ord 0$ as our representatives for the two distinct elements of $\ww(\ord 0,\ord0)$.
\end{remark}

\begin{corollary}\label{cor.mor_tensors_ww}
  In $\ww$, for an arbitrary morphism $\omega\colon\ord n\to\ord m$, we have:
  \begin{enumerate}
  \item $\omega+\id_{\ord 0}=\omega$
  \item $\omega+\eta\cp\epsilon=\omega$ if and only if $\omega\ne\id_{\ord0}$.\hfill$\qed$
  \end{enumerate}
\end{corollary}

\begin{remark}
  Later on we shall see that tensor of morphisms in $\ww$ bears interpretation as a form of logical conjunction (\cref{con.rgcalc_rgtheory}), and so we may understand the above corollary, in the language of \eqref{eqn.extra_law}: logical conjunction by ``truth'' is the identity, while logical conjunction by ``inhabitedness'' is the identity when we have any context involving the implicit sort. In addition, the logical conjunction of ``inhabitedness'' with ``inhabitedness'' is once more ``inhabitedness''.
\end{remark}

Now that we have explicated the structure of the hom posets of $\ww$, let us turn our attention to its morphisms, and for this purpose introduce a graphical notation prototypical of those to come. $\ww$ may be generated by four morphisms, and we list these generating morphisms, their canonical cospan representatives in $\ccospan\co$ (see \cref{rem.canonical_decomp}), and their graphical icons in the table below.

\setlength{\belowrulesep}{0pt} \setlength{\tabcolsep}{8pt} \renewcommand{\arraystretch}{1.5}
\begin{equation}\label{eqn.generating_wires}
  \begin{tabular}{|c|c|c|}\toprule[1pt]
    Morphism in $\ww$ & Corresponding cospan & Icon \\\toprule[1pt]
    $\epsilon\colon \ord 1\to \ord0$&
                                      $\ord1\to \ord1\from \ord0$&
                                                                   \begin{tikzpicture}[WD]
                                                                     \node[link] (epsilon) {};
                                                                     \draw (epsilon) to +(-.8,0);
                                                                   \end{tikzpicture}\\\hline
    $\delta\colon \ord1\to \ord2$&
                                   $\ord1\to \ord1\from \ord2$&
                                                                \begin{tikzpicture}[WD]
                                                                  \node[link] (delta) {};
                                                                  \draw (delta) -- +(-1,0);
                                                                  \draw (delta) to[out=60, in=180] +(1,.5);
                                                                  \draw (delta) to[out=-60, in=180] +(1,-.5);
                                                                \end{tikzpicture}\\\hline
    $\eta\colon \ord0\to \ord1$&
                                 $\ord0\to \ord1\from \ord1$&
                                                              \begin{tikzpicture}[WD]
                                                                \node[link] (eta) {};
                                                                \draw (eta) to +(.8,0);
                                                              \end{tikzpicture}\\\hline
    $\mu\colon \ord2\to \ord1$&
                                $\ord2\to \ord1\from \ord1$&
                                                             \begin{tikzpicture}[WD]
                                                               \node[link] (mu) {};
                                                               \draw (mu) -- +(1,0);
                                                               \draw (mu) to[out=120, in=0] +(-1,.5);
                                                               \draw (mu) to[out=-120, in=0] +(-1,-.5);
                                                             \end{tikzpicture}\\\toprule[1pt]
  \end{tabular}
\end{equation}
\setlength{\tabcolsep}{3pt}

\noindent As the underlying 1-category of $\ww$ is a symmetric monoidal category, we may use these icons as primitive generators in the usual string diagram language in the sense of Joyal and Street \cite{joyal1991geometry}. These generators satisfy the following equations and inequalities involving additionally the symmetry $\ord 2\to\ord 2$ of $\ww$, and we render these constraints graphically with composition indicated via horizontal juxtaposition and tensor indicated via vertical juxtaposition.
\begin{equation}\label{eqn.equations_wires}
  \renewcommand{\arraystretch}{2}
  \begin{array}{rl<{\qquad}rl<{\qquad}rl}
    \begin{tikzpicture}[WD]
      \node[link] (mu) {};
      \draw (mu) -- +(-1,0);
      \draw (mu) to[out=60, in=180] +(.5,.5) coordinate (out1);
      \draw (mu) to[out=-60, in=180] +(.5,-.5) coordinate (out2);
      \coordinate (end) at ($(out1)+(1,0)$);
      \draw (out1) to[out=0, in=180] (out2-|end);
      \draw (out2) to[out=0, in=180] (out1-|end);
    \end{tikzpicture}
    &\;\raisebox{2pt}{=}\quad
      \begin{tikzpicture}[WD]
        \node[link] (mu) {};
        \draw (mu) -- +(-1,0);
        \draw (mu) to[out=60, in=180] +(1,.5);
        \draw (mu) to[out=-60, in=180] +(1,-.5);
      \end{tikzpicture}
    &
      \begin{tikzpicture}[WD]
        \node[link] (mu) {};
        \draw (mu) -- +(-1,0);
        \draw (mu) to[out=60, in=180] +(.5,.5) node[link] {};
        \draw (mu) to[out=-60, in=180] +(1,-.5) -- +(.5,0);
      \end{tikzpicture}
    &\;\raisebox{2pt}{=}\quad
      \begin{tikzpicture}[WD,baseline=-5pt]
        \draw (0,0) -- (2,0);
      \end{tikzpicture}
    &
      \begin{tikzpicture}[WD]
        \node[link] (mu) {};
        \draw (mu) -- +(-1,0);
        \draw (mu) to[out=60, in=180] +(.5,.5) node[link] (mu2) {};
        \draw (mu2) to[out=60, in=180] +(1,.5) coordinate (end);
        \draw (mu2) to[out=-60, in=180] +(1,-.5);
        \draw (mu) to[out=-60, in=180] +(1,-.5) coordinate (h);
        \draw (h) -- (h-|end);
      \end{tikzpicture}
    &\;\raisebox{2pt}{=}\quad
      \begin{tikzpicture}[WD]
        \node[link] (mu) {};
        \draw (mu) -- +(-1,0);
        \draw (mu) to[out=-60, in=180] +(.5,-.5) node[link] (mu2) {};
        \draw (mu2) to[out=60, in=180] +(1,.5) coordinate (end);
        \draw (mu2) to[out=-60, in=180] +(1,-.5);
        \draw (mu) to[out=60, in=180] +(1,.5) coordinate (h);
        \draw (h) -- (h-|end);
      \end{tikzpicture}
    \\
    \begin{tikzpicture}[WD]
      \node[link] (mu) {};
      \draw (mu) -- +(1,0);
      \draw (mu) to[out=120, in=0] +(-.5,.5) coordinate (out1);
      \draw (mu) to[out=-120, in=0] +(-.5,-.5) coordinate (out2);
      \coordinate (end) at ($(out1)+(-1,0)$);
      \draw (out1) to[out=180, in=0] (out2-|end);
      \draw (out2) to[out=180, in=0] (out1-|end);
    \end{tikzpicture}
    &\;\raisebox{2pt}{=}\quad
      \begin{tikzpicture}[WD]
        \node[link] (mu) {};
        \draw (mu) -- +(1,0);
        \draw (mu) to[out=120, in=0] +(-1,.5);
        \draw (mu) to[out=-120, in=0] +(-1,-.5);
      \end{tikzpicture}
    &
      \begin{tikzpicture}[WD]
        \node[link] (mu) {};
        \draw (mu) -- +(1,0);
        \draw (mu) to[out=120, in=0] +(-.5,.5) node[link, anchor=center] {};
        \draw (mu) to[out=-120, in=0] +(-1,-.5) -- +(-.5,0);
      \end{tikzpicture}
    &\;\raisebox{2pt}{=}\quad
      \begin{tikzpicture}[WD,baseline=-5pt]
        \draw (0,0) -- (2,0);
      \end{tikzpicture}
    &
      \begin{tikzpicture}[WD]
        \node[link] (mu) {};
        \draw (mu) -- +(1,0);
        \draw (mu) to[out=120, in=0] +(-.5,.5) node[link, anchor=center] (mu2) {};
        \draw (mu2) to[out=120, in=0] +(-1,.5) coordinate (end);
        \draw (mu2) to[out=-120, in=0] +(-1,-.5);
        \draw (mu) to[out=-120, in=0] +(-1,-.5) coordinate (h);
        \draw (h) -- (h-|end);
      \end{tikzpicture}
    &\;\raisebox{2pt}{=}\quad
      \begin{tikzpicture}[WD]
        \node[link] (mu) {};
        \draw (mu) -- +(1,0);
        \draw (mu) to[out=-120, in=0] +(-.5,-.5) node[link, anchor=center] (mu2) {};
        \draw (mu2) to[out=120, in=0] +(-1,.5) coordinate (end);
        \draw (mu2) to[out=-120, in=0] +(-1,-.5);
        \draw (mu) to[out=120, in=0] +(-1,.5) coordinate (h);
        \draw (h) -- (h-|end);
      \end{tikzpicture}
    \\
    \begin{tikzpicture}[WD]
      \node[link] (delta) {};
      \draw (delta) -- +(-.75,0);
      \draw (delta) to[out=60, in=180] +(.5,.5) coordinate (out1);
      \draw (delta) to[out=-60, in=180] +(.5,-.5) coordinate (out2);
      \node[link, right=1 of delta] (mu) {};
      \draw (mu) -- +(.75,0);
      \draw (mu) to[out=120, in=0] +(-.5,.5) -- (out1);
      \draw (mu) to[out=-120, in=0] +(-.5,-.5) -- (out2);
    \end{tikzpicture}
    &\;\raisebox{2pt}{=}\quad
      \begin{tikzpicture}[WD,baseline=-5pt]
        \draw (0,0) -- (1.5,0);
      \end{tikzpicture}
    &
      \multicolumn{4}{c}{$
      \begin{tikzpicture}[WD]
        \coordinate (htop);
        \coordinate[below=.7 of htop] (hmid);
        \coordinate[below=.7 of hmid] (hbot);
        \node[link, left=.5] (dotL) at ($(htop)!.5!(hmid)$) {};
        \node[link, right=.5] (dotR) at ($(hbot)!.5!(hmid)$) {};
        \draw (dotL) -- +(-1,0) coordinate (l);
        \draw (dotR) -- +(1,0) coordinate (r);
        \draw (dotL) to[out=60, in=180] (htop);
        \draw (dotL) to[out=-60, in=180] (hmid);
        \draw (hmid) to[out=0, in=120] (dotR);
        \draw (hbot) to[out=0, in=-120] (dotR);
        \draw (htop) -- (htop-|r);
        \draw (hbot) -- (hbot-|l);
      \end{tikzpicture}
      \;\raisebox{2pt}{=}\quad
      \begin{tikzpicture}[WD]
        \node[link] (delta) {};
        \draw (delta) -- +(-.5,0) coordinate (end);
        \draw (delta) to[out=60, in=180] +(1,.5);
        \draw (delta) to[out=-60, in=180] +(1,-.5);
        \node[link, left=1 of delta] (mu) {};
        \draw (mu) -- (end);
        \draw (mu) to[out=120, in=0] +(-1,.5);
        \draw (mu) to[out=-120, in=0] +(-1,-.5);
      \end{tikzpicture}
      \;\raisebox{2pt}{=}\quad
      \begin{tikzpicture}[WD]
        \coordinate (htop);
        \coordinate[below=.7 of htop] (hmid);
        \coordinate[below=.7 of hmid] (hbot);
        \node[link, left=.5] (dotL) at ($(hbot)!.5!(hmid)$) {};
        \node[link, right=.5] (dotR) at ($(htop)!.5!(hmid)$) {};
        \draw (dotL) -- +(-1,0) coordinate (l);
        \draw (dotR) -- +(1,0) coordinate (r);
        \draw (dotL) to[out=-60, in=180] (hbot);
        \draw (dotL) to[out=60, in=180] (hmid);
        \draw (hmid) to[out=0, in=-120] (dotR);
        \draw (htop) to[out=0, in=120] (dotR);
        \draw (htop) -- (htop-|l);
        \draw (hbot) -- (hbot-|r);
      \end{tikzpicture}
      $}
    \\
    \begin{tikzpicture}[WD,baseline=-5pt]
      \draw (0,0) -- (1.5,0);
    \end{tikzpicture}
    &\;\raisebox{2pt}{$\leq$}\quad
      \begin{tikzpicture}[WD,baseline=-5pt]
        \node[link] (epsilon) {};
        \draw (epsilon) to +(-.8,0);
        \node[link, right=.5 of epsilon] (eta) {};
        \draw (eta) to +(.8,0);
      \end{tikzpicture}
    &
      \begin{tikzpicture}[WD,baseline=-5pt]
        \node[link] (epsilon) {};
        \draw (epsilon) to +(-.8,0);
        \node[link, left=1 of epsilon] (eta) {};
        \draw (eta) to +(.8,0);
      \end{tikzpicture}
    &\;\raisebox{2pt}{$\leq$}\quad
      \id_{\ord0}
    &
      \begin{tikzpicture}[WD]
        \node[link] (delta) {};
        \draw (delta) -- +(-.5,0) coordinate (end);
        \draw (delta) to[out=60, in=180] +(1,.5);
        \draw (delta) to[out=-60, in=180] +(1,-.5);
        \node[link, left=1 of delta] (mu) {};
        \draw (mu) -- (end);
        \draw (mu) to[out=120, in=0] +(-1,.5);
        \draw (mu) to[out=-120, in=0] +(-1,-.5);
      \end{tikzpicture}
    &\;\raisebox{2pt}{$\leq$}\quad
      \begin{tikzpicture}[WD]
        \draw (0, 0) -- (1.5, 0);
        \draw (0, .75) -- (1.5, .75);
      \end{tikzpicture}
  \end{array}
\end{equation}

We refer to the composites $\eta\cp\delta$ and $\mu\cp\epsilon$ as the \define{cup} and the \define{cap}; they are denoted $\thecup$ and $\thecap$ and are depicted as follows:
\begin{equation}\label{eqn.cap_cup}
  \begin{array}{cccc}
    \begin{tikzpicture}[inner WD]
      \draw (.5, .5) -- (0,.5) to[out=180, in=180] (0, -.5) -- (.5, -.5);
    \end{tikzpicture}
    &\;\raisebox{2pt}{$\coloneqq$}\quad
      \begin{tikzpicture}[WD]
        \node[link] (mu) {};
        \node[link, left=.3 of mu] (eta) {};
        \draw (mu) -- (eta);
        \draw (mu) to[out=60, in=180] +(.5,.5) -- +(.25,0);
        \draw (mu) to[out=-60, in=180] +(.5,-.5) -- +(.25,0);
      \end{tikzpicture}
    &\qqand
      \begin{tikzpicture}[inner WD]
        \draw (-.5, .5) -- (0,.5) to[out=0, in=0] (0, -.5) -- (-.5, -.5);
      \end{tikzpicture}
    &\;\raisebox{2pt}{$\coloneqq$}\quad
      \begin{tikzpicture}[WD]
        \node[link] (mu) {};
        \node[link, right=.3 of mu] (eta) {};
        \draw (mu) -- (eta);
        \draw (mu) to[out=120, in=0] +(-.5,.5) -- +(-.25,0);
        \draw (mu) to[out=-120, in=0] +(-.5,-.5) -- +(-.25,0);
      \end{tikzpicture}
  \end{array}
\end{equation}
It follows from \eqref{eqn.equations_wires} that $\thecap$ and $\thecup$ satisfy the ``yanking'' or adjunction identities
\begin{equation}\label{eqn.yanking}
  \begin{tikzpicture}[inner WD]
    \draw (1, 1) --
    (-.25, 1) to[out=180, in=180]
    (-.25, 0) --
    (.25, 0) to[out=0, in=0]
    (.25, -1) --
    (-1, -1);
    \node[font=\normalsize] at (3,0) {=};
    \draw (5,0) -- (7,0);
    \node[font=\normalsize] at (9,0) {=};
    \draw (13, -1) --
    (11.75, -1) to[out=180, in=180]
    (11.75, 0) --
    (12.25, 0) to[out=0, in=0]
    (12.25, 1) --
    (11, 1);
  \end{tikzpicture}
\end{equation}

\renewcommand{\arraystretch}{1}
The equations in the first and second lines of \eqref{eqn.equations_wires} are known as the \define{(co)commutativity, (co)unitality, and (co)associativity} equations for comonoids and monoids, respectively. The equations in the next line are known as the \define{special law} and the \define{frobenius law}. We refer to the inequalities in the last line as the \define{adjunction inequalities}, because they show up as the unit and co-unit of adjunctions, as we see next in the following proposition.

\begin{proposition}\label{prop.left_adj_in_ww}
  With notation as in \eqref{eqn.generating_wires}, there are adjunctions
  \begin{equation}\label{eqn.ww_adjunctions}
    \adj{\ord1}{\epsilon}{\eta}{\ord0}
    \qqand
    \adj{\ord1}{\delta}{\mu}{\ord2}.
  \end{equation}
\end{proposition}

\begin{proof}
  The inequalities $\id_{\ord1}\leq(\epsilon\cp\eta)$,\quad $(\eta\cp\epsilon)\leq\id_{\ord1}$,\quad $\id_{\ord2}\leq(\mu\cp\delta)$, and the equation $\id_{\ord1}=(\delta\cp\mu)$ are all shown in \eqref{eqn.equations_wires}, which itself is proved via computations in $\ccospan\co$. The required equalities \eqref{eqn.adjunction} are automatic in a po-category.
\end{proof}

\begin{remark}\label{rem.surprising_inequality}
  The perhaps surprising half of the ``special law'', i.e.\ the inequality $(\delta\cp\mu)\leq\id_{\ord1}$ not arising from adjointness, is in fact derivable from the rest of the structure:
  \[
    \begin{tikzpicture}
      \node (P1) {
        \begin{tikzpicture}[WD]
          \node[link] (delta) {};
          \draw (delta) -- +(-.75,0);
          \draw (delta) to[out=60, in=180] +(.5,.5) coordinate (out1);
          \draw (delta) to[out=-60, in=180] +(.5,-.5) coordinate (out2);
          \node[link, right=1 of delta] (mu) {};
          \draw (mu) -- +(.75,0);
          \draw (mu) to[out=120, in=0] +(-.5,.5) -- (out1);
          \draw (mu) to[out=-120, in=0] +(-.5,-.5) -- (out2);
        \end{tikzpicture}
      };
      \node (P2) [right=1 of P1] {
        \begin{tikzpicture}[WD]
          \node[link] (delta) {};
          \draw (delta) -- +(-.75,0);
          \draw (delta) to[out=60, in=180] +(.5,.5) coordinate (out1);
          \draw (delta) to[out=-60, in=180] +(.5,-.5) coordinate (out2);
          \node[link, right=1 of delta] (mu) {};
          \draw (mu) to[out=120, in=0] +(-.5,.5) -- (out1);
          \draw (mu) to[out=-120, in=0] +(-.5,-.5) -- (out2);
          \draw (mu) -- +(.75,0) node[link] (delta2) {};
          \draw (delta2) to[out=60, in=180] +(.5,.5) -- +(.5,0);
          \draw (delta2) to[out=-60, in=180] +(.5,-.5) node[link] {};
        \end{tikzpicture}
      };
      \node (P3) [right=1 of P2] {
        \begin{tikzpicture}[WD]
          \node[link] (delta) {};
          \draw (delta) -- +(-.75,0);
          \draw (delta) to[out=60, in=180] +(.5,.5) -- +(.5,0);
          \draw (delta) to[out=-60, in=180] +(.5,-.5) node[link] {};
        \end{tikzpicture}
      };
      \node (P4) [right=1 of P3] {
        \begin{tikzpicture}[WD]
          \draw (0,0) -- (1.5,0);
        \end{tikzpicture}
      };
      \node at ($(P1.east)!.5!(P2.west)$) {$=$};
      \node at ($(P2.east)!.5!(P3.west)$) {$\leq$};
      \node at ($(P3.east)!.5!(P4.west)$) {$=$};
    \end{tikzpicture}
  \]
  More generally, if $f\colon \ord{m}\surj \ord{n}$ is any surjective function then $f\tp\cp f=\id$, where $f\tp$ is the transpose (left adjoint) of $f$.
\end{remark}

\begin{definition}[Sub-props of $\ww$]\label{def.comon_mon_selfduals}
  The prop $\ww$ contains several other important full sub-props:
  \begin{itemize}
  \item That generated by $\epsilon$ and $\delta$ is called the \define{prop for cocommutative comonoids}.
  \item That generated by $\eta$ and $\mu$ is called the \define{prop for commutative monoids}.
  \item That generated by $\thecup$ and $\thecap$ \eqref{eqn.cap_cup} is called the \define{prop for self-duals}.
  \end{itemize}
  In fact these three props are equivalent to the monoidal category of finite sets, its opposite, and the category of unoriented cobordisms, respectively. See also \cite[Section 3]{fong2019supplying}.
\end{definition}

\begin{notation}[General arity $\delta$ and $\mu$]
  We will adopt the convention that for $n\in\nn$ we will write $\delta^{n}\colon\ord 1\to\ord n$ and $\mu^{n}\colon \ord n\to\ord 1$ for the maps associated to the cospans $\ord 1\to\ord1\from\ord n$ and $\ord n\to\ord1\from\ord 1$ respectively. In particular we have $\delta^{0}=\epsilon$, $\mu^{0}=\eta$, and $\delta^{1}=\mu^{1}=\id_{\ord 1}$ among other identities. Graphically we might render these morphisms as
  \begin{diagram*}
    \node(D){\begin{tikzpicture}[WD]
        \node[link] (delta) {};
        \draw (delta) -- +(-1,0);
        \draw (delta) to[out=60, in=180] +(1,1.3);
        \draw (delta) to[out=-60, in=180] +(1,-1.3);
        \draw (delta) to[out=60, in=180] +(1,0.8) node(t){};
        \draw (delta) to[out=-60, in=180] +(1,-0.8) node(b){};
        \path (t) -- node[right]{$m$} (b);
        \path (t) -- node[sloped]{$\ldots$} (b);
      \end{tikzpicture}};
    \node(M)[right=2cm of D]{\begin{tikzpicture}[WD]
        \node[link] (mu) {};
        \draw (mu) -- +(+1,0);
        \draw (mu) to[out=120, in=0] +(-1,1.3);
        \draw (mu) to[out=-120, in=0] +(-1,-1.3);
        \draw (mu) to[out=120, in=0] +(-1,0.8) node(t2)[anchor=center]{};
        \draw (mu) to[out=-120, in=0] +(-1,-0.8) node(b2)[anchor=center]{};
        \path (t2) -- node[left]{$n$} (b2);
        \path (t2) -- node[sloped,anchor=center]{$\ldots$} (b2);
      \end{tikzpicture}};
    \path (D) -- node[midway]{and} (M);
  \end{diagram*}
  for $\delta^{m}$ and $\mu^{n}$ respectively.
\end{notation}

\section{Picturing morphisms in \texorpdfstring{$\ww$}{W}}\label{sec.depict_ww}

The po-category $\ww$ forms the foundation of our diagrammatic language for regular logic. We have already seen that morphisms in $\ww$ can be given a graphical description by depicting generating morphisms using special icons, and working in the usual Joyal--Street string diagram language for morphisms in symmetric monoidal categories.

We now present an alternate way to depict morphisms in $\ww$. While the description we provide here is presently informal, later we shall describe a rigorous manner in which to draw, interpret, and manipulate these pictures.

\begin{notation}[Objects as IO shells]
  By definition, an object $\ord n\in\ww$ is a finite set and we represent it graphically by a circle with $n$ ports around the exterior.
  \begin{equation}\label{eqn.shell_pic}
    \ord{n}=
    \begin{tikzpicture}[inner WD, baseline=(rho)]
      \node[pack,minimum size = 3ex] (rho) {};
      \draw (rho.180) to  +(180:2pt);
      \draw (rho.110) to +(110:2pt);
      \node at ($(rho.0)+(60:4pt)$) {$\vdots$};
      \node at ($(rho.0)+(0:7pt)$) {\big\}};
      \node at ($(rho.0)+(0:29pt)$) {$n-3$ ports};
      \draw (rho.-110) to +(-110:2pt);
    \end{tikzpicture}
  \end{equation}
  Our convention will be for the ports to be numbered clockwise from the left of the circle, unless otherwise indicated. We refer to such an annotated circle as an \define{input-output shell}, or \define{IO shell}.
\end{notation}

\begin{notation}[Picturing morphisms of $\ww$] \label{notation.picturing_ww}
  Given a morphism $\omega\colon \ord n_1+ \dots + \ord n_s \to \ord n_\out$ in $\ww$, we associate to it the canonical representation as a cospan of finite sets assured by \cref{rem.canonical_decomp}.
  \begin{equation}\label{eqn.cospan}
    \ord{n}_1+\cdots+\ord{n}_s
    \To{\copair{\omega_1,\ldots,\omega_s}}
    \ord{n}_\omega
    \From{\omega_\out}
    \ord{n}_\out
  \end{equation}
  From this we draw our depiction of $\omega$ as follows.
  \begin{enumerate}[nolistsep, noitemsep]
  \item Draw the IO shell for $\ord n_\out$.
  \item Draw each object $\ord n_i$, for $i=1,\dots,s$, as non-overlapping IO shells inside the $\ord n_\out$ IO shell.
  \item For each $i \in \ord{n}_\omega$, draw a black dot anywhere in the region interior to the $\ord n_\out$ IO shell but exterior to all the $\ord n_i$ IO shells.
  \item For each element $(i,j) \in \sum_{i=1,\dots,s,\out} \ord{n}_i$, draw a wire connecting the $j$th port on the object $\ord n_i$ to the black dot $\omega_i(j)$.
  \end{enumerate}

  For a more compact notation, we may also neglect to explicitly draw the object $\ord n_\out$, leaving it implicit as comprising the wires left dangling on the boundary of the diagram.
\end{notation}

\begin{remark}
  While we take the convention of \cref{notation.picturing_ww} for drawing morphisms of $\ww$, note that once we determine the domain and codomain of the morphism, the rest of the data is combinatorial. This implies that, up to a choice of domain and codomain, the morphism represented by any picture is invariant under topological deformation.
\end{remark}

\begin{example} \label{ex.wiring_diagram}
  Here is the combinatorial data of a morphism $\omega\colon \ord n_1+ \ord n_2 + \ord n_3 \to \ord n_\out$ in $\ww$, together with its depiction:
  \begin{align*}
    \parbox{3.5in}{\raggedright
    $n_1 = 3$,\quad $n_2 = 3$,\quad $n_3 = 4$,\quad $n_\out = 6$, \quad $n_\omega = 7$\\
    $\omega_1(1)=4$, $\omega_1(2)=2$, $\omega_1(3)=1$,\\
    $\omega_2(1)=6$, $\omega_2(2)=4$, $\omega_2(3)=5$,\\
    $\omega_3(1)=1$, $\omega_3(2)=2$, $\omega_3(3)=\omega_3(4)=6$,\\
    $\omega_\out(1)=1$, $\omega_\out(2)=\omega_\out(3)=3$\\
    $\omega_\out(4)=5$, $\omega_\out(5)=6$, $\omega_\out(6)=7$.}
    &\qquad
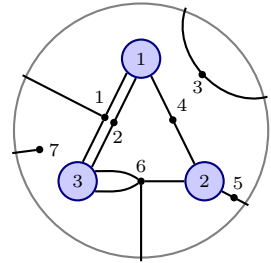
      \begin{tikzpicture}[penetration=0, inner WD,  pack size=9pt, link size=2pt, font=\tiny, scale=2, baseline=(out)]
        \node[packs] at (-1.5,-1) (f) {$3$};
        \node[packs] at (0,1.9) (g) {$1$};
        \node[packs] at (1.5,-1) (h) {$2$};
        \node[outer pack, inner sep=34pt] at (0,.2) (out) {};
        \node[link,label=90:$6$] at ($(f)!.5!(h)$) (link1) {};
        \node[link,label=0:$7$] at (-2.4,-.25) (link2) {};
        \node[link,label={[label distance=0pt,xshift=-2pt]90:$1$}] at ($(f.75)!.5!(g.-135)$) (link3) {};
        \begin{scope}[label distance=-6pt]
          \draw[fr_out] (out.270) to (link1);
          \draw[fr_out] (out.190)	to (link2);
          \draw[fr_out] (out.155) to (link3);
          \draw[fr_out] (out.-35)	to node[pos=.5,link,label=10:$5$]{} (h.-30);
          \draw[to_out,fr_out] (out.15)	to[out=-165,in=-110] node[pos=.5,link,label
          distance=3pt,label=200:$3$]{} (out.70);
          \draw (f.30) to[out=0,in=145] (link1);
          \draw (f.-30) to[out=0,in=-145] (link1);
          \draw (h.180) to (link1);
          \draw (g.-60) to node[pos=.5,link,label={[xshift=2pt]10:$4$}]{} (h.120);
          \draw (f.45) to node[pos=.5,link,label=-10:$2$]{} (g.-105);
          \draw (f.75) to (link3);
          \draw (g.-135) to (link3);
        \end{scope}
      \end{tikzpicture}
  \end{align*}
\end{example}

\begin{example}\label{ex.no_inner}
  Note that we may have $s=0$, in which case there are no inner IO shells.
  For example, the following has $n_\omega=2$, and $n_{\out}=4$.
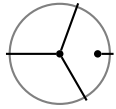
  \[
    \begin{tikzpicture}[inner WD]
      \node[link] (dot270) {};
      \node[outer pack, fit=(dot270), inner sep=12pt] (outer) {};
      \node[link] at ($(outer.0) - (0:5pt)$) (dot0) {};
      \draw[fr_out] (outer.0) -- (dot0);
      \draw[fr_out] (outer.70) -- (dot270);
      \draw[fr_out] (outer.180) -- (dot270);
      \draw[fr_out] (outer.300) -- (dot270);

      \pgfresetboundingbox
      \useasboundingbox (outer.270) rectangle (outer.north);
    \end{tikzpicture}
  \]
  In addition we may have $\ord n_{\out}=\ord 0$ in which case we render the two possible morphisms $\id_{\ord 0}$ and $\eta\cp\epsilon$ in $\ww(\ord0,\ord0)$ as the empty picture and a free floating dot, respectively.
\end{example}

\begin{remark}\label{rem.multiple}
  When multiple wires meet at a point, our convention will be to draw a dot iff the number of wires is different from two.
  
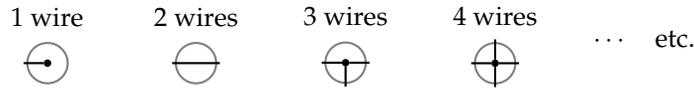
\begin{center}
    \begin{tikzpicture}[unoriented WD, font=\small]
      \node["1 wire"] (P1) {
        \begin{tikzpicture}[inner WD, surround sep=4pt]
          \node[link] (dot) {};
          \node[outer pack, fit=(dot)] (outer) {};
          \draw[to_out] (dot) -- (outer.west);
        \end{tikzpicture}
      };
      \node[right=4 of P1, "2 wires"] (P2) {
        \begin{tikzpicture}[inner WD, surround sep=4pt]
          \node[link, white] (dot) {};
          \node[outer pack, fit=(dot)] (outer) {};
          \draw[fr_out,to_out] (outer.west) -- (outer.east);
        \end{tikzpicture}
      };
      \node[right=4 of P2, "3 wires"] (P3) {
        \begin{tikzpicture}[inner WD, surround sep=4pt]
          \node[link] (dot) {};
          \node[outer pack, fit=(dot)] (outer) {};
          \draw[to_out] (dot) -- (outer.west);
          \draw[to_out] (dot) -- (outer.east);
          \draw[to_out] (dot) -- (outer.south);
        \end{tikzpicture}
      };
      \node[right=4 of P3, "4 wires"] (P4) {
        \begin{tikzpicture}[inner WD, surround sep=4pt, font=\tiny]
          \node[link] (dot) {};
          \node[outer pack, fit=(dot)] (outer) {};
          \draw[to_out] (dot) -- (outer.west);
          \draw[to_out] (dot) -- (outer.east);
          \draw[to_out] (dot) -- (outer.south);
          \draw[to_out] (dot) -- (outer.north);
        \end{tikzpicture}
      };
      \node[right=5 of P4, "$\cdots$\quad etc."] (etc){};

      \pgfresetboundingbox
      \useasboundingbox ($(P1.north west)+(0,10pt)$) rectangle ($(etc)+(0,-5pt)$);
    \end{tikzpicture}
  \end{center}
  When wires intersect and we do not draw a black dot, the intended interpretation is that the wires are \emph{not connected}:
  \begin{tikzpicture}[baseline=(P1.-20)]
    \node (P1) {
      \begin{tikzpicture}[inner WD, surround sep=2pt, font=\tiny]
        \node[link] (dot) {};
        \node[outer pack, fit=(dot)] (outer) {};
        \draw[to_out] (dot) -- (outer.west);
        \draw[to_out] (dot) -- (outer.east);
        \draw[to_out] (dot) -- (outer.south);
        \draw[to_out] (dot) -- (outer.north);
      \end{tikzpicture}
    };
    \node[right=0.25 of P1] (P2)	{
      \begin{tikzpicture}[inner WD, surround sep=2pt, font=\tiny]
        \node[link, white] (dot) {};
        \node[outer pack, fit=(dot)] (outer) {};
        \draw[fr_out,to_out] (outer.east) -- (outer.west);
        \draw[fr_out,to_out] (outer.south) -- (outer.north);
      \end{tikzpicture}
    };
    \path (P1) -- node[midway](N){$\neq$} (P2);
  \end{tikzpicture}.
\end{remark}

The following examples give a flavor of how composition, monoidal product, and
$2$-morphisms are represented using this graphical notation.

\begin{example}[Composition as substitution]\label{ex.comp_as_subst}
  Composition of morphisms is described by \define{nesting} of their depictions. Let $\omega'\colon \ord n'\to \ord n_1$ and $\omega\colon \ord n_1 \to \ord n_\out$ be morphisms in $\ww$. Then the composite morphism $\omega'\cp \omega\colon \ord n' \to \ord n_\out$ is depicted by
  \begin{enumerate}[nolistsep, noitemsep]
  \item drawing  $\omega'$ inside the inner circle of the picture for $\omega$,
  \item erasing the IO shell representing $\ord n_1$,
  \item amalgamating any connected black dots into a single black dot,
  \item removing either
    \begin{enumerate}[nolistsep,noitemsep, label=(\roman*)]
    \item all but one of the black dots not connected to a IO shell (if $n' = n_\out = 0$) or
    \item all black dots not connected to a IO shell (if $n' \ne 0$ or $n_\out\ne 0$).
    \end{enumerate}
  \end{enumerate}
  Note that step 3 corresponds to taking pushouts in $\finset$, while step 4 corresponds to taking the poset reflection.

  As a shorthand for composition, we simply draw one morphism directly substituted into another, as per step 1. For example, we have
  
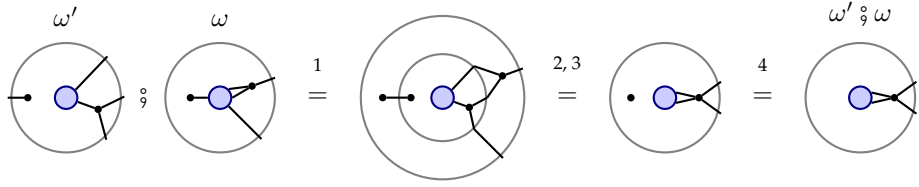
\begin{center}
    \begin{tikzpicture}[unoriented WD, font=\small]
      \node["$\omega'$"] (P1a) {
        \begin{tikzpicture}[inner WD, pack size=8pt]
          \node[packs] (a) {};
          \node[outer pack, inner sep=10pt, fit=(a)] (outer) {};

          \node[link] (link1) at ($(a.west)!.6!(outer.west)$) {};
          \node[link] (link3) at ($(a.-20)!.5!(outer.-20)$) {};

          \draw[fr_out] (outer.west) -- (link1);
          \draw[to_out] (a.40) -- (outer.45);
          \draw (a.-20) -- (link3);
          \draw[to_out] (link3) -- (outer.0);
          \draw[to_out] (link3) -- (outer.-45);
        \end{tikzpicture}
      };
      \node[right=1 of P1a, "$\omega$"] (P1b) {
        \begin{tikzpicture}[inner WD, pack size=8pt]
          \node[packs] (c) {};
          \node[outer pack, inner sep=10pt, fit=(c)] (outer2) {};

          \node[link] (link4) at ($(c.west)!.4!(outer2.west)$) {};
          \node[link] (link5) at ($(c.20)!.5!(outer2.20)$) {};

          \draw (c.west) -- (link4);
          \draw (c.45) -- (link5);
          \draw (c.0) -- (link5);
          \draw[to_out] (link5) -- (outer2.20);
          \draw[to_out] (c.-45) -- (outer2.-45);
        \end{tikzpicture}
      };
      \node[right=3 of P1b] (P2) {
        \begin{tikzpicture}[inner WD, pack size=8pt]
          \node[packs] (a) {};
          \node[outer pack, inner sep=7pt, fit=(a)] (c) {};
          \node[outer pack, inner sep=5pt, fit=(c)] (outer2) {};

          \node[link] (link1) at ($(a.west)!.6!(c.west)$) {};
          \node[link] (link3) at ($(a.-20)!.5!(c.-20)$) {};

          \draw (c.west) -- (link1);
          \draw (a.40) -- (c.45);
          \draw (a.-20) -- (link3);
          \draw (link3) -- (c.0);
          \draw (link3) -- (c.-45);

          \node[link] (link4) at ($(c.west)!.4!(outer2.west)$) {};
          \node[link] (link5) at ($(c.20)!.5!(outer2.20)$) {};

          \draw (c.west) -- (link4);
          \draw (c.45) -- (link5);
          \draw (c.0) -- (link5);
          \draw[to_out] (link5) -- (outer2.20);
          \draw[to_out] (c.-45) --  (outer2.-45);
        \end{tikzpicture}
      };
      \node[right=3 of P2] (P3) {
        \begin{tikzpicture}[inner WD, pack size=8pt]
          \node[packs] (c) {};
          \node[outer pack, inner sep=10pt, fit=(c)] (outer2) {};

          \node[link] at ($(c.180)!.5!(outer2.180)$) {};
          \node[link] (link) at ($(c.0)!.5!(outer2.0)$) {};

          \draw (c.25) -- (link);
          \draw (c.-25) -- (link);
          \draw[to_out] (link) -- (outer2.15);
          \draw[to_out] (link) -- (outer2.-15);
        \end{tikzpicture}
      };
      \node[right=3 of P3, "$\omega'\cp\omega$"] (P4) {
        \begin{tikzpicture}[inner WD, pack size=8pt]
          \node[packs] (c) {};
          \node[outer pack, inner sep=10pt, fit=(c)] (outer2) {};

          \node[link] (link) at ($(c.0)!.5!(outer2.0)$) {};

          \draw (c.25) -- (link);
          \draw (c.-25) -- (link);
          \draw[to_out] (link) -- (outer2.15);
          \draw[to_out] (link) -- (outer2.-15);
        \end{tikzpicture}
      };
      \node (P1) at ($(P1a.east)!.5!(P1b.west)$) {$\cp$};
      \node (E1) at ($(P1b.east)!.5!(P2.west)$) {$=$};
      \node[above=0 of E1] {\tiny 1};
      \node (E2) at ($(P2.east)!.5!(P3.west)$) {$=$};
      \node[above=0 of E2] {\tiny 2, 3};
      \node (E3) at ($(P3.east)!.5!(P4.west)$) {$=$};
      \node[above=0 of E3] {\tiny 4};
    \end{tikzpicture}
  \end{center}

  For the more general $n$-ary or operadic case, we may obtain the composite
  \[
    (\ord n_1 + \dots + \ord n_{i-1} + \omega' +
    \ord n_{i+1} + \dots + \ord n_j) \cp \omega
  \]
  of any two morphisms $\omega'\colon \ord n'_1+ \dots + \ord n'_{s'} \to \ord n_i$ and $\omega\colon \ord n_1+ \dots + \ord n_s \to \ord n_\out$, with $1\leq i\leq s$, by substituting the picture for $\omega'$ into the $i$\textsuperscript{th} inner circle of the picture for $\omega$, and following a procedure similar to that in \cref{ex.comp_as_subst}.
\end{example}

\begin{example}[Monoidal product as juxtaposition]
   Recall \cref{cor.mor_tensors_ww}: $\id_{\ord 0}$ acts as an identity under tensoring of morphisms. In our graphical notation the monoidal product of two morphisms in $\ww$ is simply their juxtaposition if neither morphism is $\eta\cp\epsilon$. For example, we might have:
  \begin{center}
    \begin{tikzpicture}[unoriented WD, font=\small]
      \node (P1a) {
        \begin{tikzpicture}[inner WD, pack size=8pt]
          \node[packs] (a) {};
          \node[packs, below right=.1 and 2 of a] (b) {};
          \node[outer pack, fit=(a) (b)] (outer) {};
          \draw (a.180) -- (a.180-|outer.west);
          \draw (a.20) to[out=0, in=120] (b.140);
          \draw (a.-30) to[out=-40, in=180] (b.180);
          \draw (b.east) -- (b.east-|outer.east);
        \end{tikzpicture}
      };
      \node[below=1 of P1a] (P1b) {
        \begin{tikzpicture}[inner WD, pack size=8pt]
          \node[packs, below=2 of $(a)!.5!(b)$] (c) {};
          \node[link, right=.8 of c] (link1) {};
          \node[outer pack, fit=(link1) (c)] (outer) {};
          \draw (c.20) -- (link1);
          \draw (c.-20) -- (link1);
          \draw[to_out] (link1) -- (link1-|outer.east);
        \end{tikzpicture}
      };
      \node (P1) at ($(P1a.south)!.5!(P1b.north)$) {$+$};
      \node[right=5 of $(P1a.north)!.5!(P1b.south)$] (e) {$=$};
      \node[right=1 of e] (P2) {
        \begin{tikzpicture}[inner WD, pack size=8pt]
          \node[packs] (a) {};
          \node[packs, below right=.1 and 2 of a] (b) {};
          \node[packs, below=2 of $(a)!.5!(b)$] (c) {};
          \node[link, right=.8 of c] (link1) {};
          \node[outer pack, fit=(a) (b) (c)] (outer) {};
          \draw (a.180) -- (a.180-|outer.west);
          \draw (a.20) to[out=0, in=120] (b.140);
          \draw (a.-30) to[out=-40, in=180] (b.180);
          \draw[to_out] (b.east) -- (b.east-|outer.east);
          \draw (c.20) -- (link1);
          \draw (c.-20) -- (link1);
          \draw (link1) -- (link1-|outer.east);
        \end{tikzpicture}
      };

      \pgfresetboundingbox
      \useasboundingbox (P1b.190-|P1a.west) rectangle (P2.east|-P1a.150);
    \end{tikzpicture}
  \end{center}
  By \cref{cor.mor_tensors_ww}, $\omega+\eta\cp\epsilon=\omega$ if and only if $\omega\ne\id_{\ord 0}$, and so we render the monoidal product $\omega+\eta\cp\epsilon$ as $\omega$, unless $\omega=\id_{\ord 0}$, in which case it is once more $\eta\cp\epsilon$.
\end{example}

\begin{example}[$2$-morphisms as breaking wires and removing disconnected black dots]\label{lemma.breaking}
  Let $\omega,\omega'\colon \ord n \to \ord n_\out$ be morphisms in $\ww$. Each is canonically represented by a cospan of finite sets, \[ \ord n\to \ord n_{\omega}\from \ord n_\out \qqand \ord n\to \ord n_{\omega'}\from \ord n_\out\ .\] By definition, there exists a $2$-morphism $\omega\leq\omega'$ iff there is a function $x\colon \ord n_{\omega'}\to \ord n_\omega$ making the requisite diagrams commute. For any element $i\in\ord n_\omega$, the pre-image $x^{*}(i)$ is either empty, has one element, or has multiple elements. In the first case, the pair of pictures depicting each side of the inequality $\omega\leq\omega'$ would show dot $i$ being removed; in the second case, it would show dot $i$ remaining as it was; and in the third case, it would show a connection being broken at dot $i$. For example, we have $2$-morphisms
  \begin{tikzpicture}[inner WD,baseline=(P1.-20)]
    \node (P1) {
      \begin{tikzpicture}[inner WD,baseline=(current bounding box.south)]
        \node[link] (dot) {};
        \node[outer pack, surround sep=3pt, fit=(dot)] (outer) {};
      \end{tikzpicture}
    };
    \node[right=2 of P1] (P2) {
      \begin{tikzpicture}[inner WD,baseline=(current bounding box.south)]
        \node[link, white] (dot) {};
        \node[outer pack, surround sep=3pt, fit=(dot)] (outer) {};
      \end{tikzpicture}
    };
    \node at ($(P1.east)!.5!(P2.west)$) {$\leq$};
  \end{tikzpicture}
  and
  \begin{tikzpicture}[inner WD,baseline=(P3.-20)]
    \node (P3) {
      \begin{tikzpicture}[inner WD,baseline=(current bounding box.south)]
        \node[link] (dot) {};
        \node[outer pack, surround sep=3pt, fit=(dot)] (outer) {};
        \draw[to_out] (dot) -- (outer.0);
        \draw[to_out] (dot) -- (outer.120);
        \draw[to_out] (dot) -- (outer.240);
      \end{tikzpicture}
    };
    \node[right=2 of P3] (P4) {
      \begin{tikzpicture}[inner WD,baseline=(current bounding box.south)]
        \node[link, draw=none,fill=none] (fake dot) {};
        \node[outer pack, surround sep=3pt, fit=(fake dot)] (outer) {};
        \node[link] (dot) at ($(outer.0)+(0:-5pt)$) {};
        \draw[to_out] (dot) -- (outer.0);
        \draw[to_out,fr_out] (outer.120) to[out=300, in=60] (outer.240);
      \end{tikzpicture}
    };
    \node at ($(P3.east)!.5!(P4.west)$) {$\leq$};
  \end{tikzpicture}.
\end{example}

\section{Supply}\label{sec.supply}

It often happens that every object in a symmetric monoidal category $\cat{C}$ is equipped with the same sort of algebraic structure -- say coming from a prop $\pp$ -- with the property that these algebraic structures are compatible with the monoidal structure. In \cite{fong2019supplying}, we refer to this situation by saying that $\cat{C}$ \emph{supplies} $\pp$. For our purposes we need to slightly generalise this theory, from props to po-props and from symmetric monoidal categories $\cat{C}$ to symmetric monoidal po-categories $\cc$.

\begin{definition}[Supply]\label{def.supply}
  Let $\pp$ be a po-prop and $\cc$ a symmetric monoidal po-category. A \define{supply of $\pp$ in $\cc$} consists of a strong monoidal po-functor $s_c\colon\pp\to\cc$ for each object $c\in\cc$, such that
  \begin{enumerate}[label=(\roman*)]
  \item $s_c(m)=c\tpow{m}$ for each $m\in\nn$,
  \item the strongator $c\tpow{m}\otimes c\tpow{n}\to c\tpow{(m+n)}$ is equal to the associator for each $m,n\in\nn$,
  \item the following diagrams commute for every $c,d\in\cc$ and $\mu\colon m\to n$ in $\pp$, where the $\sigma$'s are the symmetry isomorphisms from \eqref{eqn.symmetry}.
    \begin{diagram}[][][baseline=(s.base)]\label{eqn.supply_commute_tensors}
      \node(1)[]{$c\tpow m\otimes d\tpow m$};
      \node(2)[right=2cm of 1.east,anchor=west]{$c\tpow n\otimes d\tpow n$};
      \node(3)[below= of 1]{$(c\otimes d)\tpow m$};
      \node(4)[below= of 2]{$(c\otimes d)\tpow n$};
      \draw[a](1)to node[la,above]{$s_{c}(\mu)\otimes s_{d}(\mu)$}(2);
      \draw[a](3)to node[la,below]{$s_{c\otimes d}(\mu)$}(4);
      \draw[a](1)to node(s)[la,left]{$\sigma$}(3);
      \draw[a](2)to node[la,right]{$\sigma$}(4);
      \node(5)[right=2cm of 2.east]{$I$};
      \node(6)[right= of 5]{$I$};
      \node(7)[below= of 5]{$I\tpow m$};
      \node(8)[below= of 6]{$I\tpow n$};
      \draw[a](5)to node[la,left]{$\sigma$}(7);
      \draw[a](6)to node[la,right]{$\sigma$}(8);
      \draw[a](7)to node[la,below]{$s_{I}(\mu)$}(8);
      \draw[d](5)to(6);
    \end{diagram}
  \end{enumerate}
  We often denote the morphism $s_c(\mu)$ in $\cc$ simply by $\mu_c\colon c\tpow{m}\to c\tpow{n}$ for typographical reasons; i.e.\ we elide explicit mention of $s$.

  We further say that $f\colon c\to d$ in $\cc$ is a \define{lax $s$-homomorphism} (resp.\ \define{oplax $s$-homomorphism}) if, for each $\mu\colon m\to n$ in the prop $\pp$, there is a $2$-morphism as shown in the left-hand (resp.\ right-hand) diagram:
  \begin{diagram*}[.][4][node distance=1.5cm]
    \node(1)[]           {$c\tpow{m}$};
    \node(2)[right of= 1]{$d\tpow{m}$};
    \node(3)[below of= 1]{$c\tpow{n}$};
    \node(4)[below of= 2]{$d\tpow{n}$};
    \draw[a](1)to node[la,above]{$f\tpow m$}(2);
    \draw[a](1)to node[la,left]{$\mu_{c}$}(3);
    \draw[a](2)to node[la,right]{$\mu_{d}$}(4);
    \draw[a](3)to node[la,below]{$f\tpow n$}(4);
    \path (2) to node[la,sloped]{$\geq$}(3);
    \node(1)[right=3cm of 2]{$c\tpow{m}$};
    \node(2)[right of= 1]   {$d\tpow{m}$};
    \node(3)[below of= 1]   {$c\tpow{n}$};
    \node(4)[below of= 2]   {$d\tpow{n}$};
    \draw[a](1)to node[la,above]{$f\tpow{m}$}(2);
    \draw[a](1)to node[la,left]{$\mu_{c}$}(3);
    \draw[a](2)to node[la,right]{$\mu_{d}$}(4);
    \draw[a](3)to node[la,below]{$f\tpow{n}$}(4);
    \path (3) to node[la,sloped]{$\leq$}(2);
  \end{diagram*}

  Since $\cc$ is locally posetal, if $f$ is both a lax and an oplax $s$-homomorphism, then these diagrams commute and we simply say $f$ is an \define{$s$-homomorphism}.

  We say that $\cc$ \define{(lax-/oplax-) homomorphically supplies $\pp$} if every morphism $f$ in $\cc$ is a (lax/oplax) $s$-homomorphism.
\end{definition}

\begin{example}
  An important class of examples of homomorphic supply are those categories with finite products. By the main  theorem of \cite{fox1976coalgebras}, replicated below, such categories are precisely the discretely ordered po-categories that homomorphically supply cocommutative comonoids (\cref{def.comon_mon_selfduals}).
\end{example}

\begin{proposition}\label{prop.fox}
  A category $\cat{C}$ has finite products iff it can be equipped with a homomorphic supply of commutative comonoids. If $\cat{C}$ and $\cat{D}$ have finite products, a functor $\cat{C}\to\cat{D}$ preserves them iff it preserves the supply of comonoids.
\end{proposition}

\begin{example}
  We shall meet another large class of examples of lax homomorphic supply in \cref{chap.relational_pocats}, wherein we shall find that regular categories are equivalently po-categories with a lax homomorphic supply of $\ww$ and some additional structure.
\end{example}

Now that we have established the definition of supply we collect some results which we will variously leverage in our later sections.

\begin{notation}[Coproduct of symmetric monoidal po-categories]
  We will write $\bigsqcup$ to denote the coproduct of symmetric monoidal po-categories in the $2$-category \textsf{SMC} of symmetric monoidal po-categories, symmetric monoidal po-functors, and monoidal natural transformations.
\end{notation}

\begin{warning}\label{warn.coprod}
  The coproduct of symmetric monoidal categories in \textsf{SMC} \emph{does not} coincide with the po-categorical coproduct in $\CCat{PoCat}$. Instead, for a set $J$ and symmetric monoidal po-categories $\{\cc_{j}\}_{j\in J}$,
  \[
    \ob\Big(\bigsqcup_{j\in J}\cc_{j}\Big)\coloneqq\Big\{ (c_{j})\in\prod_{j\in J}\ob\cc_{j}\ \Big|\  c_{j}=I_{j} \text{ for all but finitely many } j\in J\Big\}\ .
  \]
  See \cite[Theorem 2.2 \& Appendix A]{fong2019supplying} for details.
\end{warning}

\begin{lemma}\label{lemma.irritating_supply_coprod}
  Let $\cc$ be a symmetric monoidal po-category supplying $\pp$ a po-prop. Then for any set $J$, the supply $s$ of $\pp$ in $\rr$ extends to a supply $\tilde s$ of $\cc$ in $\bigsqcup_{J}\cc$ such that \[\tilde s_{(c_{j})_{J}}\cp\pi_{i} =
    \begin{cases}
      I,& c_{i}=I\\
      s_{c_{i}},& \mathrm{otherwise}
    \end{cases}\quad ,\]
  as functors $\pp\to\cc$. If in particular $\cc$ is symmetric strict monoidal then $\tilde s$ satisfies $\tilde s_{(c_{j})_{j}}\cp\pi_{i}=s_{c_{i}}$.
\end{lemma}

\begin{proof}
  The supply conditions (\cref{def.supply}) hold ``point-wise'' for each $c_{i}$, and as the symmetries of $\bigsqcup_{J}\cc$ are point-wise those of $\cc$, the functors $\tilde s_{(c_{i})_{I}}$ constitute a supply of $\pp$ in $\bigsqcup\cc$.
\end{proof}

In what follows, we do not produce proofs for \cref{prop.supply_coprod,prop.p_supplies_itself,cor.change_of_supply} as these were essentially proven in \cite[Propositions 3.13, 3.14, and 3.21]{fong2019supplying}; the change from props to po-props makes no difference in this context.

\begin{proposition}\label{prop.supply_coprod}
  A supply $s$ of $\pp$ in $\cc$ induces a strong monoidal po-functor \[s^{\sqcup}\colon\bigsqcup\limits_{\ob\cc}\pp\to\cc\] uniquely determined by $\iota_{c}\cp s^{\sqcup}=s_{c}$ for each $c\in\cc$ and inclusion $\iota_{c}\colon\pp\rightarrowtail\bigsqcup_{\ob\cc}\pp$.\hfill$\qed$
\end{proposition}

\begin{proposition}\label{prop.p_supplies_itself}
  Let $\pp$ be a po-prop. Then there is a supply of $\pp$ in $\pp$ where the functors $s_{c}$ for $c\in\ob\pp$ and $\mu\colon m\to n$ of $\pp$ are given by
  \[ s_{c}(\mu)\colon s_{c}(m)=\underbrace{c+\ldots +c}_{m}=\underbrace{m+\ldots+m}_{c}\xrightarrow{\mu+\ldots+\mu} \underbrace{n+\ldots+n}_{c}=\underbrace{c+\ldots+c}_{n}=s_{c}(n)\ .\hfill\qed\]
\end{proposition}

\begin{proposition}[Change of supply]\label{cor.change_of_supply}
  Let $G\colon\pp\to\qq$ be a po-prop functor. For any supply $s$ of $\qq$ in $\cc$, we have a supply $(G\cp s)$ of $\pp$ in $\cc$.\hfill$\qed$
\end{proposition}

\begin{example}\label{ex.wiring_comonoids}
  Any supply of $\ww$ in $\cc$ induces a supply of cocommutative comonoids in $\cc$ by change of supply and \cref{def.comon_mon_selfduals}, and hence implies that $\cc$ has finite products (among left adjoints) by \cite{fox1976coalgebras}.
\end{example}

\begin{example} \label{ex.W_self_dual}
  We say that a symmetric monoidal po-category $(\cc,I,\otimes)$ is self-dual compact closed iff it supplies self duals (see \cref{def.comon_mon_selfduals}). Thus $\ww$ is self-dual compact closed by \cref{prop.p_supplies_itself}, as is any po-category $\cc$ supplying $\ww$ by \cref{cor.change_of_supply}.
\end{example}

\begin{example}\label{ex.corel}
  Consider the po-prop of finite sets and corelations, where a morphism $m\to n$ is an equivalence relation on $m+n$, and the order is given by coarsening. Since this po-prop receives a po-prop functor from $\ww$ (see \cref{prop.ww_explicit}), it supplies $\ww$ by \cref{prop.p_supplies_itself,cor.change_of_supply}.
\end{example}

\begin{definition}[Preservation of supply]\label{def.preserve_supply}
  Let $\pp$ be a po-prop, $\cc$ and $\dd$ symmetric monoidal po-categories, and suppose $s$ is a supply of $\pp$ in $\cc$ and $t$ is a supply of $\pp$ in $\dd$. We say that a strong symmetric monoidal po-functor $(F,\varphi)\colon\cc\to\dd$ \define{preserves the supply} if the strongators $\varphi$ provide an isomorphism $t_{Fc}\cong (s_c\cp F)$ of po-functors $\pp\to\dd$ for each $c\in\cc$.
\end{definition}

Unpacking, a strong monoidal po-functor $(F,\varphi)$ preserves the supply iff the following diagram commutes for each morphism $\mu\colon m\to n$ in $\pp$ and object $c\in\cc$:
\begin{diagram}[][][baseline=(p.base)]\label{eqn.unpack_preserve_supply}
  \node(1)[]{$F(c)\tpow m$};
  \node(2)[right= of 1]{$F(c)\tpow n$};
  \node(3)[below= of 1]{$F(c\tpow m)$};
  \node(4)[below= of 2]{$F(c\tpow n)$};
  \draw[a](1)to node[la,above]{$\mu_{F(c)}$}(2);
  \draw[a](3)to node[la,below]{$F(\mu_{c})$}(4);
  \draw[a](1)to node(p)[la,left]{$\varphi$}(3);
  \draw[a](2)to node[la,right]{$\varphi$}(4);
\end{diagram}

One common example of preservation of supply arises from the supply of a po-prop $P$ in itself.

\begin{lemma}\label{lemma.supply_preserves_supply}
  Let $\cc$ be a po-category supplying $\pp$ a po-prop, and let $c\in\ob\cc$ be an object. Then the strong symmetric monoidal po-functor $s_{c}\colon\pp\to\cc$ determined by the supply of $\pp$ in $\cc$ preserves the supply of $\pp$ in itself of \cref{prop.p_supplies_itself}.
\end{lemma}

\begin{proof}
  By unwinding this claim we may see that it reduces to iterated applications of the supply conditions \eqref{eqn.supply_commute_tensors} for $s_{c}$.
\end{proof}

\section{Supplying \texorpdfstring{$\ww$}{W}}\label{sec.ww_suppliers}

Our definition of regular calculi will be built upon a po-category supplying $\ww$. The supply of $\ww$ provides a notion of equality: roughly speaking, the morphisms in $\ww$, represented by cospans, allow the marking of any number of objects as `the same', and composition in $\ww$ takes the transitive closure of two markings, re-completing the result to an equivalence relation. Graphically, we reflect this in our depicting of the morphisms as $\ww$ as wirings: wires connect objects that are to be considered `the same'.

In this section we explore a few properties of po-categories supplying $\ww$, emphasising the sense in which their underlying 1-categories can be understood as undirected thus fulfilling a promise from the introduction. We have already seen a hint of this in \cref{ex.W_self_dual}, where we noted that any po-category supplying $\ww$ has an induced supply of self-dualities (i.e.\ is self-dual compact closed). This implies that we may bijectively move tensor factors of the domain to the codomain and visa-versa by making use of the supplied images of the morphisms $\thecup$ and $\thecap$. Roughly speaking, this allows us to treat the distinction between domain and codomain as a matter of bookkeeping. We give two important examples of this, depicting them in string diagrams.

The first is the transpose of a morphism, which exchanges domain and codomain.
\begin{definition}[Transpose]
  Let $f\colon c \to d$ be a morphism in a self dual compact closed po-category $\cc$. We define its \define{transpose} $f\tp\colon d \to c$ to be the morphism
  \[
    f\tp
    \coloneqq\quad
    \begin{tikzpicture}[inner WD, baseline=(f.-20)]
      \node[oshellr, syntax] (f) {$f$};
      \draw (f.west) --
      ++(-.5, 0) to[out=180, in=180]
      ++(0, 1.25) --
      ++(4.5, 0)
      ;
      \draw (f.east) --
      ++(.5,0) to[out=0, in=0]
      ++(0, -1.25) --
      ++(-4.5, 0)
      ;
    \end{tikzpicture}
  \]
\end{definition}

\begin{notation}[Transpose]\label{notation.transpose}
  Note that we have used the asymmetric icon
  $
  \begin{tikzpicture}[inner WD, baseline=(f.-20)]
    \node[oshellr, syntax] (f) {$f$};
    \draw (f.west) -- +(-.5,0);
    \draw (f.east) -- +(.5,0);
  \end{tikzpicture}
  $
  to denote $f$. This asymmetry allows us to distinguish the domain against the codomain without appealing to the orientation of the picture: the domain $c$ is attached at largest angle of the kite, while the domain $d$ is attached at the least angle.

  In line with this convention, and the adjunction equations \eqref{eqn.yanking}, we depict the transpose simply by the 180-degree rotation of $f$:
  \[
    f\tp =
    \begin{tikzpicture}[inner WD, baseline=(f.-20)]
      \node[oshellr, syntax] (f) {$f\tp$};
      \draw (f.west) -- +(-.5,0);
      \draw (f.east) -- +(.5,0);
    \end{tikzpicture}
    =
    \begin{tikzpicture}[inner WD, baseline=(f.-25)]
      \node[oshelll, syntax] (f) {$f\phantom{\tp}$};
      \draw (f.west) -- +(-.5,0);
      \draw (f.east) -- +(.5,0);
    \end{tikzpicture}
  \]
  We will make liberal use of asymmetric icons when it will help to have rotation denote transpose.
\end{notation}

Transposition gives an isomorphism $\cc(c,d) \cong \cc(d,c)$ of posets. The second example of moving factors between domain and codomain is given by naming and unfolding. These give an isomorphism $\cc(c,d) \cong \cc(I,c\otimes d)$.

\begin{definition}[Name and unfolding]\label{def.name_unfolding}
  Let $f\colon c\to d$ and $g\colon I \to c\otimes d$ be morphisms in a self dual compact closed po-category. We define the \define{name} $\name{f}\colon I\to c\otimes d$ of $f$ and the \define{unfolding} $\unname{g}\colon c\to d$, to be respectively:
  \begin{center}
    \begin{tikzpicture}[inner WD, syntax]
      \node[oshellr] (f) {$f$};
      \draw (f.east) --
      ++(.5,0) -- +(.5, 0) coordinate (end)
      node[right] {$c$}
      ;
      \draw (f.west) --
      ++(-.5, 0) to[out=180, in=180]
      ++(0, 1.5) coordinate (top) --
      (top-|end)
      node[right] {$d$};
      ;

      \node[oshellr,right=4cm of end] (g) {$g\vphantom{f}$};
      \draw (g.30) --
      ++(1, 0) to[out=0, in=0]
      ++(0, 1) --
      ++(-4, 0)
      node[left] {$c$};
      ;
      \draw (g.-30) --
      ++(2, 0)
      node[right] {$d$};
      ;

      \begin{scope}[font=\normalsize]
        \node[left=1.5 of f] {$\name{f}\coloneqq$};
        \node[below=.5 of f] (formula) {$\thecup_c\cp(\id_c\otimes f)$};
        \node[left=1.5 of g] {$\unname{g}\coloneqq$};
        \node at (g|-formula) {${\rho_{c}}\inv\cp(\id_c\otimes g)\cp(\thecap_c\otimes\id_d)\cp\lambda_{d}$};
      \end{scope}
    \end{tikzpicture}
  \end{center}
\end{definition}

We invite the reader to construct the pleasing graphical proof of the following.

\begin{lemma}\label{lemma.name_unfolding_iso}
  The maps $\name{\cdot}$ and $\unname{\cdot}$ are mutually inverse.\hfill$\qed$
\end{lemma}

More generally, this `undirected' nature of 1-morphisms is a manifestation of the fact that the underlying 1-category of a po-category supplying $\ww$ is a \emph{hypergraph category}. While we will not go into this in detail, to give the rough idea, recall that we refer to the generating morphisms of $\ww$ as $(\eta,\mu,\epsilon,\delta)$; see \eqref{eqn.generating_wires}. These morphisms allow, visually, the termination, combining, initialisation, and splitting of wires. Moreover, the equations relating these generators in \eqref{eqn.equations_wires} define what is known as a \emph{special commutative frobenius monoid}, and capture precisely the fact that only the connectivity of their composites matter---for example, the special law says that splitting a wire and then combining the resulting pair is the same as doing nothing to that wire. For more on hypergraph categories, see \cite{fong2019hypergraph}. In our work here, rather than explore the intricacies of hypergraph categories, we will use wiring diagrams to capture this invariance.

\section{Wiring diagrams for po-categories supplying \texorpdfstring{$\ww$}{W}}\label{sec.wds_for_ww_suppliers}

We have used two methods for depicting morphisms in $\ww$. The first is the string diagrams of Joyal--Street, which we used in \cref{sec.prop_ww}, and which works for any monoidal category. In these diagrams, the domain of the morphism is represented on the left of the diagram, and the codomain is on the right. The second method is the pictures of \cref{sec.depict_ww}. For these diagrams, the domain of the morphism is represented by the interior blue circles, while the codomain is represented by the outer circle.

We may extend this latter notation to a notation for morphisms any po-category $\cc$ supplying $\ww$, which we term \emph{wiring diagrams}. The change in our notation is essentially that each wire is labelled by an object in $\cc$ and that in addition to IO shells, we may also have kites labelled by morphisms of $\cc$. We first define these IO shells and kites, and describe how to draw them.

\begin{definition}[IO shell]
  Let $\cc$ be a po-category supplying $\ww$. An \define{input-output shell}, or \define{IO shell}, $\Gamma = (n,\tau)$ in $\cc$ is a function $\tau\colon \ord{n} \to \ob\cc$.
\end{definition}

Whereas in picturing morphisms of $\ww$ (\cref{sec.depict_ww}) IO shells stand for objects, in a po-category $\cc$ supplying $\ww$ IO shells are intended to capture ``contexts'': specified, ordered collections of ``types''. These IO shells will help us organise and capture the calculus we need in order to support a regular logic of relations.

\begin{notation}[Depicting IO shells]
  We depict an IO shell as a circle with $n$ ports, labelled clockwise starting from its left with the objects $\tau(i) \in \ob\cc$:
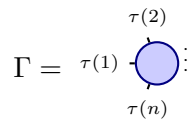
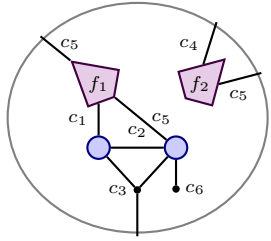
  \[
    \Gamma =
    \begin{aligned}
      \begin{tikzpicture}[inner WD]
        \node[pack,minimum size=16pt] (rho) {};
        \draw (rho.180) to[pos=1] node[left] (w) {$\tau(1)$} +(180:2pt);
        \draw (rho.110) to[pos=1] node[above] (n) {$\tau(2)$} +(110:2pt);
        \node at ($(rho.0)+(60:5pt)$) {$\vdots$};
        \draw (rho.-110) to[pos=1] node[below] (e) {$\tau(n)$} +(-110:2pt);
      \end{tikzpicture}
    \end{aligned}
  \]
  We shall abuse notation by writing $\Gamma$ for both the pair $(n,\tau)$ as well as the tensor product $\Gamma \coloneqq \tau(1) \otimes \dots \otimes \tau(n)\in\ob\cc$.
\end{notation}

\begin{definition}[Kite]
  Let $\cc$ be a po-category supplying $\ww$. A \define{kite} $(\kappa,f) \coloneqq (a,b,\tau_l,\tau_r,f)$ is a pair of natural numbers $a,b$, functions $\tau_l\colon \ord{a} \to \ob\cc$, $\tau_r\colon \ord{b} \to \ob\cc$, and a morphism $f\colon \bigotimes_{i \in \ord{a}} \tau_l(i) \to  \bigotimes_{i \in \ord{b}} \tau_r(i)$ in $\cc$.
\end{definition}

\begin{notation}[Depicting kites]
  We depict a kite as a kite with $a$ ports on the side with greater angle, $b$ ports on the side with lesser angle, with each port labelled by the objects $\tau_l(i), \tau_r(i) \in \ob\cc$ in sequence, and the kite itself labelled by $f$:
  \[
    \kappa =
    \begin{tikzpicture}[WD, baseline=(rho.base)]
      \node[oshellr,syntax,minimum size=21pt] (rho) {$f$};
      \draw (rho.120) to[pos=1] node[left,la] (n) {$\tau_l(1)$} +(150:10pt);
      \draw (rho.-120) to[pos=] node[left,la] (e) {$\tau_l(a)$} +(210:10pt);
      \path (n) to node[sloped]{$\ldots$} (e);
      \draw (rho.60) to[pos=1] node[right,la] (m) {$\tau_r(1)$} +(30:10pt);
      \draw (rho.-60) to[pos=1] node[right,la] (d) {$\tau_r(b)$} +(330:10pt);
      \path (m) to node[sloped]{$\ldots$} (d);
    \end{tikzpicture}
  \]
  As we do for IO shells, we shall abuse notation by writing $\kappa$ for both the tuple $(a,b,\tau_l,\tau_l)$ as well as the tensor product $\kappa \coloneqq \tau_l(1) \otimes \dots \otimes \tau_l(a) \otimes \tau_r(1) \otimes \dots \otimes \tau_r(b)$.
\end{notation}

A wiring diagram consists of IO shells and kites wired together by morphisms in the supply of $\ww$. Recall from \cref{prop.supply_coprod} that a supply of $\ww$ in $\cc$ induces a strong monoidal functor $s^\sqcup \colon \bigsqcup_{\ob\cc} \ww \to \cc$.

\begin{definition}[Wiring diagram]\label{def.extended_wiring_diagram}
  Let $\cc$ be a po-category supplying $\ww$. A \define{wiring diagram for $\cc$} is:
  \begin{enumerate}[label=(\roman*)]
  \item natural numbers $s$ and $k$,
  \item for each $i\in\{1,\dots, s,\text{out}\}$, an IO shell $\Gamma_i$,
  \item for each $i\in \{1, \dots, k\}$, a kite $(\kappa_i, f_i)$,
  \item a morphism $\omega^\ast\colon \Gamma_1 \otimes \dots \otimes \Gamma_s \otimes
    \kappa_1 \otimes \dots \otimes \kappa_k
    \to \Gamma_\out$ in the image of $s^\sqcup$ in $\cc$.
  \end{enumerate}
\end{definition}

\begin{notation}[Wiring diagrams for a supply]\label{notation.wiring_diagrams_supplying_ww}
  We depict wiring diagrams for $\cc$ in much the same way as we depict wiring diagrams for $\ww$ (see \cref{notation.picturing_ww}), except that wires are now labelled by objects of $\cc$, and we draw kites as well as IO shells.

  To give a bit more detail, since the morphism $\omega^\ast$ is in the image of $\bigsqcup_{\ob\cc} \ww$ under $s^\sqcup$, it may be canonically represented by a cospan of finite sets
  \[
    \begin{tikzcd}
      \ord{n_1}+\dots+\ord{n_s}+\ord{a_1}+\ord{b_1}+\dots+\ord{a_k}+\ord{b_k} \ar[r] \ar[dr,"{[\tau_1,\dots,\tau_s,\tau_{l_1},\dots,\tau_{r_k}]}"']
      & \ord{n_\omega} \ar[d,"\tau_\omega"]
      & \ord{n_{\text{out}}} \ar[l] \ar[dl,"\tau_{\text{out}}"] \\
      & \ob\cc
    \end{tikzcd}
  \]
  such that the above diagram commutes. We depict this cospan as we do in \cref{notation.picturing_ww}, labelling each wire with the object of $\cc$ indicted by the $\tau_i$.
\end{notation}

\begin{example} \label{ex.a_wiring_diagram}
  Here's a depiction of a wiring diagram in which $s = k= 2$, $\Gamma_1\coloneqq c_1\otimes c_2\otimes c_3$, $\Gamma_2\coloneqq c_2 \otimes c_5\otimes c_6\otimes c_3$, $\kappa_1 = c_1\otimes c_5 \otimes c_5$, $\kappa_2 = c_4 \otimes c_5$, and $\Gamma_\out\coloneqq c_5\otimes c_4\otimes c_5\otimes c_3$. Note in particular that $c_{5}$ appears twice in $\Gamma_{\out}$.
  \begin{center}
    \examplewiringdiagram{}
  \end{center}
\end{example}

Each wiring diagram itself represents a morphism $\Gamma_1 \otimes \dots \otimes \Gamma_s \to \Gamma_{\text{out}}$ in $\cc$.
\begin{definition}[The morphism represented by a wiring diagram] \label{def.morphism_represented_by_wd}
  Let
  \[
    (s,k,\{\Gamma_{i}\}_{i\in\ord s},\, \{(\kappa_i,f_i)\}_{i \in \ord k},\, \omega^\ast)
  \]
  be a wiring diagram. We say that this wiring diagram \define{represents} the following morphism of $\cc$.
  \[
    \omega \coloneqq \big(\Gamma_1 \otimes \dots \otimes \Gamma_s \otimes \name{f_1} \otimes \dots \otimes \name{f_k}\big) \cp \omega^\ast \colon \Gamma_1 \otimes \dots \otimes \Gamma_s \longrightarrow \Gamma_{\text{out}}
  \]
\end{definition}

\begin{example}
  The wiring diagram in \cref{ex.a_wiring_diagram} represents the following morphism $\Gamma_1\otimes \Gamma_2\to \Gamma_\out$ of $\cc$. Note in particular that $c_{5}$ appears twice in $\Gamma_{\out}$, that is, $\Gamma_{\out}$ includes the $c_{5}\otimes c_{5}$ up to symmetry.
  
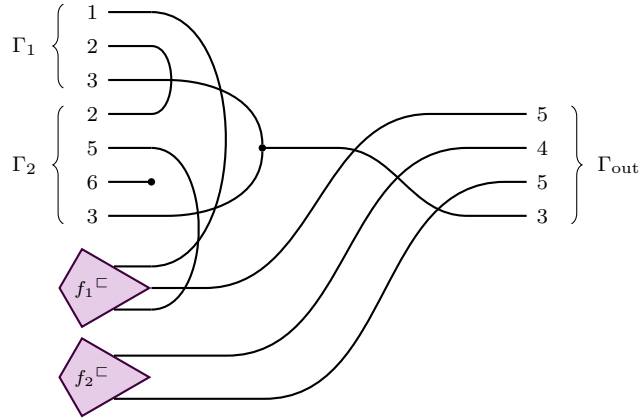
\begin{center}
    \begin{tikzpicture}
      \begin{scope}[font=\scriptsize,x=1em, decoration={brace, amplitude=4pt},y=3ex]
        \node (a1) {$1$};
        \node[below=0 of a1] (a2) {$2$};
        \node[below=0 of a2] (a3) {$3$};
        \node[below=0 of a3] (b1) {$2$};
        \node[below=0 of b1] (b2) {$5$};
        \node[below=0 of b2] (b3) {$6$};
        \node[below=0 of b3] (b4) {$3$};

        \node[right=15] at (b1) (out1) {$5$};
        \node[below=0 of out1] (out2) {$4$};
        \node[below=0 of out2] (out3) {$5$};
        \node[below=0 of out3] (out4) {$3$};

        \draw[decorate] ($(a3.west)+(-5pt, -3pt)$) to node[left=6pt] {$\Gamma_1$} ($(a1.west)+(-5pt, 3pt)$);
        \draw[decorate] ($(b4.west)+(-5pt, -3pt)$) to node[left=6pt] {$\Gamma_2$} ($(b1.west)+(-5pt, 3pt)$);
        \draw[decorate] ($(out1.east)+(5pt, 3pt)$) to node[right=6pt] {$\Gamma_\out$} ($(out4.east)+(5pt, -3pt)$);
      \end{scope}
      \begin{scope}[inner WD]
        \node[oshellr, syntax, below=of b4] (f1) {$\name{f_1}$};
        \node[oshellr, syntax, below=of f1] (f2) {$\name{f_2}$};

        \draw(a1) -- (a1-|f1.east) coordinate(a1e);
        \draw(f1.45) -- (f1.45-|f1.east) coordinate(f1t);
        \draw(f1t) to[out=0,in=0] (a1e);

        \draw(b2) -- (b2-|f1.east) coordinate(b2e);
        \draw(f1.-45) -- (f1.-45-|f1.east) coordinate(f1b);
        \draw(f1b) to[out=0,in=0] (b2e);

        \draw(f1) to ($(f1)+(1.5,0)$) to[out=0,in=180] ($(out1)-(1.5,0)$) to (out1);

        \draw(f2.45) to ($(f2.45)+(1.5,0)$) to[out=0,in=180] ($(out2)-(1,0)$) to (out2);
        \draw(f2.-45) to ($(f2.-45)+(2,0)$) to[out=0,in=180] ($(out3)-(0.5,0)$) to (out3);

        \draw(b3) -- (b3-|f1.east) node[link]{};

        \draw(a2) to (a2-|f1.east) to[out=0,in=0] (b1-|f1.east) to (b1);

        \node(l)[right=2 of b2,link]{};
        \draw(a3) to ($(a3)+(1,0)$) to[out=0,in=90] (l);
        \draw(b4) to ($(b4)+(1,0)$) to[out=0,in=270] (l);
        \draw(l) to ($(l)+(1,0)$) to[out=0,in=180] ($(out4)-(1,0)$) to (out4);
      \end{scope}
    \end{tikzpicture}
  \end{center}
  Note that we may use the manipulations afforded to us by the string diagram calculus to draw this as the following string diagram instead
  \begin{diagram*}[,][out4]
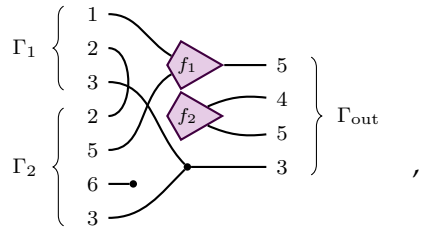
    \begin{scope}[font=\scriptsize,x=1em, decoration={brace, amplitude=4pt},y=3ex]
      \node (a1) {$1$};
      \node[below=0 of a1] (a2) {$2$};
      \node[below=0 of a2] (a3) {$3$};
      \node[below=0 of a3] (b1) {$2$};
      \node[below=0 of b1] (b2) {$5$};
      \node[below=0 of b2] (b3) {$6$};
      \node[below=0 of b3] (b4) {$3$};

      \coordinate (a2a3) at ($(a2)!.5!(a3)$);
      \node[right=6] at (a2a3) (out1) {$5$};
      \node[below=0 of out1] (out2) {$4$};
      \node[below=0 of out2] (out3) {$5$};
      \node[below=0 of out3] (out4) {$3$};

      \draw[decorate] ($(a3.west)+(-5pt, -3pt)$) to node[left=6pt] {$\Gamma_1$} ($(a1.west)+(-5pt, 3pt)$);
      \draw[decorate] ($(b4.west)+(-5pt, -3pt)$) to node[left=6pt] {$\Gamma_2$} ($(b1.west)+(-5pt, 3pt)$);
      \draw[decorate] ($(out1.east)+(5pt, 3pt)$) to node(go)[right=6pt] {$\Gamma_\out$} ($(out4.east)+(5pt, -3pt)$);
    \end{scope}
    \begin{scope}[inner WD]
      \coordinate (helper) at ($(a2a3)!.5!(out1)$);
      \coordinate (a1b1) at ($(a1)!.5!(b1)$);
      \coordinate (a2b4) at ($(a2)!.5!(b4)$);
      \coordinate (a3b3) at ($(a3)!.4!(b3)$);
      \coordinate (out4out5) at ($(out2)!.5!(out3)$);
      \node[oshellr, syntax] at (helper|-a1b1) (f1) {$f_1$};
      \node[oshellr, syntax] at (helper|-out4out5) (f2) {$f_2$};
      \node[link, right=of b3] (dot6) {};
      \node[link] at (helper|-out4) (dot3) {};
      \draw (a1) to[out=0,in=150] (f1);
      \draw (b2) to[out=0,in=210] (f1);
      \draw (f1) -- (out1);
      \draw (b3) -- (dot6);
      \draw (a2) to[out=0, in=0] (b1);
      \draw (a3) to[out=0, in=135] (dot3);
      \draw (b4) to[out=0, in=225] (dot3);
      \draw (dot3) -- (out4);
      \draw (out2) to[bend right=15pt] (f2);
      \draw (out3) to[bend left=15pt] (f2);
    \end{scope}
  \end{diagram*}
  which is a bit easier to read.
\end{example}

\begin{remark}
  Wiring diagrams are robust to topological deformation, so long as we keep the connections to the outer IO shell---our ``context'', so to speak---consistent. For example, the following pictures all represent the same wiring diagram and indeed same morphism of $\cc$:
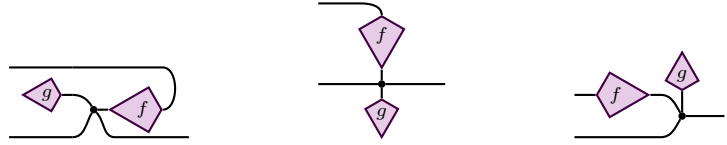
  \begin{center}
    \begin{tikzpicture}[inner WD]
      \coordinate (g1);
      \coordinate[below=2 of g1] (g2);
      \node[link] at ($(g1)!.5!(g2)+(1,0.3)$) (dot) {};
      \node[oshelll, syntax, right=.5 of dot] (f0) {$f$};
      \node[oshelll, syntax, left=.5 of g1,inner sep=1pt] (g0) {$g$};
      \draw (f0) -- (dot);
      \draw (g0) -- (g1);
      \draw (g1) to[out=0, in=120] (dot);
      \draw (g2) to[out=0, in=-120] (dot);
      \draw (g2) -- +(-3,0);
      \draw (dot) to[out=-45, in=180] +(1,-1.3) -- +(3.5,0);
      \draw (f0.east) to[out=0, in=0] +(0,2) -- ($(dot-|g1)+(0,2)$);
      \draw ($(dot-|g1)+(0,2)$) -- +(-3,0);
    \end{tikzpicture}
    \hspace{1.5cm}
    \begin{tikzpicture}[inner WD]
      \node[link] (dot) {};
      \node[oshelld, syntax, above=.5 of dot] (f0) {$f$};
      \node[oshelld, syntax, below=.5 of dot,inner sep=1pt] (g0) {$g$};
      \draw (f0) -- (dot);
      \draw (g0) -- (dot);
      \draw (dot) -- +(-3,0);
      \draw (dot) -- +(3,0);
      \draw (f0) to[out=90, in=0] +(-1,1.5) -- +(-2,0);
    \end{tikzpicture}
    \hspace{1.5cm}
    \begin{tikzpicture}[inner WD]
      \coordinate (g1);
      \coordinate[below=2 of g1] (g2);
      \node[link] at ($(g1)!.5!(g2)+(1,0)$) (dot) {};
      \node[oshellr, syntax, left=.5 of g1] (f0) {$f$};
      \node[oshellu, syntax, above=1 of dot,inner sep=1pt] (g0) {$g$};
      \draw (f0.west) -- +(-1,0);
      \draw (f0.east) -- (g1);
      \draw (g1) to[out=0, in=120] (dot);
      \draw (g2) to[out=0, in=-120] (dot);
      \draw (dot) -- +(2,0);
      \draw (g0) -- (dot);
      \draw (g2) -- (g2-|f0.west) -- +(-1,0);
    \end{tikzpicture}
  \end{center}
\end{remark}

\begin{remark}\label{rem.suppress}
  In the coming sections, where it causes no ambiguity, we will gradually suppress infer-able details in our wiring diagrams, as well as choose topological layouts for wiring diagrams that emphasise easy reading and understanding, rather than strictly following our procedure for laying them out as per \cref{notation.picturing_ww}.
\end{remark}

\begin{definition}[Bare wiring diagram]
  A wiring diagram is a \define{bare wiring diagram} if $k=0$; that is, if it contains no kites.
\end{definition}

\begin{remark} \label{rem.bare_wd_represents_itself}
  In the case of a bare wiring diagram, $\omega= \omega^\ast \colon \Gamma_1 \otimes \dots \otimes \Gamma_s \to \Gamma_{\text{out}}$.
\end{remark}

\begin{proposition}\label{prop.barebare}
  Let $ \big(s,k,\{\Gamma_{i}\}_{i\in\ord s},\, \{(\kappa_i,f_i)\}_{i \in \ord k},\, \omega^\ast\big) $ be a wiring diagram in a po-category $\bigsqcup_{J} \ww$ that is a coproduct of copies of $\ww$. Then
  \[
    \big(s,0,\{\Gamma_{i}\}_{i\in\ord s},\, \varnothing,\, (\Gamma_1 \otimes \dots \otimes \Gamma_s \otimes \name{f_1} \otimes \dots \otimes \name{f_k}\big) \cp \omega^\ast)\big)
  \]
  is a bare wiring diagram representing the same morphism, where $\name{f}\colon I\to\dom(f)\otimes\cod(f)$ is the name of $f$, see \cref{def.name_unfolding}.
\end{proposition}

\begin{proof}
  This follows from \cref{rem.bare_wd_represents_itself} and the fact that all morphisms in $\bigsqcup_{J} \ww$ are in the image of the supply functor $s^\sqcup$.
\end{proof}

Although wiring diagrams in their general form will play a critical role in our upcoming explorations of regular calculi and regular logic, a corollary of our main theorem (\cref{cor.vague_eqv}) is that bare wiring diagrams suffice in general for our purposes.

\chapter{Regular calculi}\label{chap.regular_calculi}

Graphical regular logic is the graphical representation of the logical calculus carried by a \emph{regular calculus}: a structure understanding contexts, predicates, and the regular fragment of logic thereupon. In this section we will introduce the framework of \emph{right ajax po-functors} which will allow us to compactly capture all of the desired properties of regular calculi. We will then define regular calculi in terms of these right ajax po-functors and assemble regular calculi into a suitable $2$-category.

\section{Right ajax po-functors and right adjoint monoids}

\begin{definition}[Right ajax po-functor]\label{def.ajax}
  Let $\ccat{C}$ and $\ccat{D}$ be symmetric monoidal po-categories. A \define{right adjoint-lax} or \define{right ajax} po-functor $F\colon \ccat{C} \to \ccat{D}$ is a lax symmetric monoidal po-functor for which the laxators are right adjoints.

  We denote the laxators by $\rho$ and their left adjoints by $\lambda$:
  \begin{equation}\label{eqn.ajax}
    \adjr{I}{\rho}{\lambda}{F(I)}
    \qqand
    \adjr{F(c)\otimes F(c')}{\rho_{c,c'}}{\lambda_{c,c'}}{F(c\otimes c')}.
  \end{equation}
\end{definition}

\begin{example}[Cartesian monoidal po-categories]\label{ex.terminal_ajax}
  If $\ccat{D}$ is a cartesian monoidal po-category then every symmetric monoidal po-category $\ccat{C}$ has a canonical right ajax po-functor $\ccat{C}\to\ccat{D}$, viz., the constant po-functor at the terminal object of $\ccat{D}$. This po-functor is easily seen to be terminal in the po-category of right ajax po-functors $\ccat{C}\to\ccat{D}$ and monoidal $2$-natural transformations. Once we have proven \cref{prop.adjoint_monoids} below we will see that this is a special case of the composition of po-functors $\cc\to1\to\dd$ and the fact that $1$ is an \emph{right adjoint monoid} in $\ccat D$. Although we do not prove it here, for a class of symmetric monoidal po-categories $\ccat{C}$ we shall see in the companion that the initial right ajax po-functor $\ccat{C}\to\pposet$ is given by $\ccat{C}(I,-)$.
\end{example}

\begin{example}[Represented po-functors]\label{ex.ajax_defer}
  As we shall see later in \cref{sec.regular_calc,sec.prd} there are many important classes of examples of right ajax po-functors. One such class, established in \cref{prop.prerel_ajax}, is the represented po-functor $\rr(I,-)\colon\rr\to\pposet$ for any \emph{prerelational} po-category $\rr$. Examples of prerelational po-categories include the po-category $\fladj\cat{R}$ of left adjoints in a regular category $\rr$---see \cref{thm.carboni_equivalence} for details.
\end{example}

\begin{warning}
  If $F$ is a right ajax po-functor, the left adjoints $\lambda_{c,c'}$ of the laxators \emph{do not} in general equip $F$ with the structure of an oplax monoidal functor. \emph{A priori}, the naturality squares for the laxators $\rho$ are only lax-naturality squares for their left adjoints $\lambda$.
\end{warning}

Let us record two immediate but nevertheless useful consequences of the definition of right ajax po-functors.
\begin{lemma}\label{lemma.ajax}\
  \begin{enumerate}
  \item Every strong monoidal functor between monoidal po-categories is right ajax.
  \item The composite of right ajax po-functors is again right ajax.
  \end{enumerate}
\end{lemma}

To understand right ajax po-functors and their importance in the story of regular logic, it is useful to introduce the notion of \emph{right adjoint monoid}. Let us write $1$ for the terminal monoidal po-category.

\begin{proposition}\label{prop.adjoint_monoids}
  Let $(\ccat{C},I,\otimes)$ be a symmetric monoidal po-category. There is a bijection between:
  \begin{enumerate}
  \item The set of right ajax po-functors $1\to\ccat{C}$,
  \item The set of commutative monoid objects $(c,\mu,\eta)$ such that $\mu$ and $\eta$ are right adjoints,
  \item The set of cocommutative comonoid objects $(c,\delta,\epsilon)$ such that $\delta$ and $\epsilon$ are left adjoints.
  \end{enumerate}
\end{proposition}
\begin{proof}
  $(1)\Leftrightarrow(2)$: The set $\lax(1,\ccat{C})$ of symmetric lax monoidal po-functors $1\to\ccat{C}$ may be seen to be in bijection with the set of commutative monoid objects $(c,\mu,\eta)$ in $\ccat{C}$. Indeed, $\eta$ and unit $\mu$ come from the 0-ary and $2$-ary laxators respectively: $\eta=\rho$ and $\mu=\rho_{1,1}$. Hence the added condition that $\eta$ and $\mu$ have left adjoints is precisely the right ajax condition.

  \noindent
  \hphantom{\textit{Proof.} }
  $(2)\Leftrightarrow(3)$: This bijection is implemented by taking adjoints. That is, given an object $c\in\ccat{C}$ and two adjunctions
  \begin{equation}\label{eqn.adjmon}
    \adjr{I}{\eta}{\epsilon}{c}
    \qqand
    \adjr{c\otimes c}{\mu}{\delta}{c},
  \end{equation}
  it may be verified that $\mu$ and $\eta$ satisfy the commutative monoid laws if and only if $\delta$ and $\epsilon$ satisfy the cocommutative comonoid laws. \qedhere
\end{proof}
To summarise, if $(c,\rho,\lambda)\colon 1\to\ccat{C}$ is a right ajax po-functor then the corresponding monoid and comonoid structures on $c$ are given by
\begin{equation}\label{eqn.monoid_comonoid_ajax}
  \eta=\rho\qquad
  \mu=\rho_{1,1}
  \qqand
  \epsilon=\lambda\qquad
  \delta=\lambda_{1,1}
\end{equation}

\Cref{prop.adjoint_monoids} motivates the following definition.

\begin{definition}[Right adjoint monoid]
  Let $(\ccat{C},I,\otimes)$ be a monoidal po-category. A \define{right adjoint monoid} in $\ccat{C}$ is a commutative monoid object $(c,\mu,\eta)$ in $\ccat{C}$ such that $\mu$ and $\eta$ are right adjoints.
\end{definition}

Right adjoint monoids are like \emph{internal $\wedge$-semilattices}; see \cite[Chapter 5]{schalk1994algebras} and references therein. To get a feel for why this might be so, it helps to first recall a po-categorical analogue of a classical lemma of \cite{fox1976coalgebras}.

\begin{lemma}\label{lemma.comonoids_unique}
  Let $\ccat{C}$ be a monoidal po-category. If the monoidal structure is cartesian (given by finite products in the underlying 1-category) then every object has a unique comonoid structure, and it is cocommutative.
\end{lemma}
\begin{proof}
  Since the unit object is terminal, the maps $c\times\epsilon$ and $\epsilon\times c$ are forced to be the projections $c\times c\to c$, so $\delta$ is forced to be the diagonal. Commutativity follows by universal property arguments.
\end{proof}

\begin{example}[Right adjoint monoids in $\pposet$ are $\wedge$-semilattices]\label{prop.adjmon_msl}
  A poset $P\in\pposet$ is an adjoint monoid iff it is a meet-semilattice, in which case $\eta=\true$ and $\mu=\wedge$. To see this we can use \cref{lemma.comonoids_unique}. This states that any poset $P$ has a unique comonoid structure given by the terminal and diagonal maps $\epsilon\colon P\to 1$ and $\delta\colon P\to P\times P$. Thus $P$ is an adjoint monoid iff these maps have right adjoints as in \eqref{eqn.adjmon}, which holds iff $\eta$ is a top element and $\mu$ is a meet.

  Note however that a map of right adjoint monoids in $\pposet$, defined to be a monoidal $2$-natural transformation (\cref{def.symm_mon_po}) between the corresponding right ajax po-functors $1\to\pposet$ (\cref{prop.adjoint_monoids}), does \emph{not} preserve meets---instead it is merely a monotone map between posets which happen to be $\wedge$-semilattices. This is a deliberate choice, and will feature in our coming definition of morphism of regular calculus.
\end{example}

\begin{example}[Right adjoint monoids in a relational po-category are objects]\label{prop.adjmon_reg}
  In later sections we will have cause to consider \emph{relational} and \emph{prerelational} po-categories. By \cref{thm.carboni_equivalence}, every relational po-category is equivalent to the po-category of left adjoints $\fladj\cat R$ in a regular category $\cat R$. In particular, $\fladj\cat R$ is cartesian monoidal and so \cref{lemma.comonoids_unique} above applies to relational po-categories: each object has a unique cocommutative comonoid structure. By \cref{prop.adjoint_monoids} $(2)\Leftrightarrow(3)$ we see that each object in fact must have a unique commutative monoid structure.
\end{example}

It will be useful to know that, just as lax monoidal functors map monoids to monoids, right ajax po-functors map right adjoint monoids to right adjoint monoids.

\begin{proposition}\label{prop.ajax_pres_adjmon}
  Right ajax po-functors send right adjoint monoids to right adjoint monoids.
\end{proposition}
\begin{proof}
  The composite of right ajax po-functors is again right ajax, so the result follows from \cref{prop.adjoint_monoids}.
\end{proof}

Although this result may seem anodyne at present, we will consistently leverage this in our development of graphical regular calculus in \cref{chap.graphical_reglog} as it lends our graphical calculus its ``regular'' aspect.

\section{Regular calculi as indexed right adjoint monoids}\label{sec.regular_calc}

Now that we have understood the notion of right ajax po-functor, we move to define one of the central notions in this paper.

\begin{definition}[Regular calculus]\label{def.reg_calc}
  A \define{regular calculus} P consists of a pair $(\cont{P},P)$ where $(\cont{P},I,\otimes)$ is a symmetric monoidal po-category supplying $\ww$ and $P\colon\cont{P}\to\pposet$ is a right ajax po-functor, whose laxators we denote by $\true$ and $\boxplus$:
  \begin{equation}\label{eqn.ajax_reg_calc}
    \adjr{1}{\true}{\lambda_{I}}{P(I)}
    \qqand
    \adjr{P(\Gamma_1)\times P(\Gamma_2)} {\boxplus_{\Gamma_{1},\Gamma_{2}}}{\lambda_{\Gamma_{1},\Gamma_{2}}} {P(\Gamma_1\otimes \Gamma_2)}.
  \end{equation}
  We call $\cont{P}$ the po-category of \define{contexts}, and $P$ the \define{predicates} po-functor. If $\Gamma\in\ob\cont P$ is a context, and $\theta,\theta'\in P(\Gamma)$ are predicates in context $\Gamma$ such that $\theta\leq \theta'$, then we say that \define{$\theta$ entails $\theta'$}, and write $\theta\vdash \theta'$. Finally, a regular calculus is termed \define{bare} if $\cont{P}=\bigsqcup_{J}\ww$ is a coproduct\footnote{see \cref{warn.coprod}} of copies of $\ww$.
\end{definition}

As a first example let us cast $\wedge$-semilattices as regular calculi.

\begin{example}[$\wedge$-semilattices are regular calculi]\label{ex.reg_calc_meet_sl}
  Observe that the terminal po-category $1$ supplies $\ww$. By \cref{prop.adjoint_monoids,prop.adjmon_msl} a right ajax po-functor $1\to\pposet$ is the same as a $\wedge$-semilattice. Hence a regular calculus whose po-category of contexts is terminal is a $\wedge$-semillatice.

  Should we fix a $\wedge$-semilattice $L$ viewed as a right ajax po-functor $L\colon1\to\pposet$, we may extend this idea to endow any po-category $\cc$ which supplies $\ww$ with the structure of a regular calculus by considering the composite po-functor $\cc\To{!}1\To{L}\pposet$.
\end{example}

This phenomenon, that regular calculi induce $\wedge$-semilattice structures on their posets of predicates, is not special to factoring through the terminal po-category $1$. In the following sense it must occurs point-wise for any regular calculus.

\begin{proposition}\label{prop.meet_sl}
  If $\rc{P}$ is a regular calculus then the poset $P(\Gamma)$ has the structure of a $\wedge$-semilattice for each $\Gamma\in\cont P$.
\end{proposition}

\begin{proof} Recall that by supplying $\ww$, $\cont P$ has a supply of chosen right adjoint monoid structures $(\eta_{\Gamma},\, \mu_{\Gamma})$ for each object $\Gamma\in\cont P$. The result then follows from fact that ajax po-functors send right adjoint monoids to right adjoint monoids by \cref{prop.ajax_pres_adjmon}, as well as \cref{prop.adjmon_msl}.
  Explicitly, \eqref{eqn.ajax} and \eqref{eqn.ww_adjunctions} give rise to the following adjunctions.
  \[
    \begin{tikzpicture}[baseline=(1.base)]
      \node(1)[]{\ensuremath{1}};
      \node(2)[right= of 1]{\ensuremath{P(I)}};
      \node(3)[right= of 2]{\ensuremath{P(\Gamma)}};
      \draw[a]($(1.east)+(0,5pt)$)to node(left)[la,above]{\ensuremath{\true}}($(2.west)+(0,5pt)$);
      \draw[a]($(2.west)-(0,5pt)$)to node(right)[la,below]{\ensuremath{\lambda_{I}}}($(1.east)-(0,5pt)$);
      \draw[a]($(2.east)+(0,5pt)$)to node(left)[la,above]{\ensuremath{P(\eta_{\Gamma})}}($(3.west)+(0,5pt)$);
      \draw[a]($(3.west)-(0,5pt)$)to node(right)[la,below]{\ensuremath{P(\epsilon_{\Gamma})}}($(2.east)-(0,5pt)$);
      \path (1) -- node[midway,rotate=90]{$\dashv$} (2);
      \path (2) -- node[midway,rotate=90]{$\dashv$} (3);
    \end{tikzpicture}\qand
    \begin{tikzpicture}[baseline=(1.base)]
      \node(1)[]{\ensuremath{P(\Gamma)\times P(\Gamma)}};
      \node(2)[right= of 1]{\ensuremath{P(\Gamma\otimes\Gamma)}};
      \node(3)[right= of 2]{\ensuremath{P(\Gamma)}};
      \draw[a]($(1.east)+(0,5pt)$)to node(left)[la,above]{\ensuremath{\boxplus_{\Gamma,\Gamma}}}($(2.west)+(0,5pt)$);
      \draw[a]($(2.west)-(0,5pt)$)to node(right)[la,below]{\ensuremath{\lambda_{\Gamma,\Gamma}}}($(1.east)-(0,5pt)$);
      \draw[a]($(2.east)+(0,5pt)$)to node(left)[la,above]{\ensuremath{P(\mu_{\Gamma})}}($(3.west)+(0,5pt)$);
      \draw[a]($(3.west)-(0,5pt)$)to node(right)[la,below]{\ensuremath{P(\delta_{\Gamma})}}($(2.east)-(0,5pt)$);
      \path (1) -- node[midway,rotate=90]{$\dashv$} (2);
      \path (2) -- node[midway,rotate=90]{$\dashv$} (3);
    \end{tikzpicture}
    \qedhere
  \]
\end{proof}

This result motivates the following notation.

\begin{notation}[The $\wedge$-semilattice on each object in a regular calculus]
  In a regular calculus $\rc P$, for each context $\Gamma\in\ob\cont P$ we denote the $\wedge$-semilattice structure on $P(\Gamma)$ given by \cref{prop.meet_sl} by:
  \begin{equation}\label{eqn.meet_sl}
    \adjr{1}{\true_{\Gamma}}{!}{P(\Gamma)}
    \qqand
    \adjr{P(\Gamma)\times P(\Gamma)}{\wedge_{\Gamma}}{\Delta_{\Gamma}}{P(\Gamma)}.
  \end{equation}
  In keeping with the conventions thus far, for $n\in\nn$ we will write $\wedge^{n}_{\Gamma}$ for the composite map $\boxplus^{n}_{\Gamma}\cp P(\mu_{\Gamma}^{n})\colon P(\Gamma)^{\times n}\to P(\Gamma)$. In the case of $n=0$ we see that $\wedge_{\Gamma}^{0}=\true_{\Gamma}$, and we will freely confuse this map $\true_{\Gamma}$ with the top element $\true_{\Gamma}(*)$. Where it will cause no ambiguity we will omit the label $\Gamma$ on the maps $\true$, $\wedge$, $\wedge^{n}$, $\boxplus$, and so forth.
\end{notation}

Further examples of regular calculi abound. In the next section we will meet our first large class of regular calculi, those arising syntactically from ``regular theories''. After that, \cref{chap.relational_pocats} is devoted to constructing another large class, those arising semantically from regular categories.

\section{The \texorpdfstring{$2$}{2}-category \texorpdfstring{$\mathcal{R}\Cat{gCalc}$}{RgCalc} of regular calculi}

The central result of the companion paper compares the $2$-category theory of regular calculi with that of regular categories. In order to effect such a comparison we must first define suitable notions of $1$- and $2$-morphisms of regular calculi.

\begin{definition}[Morphism of regular calculi]\label{def.reg_calc_morphisms}
  Let $(\cont P, P)$ and $(\cont{P'}, P')$ be regular calculi. A \define{morphism of regular calculi} from $(\cont P, P)$ to $(\cont{P'}, P')$ is a pair $(F,F^\sharp)$ where $(F,\varphi)\colon\cont P\to\cont{P'}$ is a strong monoidal po-functor preserving the supply and $F^\sharp$ is a monoidal $2$-natural transformation (\cref{def.symm_mon_po}) as follows.
  \begin{diagram*}
    \node(1)[]{$\cont {P\phantom{'}}$};
    \node(2)[below=0.50cm of 1]{$\cont{P'}$};
    \node(3)[right=1.00cm of 2]{$\pposet$};
    \draw[a](1)to node(F)[la,left]{$F$}(2);
    \draw[a,rc](1) to (3|-1) to (3);
    \draw[a](2)to node(P')[la,below]{$P'$}(3);
    \node[above,la] at (P'|-1) {$P$};
    \cell[right,la]{(P'|-1)}{(P'.north)}{$F\smash{^{\sharp}}$};
  \end{diagram*}
  A \define{$2$-morphism of regular calculi} from $(F,F^{\sharp})$ to $(G,G^{\sharp})$ is the data of a monoidal left adjoint oplax-natural transformation (\cref{def.symm_mon_po,def.ladj_lax_nt}) $\alpha\colon F\tto G$ which satisfies the property $F^{\sharp}\cp P'\alpha\leq G^{\sharp}$ as explicated below.
  \begin{equation}\label{eqn.two_morphism_modification}
    \begin{tikzpicture}[baseline=(F.base)]
      \node(1)[]{$\cont P$};
      \node(2)[below=1cm of 1]{$\cont{P'}$};
      \node(4)[right=3cm of 2]{$\pposet$};
      \draw[a](2)to node(Pp)[la,below]{$P'$}(4);
      \node(P)[la,yshift=-4pt] at (Pp|-1.north east){$P$};
      \draw[a,rc](1) -- (1-|4) -- (4);
      \draw[a,bend right](1)to node(F)[la,left]{$F$}(2);
      \draw[a,bend left](1)to node(G)[la,right]{$G$}(2);
      \cell[la,above,inner sep=5pt]{(F)}{(G)}{$\alpha$};
      \draw[n,bend right,shorten <=6pt,shorten >=6pt](P)to node(Fs)[]{}(Pp);
      \draw[n,bend left,shorten <=4pt, shorten >=1pt](P)to node(Gs){}(Pp);
      \node(Fs)[la,left] at (Fs|-F) {$F\smash{^{\sharp}}$};
      \node(Gs)[la,right] at (Gs|-G) {$G\smash{^{\sharp}}$};
      \path (Fs.east) -- (Gs.west) node[midway,la]{$\leq$};
    \end{tikzpicture}
    \quad\parbox{3.5cm}{i.e., for each $\Gamma\in\cont P$,\\an inequality in $\pposet\big(P(\Gamma),P'G(\Gamma)\big)$:}\quad
    \begin{tikzpicture}[baseline=(Fs.base)]
      \node(1)[]{$P(\Gamma)$};
      \node(2)[below=0.75cm of 1]{$P'F(\Gamma)$};
      \node(3)[right=0.75cm of 2]{$P'G(\Gamma)$};
      \draw[a](1)to node(Fs)[la,left]{$F^{\sharp}_{\Gamma}$}(2);

      \draw[a,rc] (1) -- (1-|3) -- (3);
      \draw[a](2)to node(P')[la,below]{$P'\alpha_{\Gamma}$}(3);
      \node[la,above](Gs) at (P'|-1) {$G^{\sharp}_{\Gamma}$};
      \path (P') to node[la,sloped,pos=0.45]{$\leq$}(Gs);
    \end{tikzpicture}
  \end{equation}
\end{definition}

\begin{remark}
  The reader would be justified in wondering whether these particular choices of $1$- and $2$-morphisms preserves all of the relevant structure at hand. We provide some discussion on this topic in \cref{sec.graphical_morphisms}, and for greater detail still direct the reader to the companion paper \cite{grl2} in which we show that these morphisms support the construction of a pseudo-reflection of regular categories into regular calculi.
\end{remark}

It is a straightforward if lengthy task to verify that morphisms of regular calculi obey strict composition laws, and moreover that $2$-morphisms may be composed along shared morphism boundaries as well as whiskered with morphisms in a functorial manner. With these facts, we are justified in making the following definition.

\begin{definition}[$\mathcal{R}\Cat{gCalc}$]\label{def.rgcalc_2cat}
  The \define{$2$-category of regular calculi}, denoted $\mathcal{R}\Cat{gCalc}$, has as objects regular calculi, and as $1$- and $2$-morphisms the $1$- and $2$-morphisms of regular calculi of \cref{def.reg_calc_morphisms}.
\end{definition}

\begin{example}\label{ex.rgcalc_subcat}
  Recall from \cref{ex.reg_calc_meet_sl} that a regular calculus $\rc P$ with a terminal context po-category $\cont P=1$ is a $\wedge$-semilattice. It is furthermore the case that the full subcategory of $\rrgcalc$ spanned by the regular calculi with terminal context po-category is terminal is isomorphic to the full subcategory of $\pposet$ spanned by the $\wedge$-semilattices. Note in particular that the morphisms here are arbitrary monotone maps, and not just the $\wedge$-preserving ones.
\end{example}

\section{The bare regular calculus of a regular theory}\label{sec.regular_theories}

Now that we have defined regular calculi, we turn our attention to the first interesting class of examples, viz., the \emph{regular theories}. Indeed the main purpose of regular calculi is to capture syntactic regular logic of relations: as a data structure to support the understanding of regular logic in graphical terms. To illustrate this role, we now explore how regular calculi can be constructed from presentations of regular theories, and hence how they are a tool for describing regular theories in categorical terms.

To begin let us briefly recall the definition of a regular theory. For further detail we direct the curious reader to \cite[Definition D.1.1.1]{Johnstone:2002a}, for example.

\begin{definition}[Regular theory]
  A \define{regular theory} $((\Sigma\text{-sort},\Sigma\text{-rel}),\mathbb T)$ is the data of
  \begin{enumerate}
  \item a set $\Sigma\text{-sort}$ of sorts
  \item a set $\Sigma\text{-rel}$ of relation symbols, each of which is paired with a set of sorts
  \item a set $\mathbb T$ of axioms
  \end{enumerate}
  From $\Sigma\text{-sort}$ we define \emph{contexts}: formal lists $\vec X$ of sorts $X_{i}\in\Sigma\text{-sort}$ where repetition is allowed. Using this, and the first two sets above, we inductively define the collection of formulae over formal variables associated to \emph{suitable} contexts, by using the relation symbols and the regular fragment of logic $(\true,\exists,=,\wedge)$. For instance, if $\varphi(\vec x)$ is a formula in variables $\langle x_{1} : X_1,\ x_2 : X_{2}\rangle$ in the suitable context $\vec X=\langle X_{1},\ X_{2}\rangle$, and $\psi(y)$ is a formula in variable $y:Y$, then $\exists_{y}[\varphi(\vec x)\wedge\psi(y)]$ is a formula in the context $\vec X$ once more. This construction depends on a notion of {suitability}: a context $\vec X$ is suitable for a formula $\varphi(\vec x)$ if the variables $x_{i}\in\vec x$ may be drawn appropriately from the sorts $X_{i}\in\vec X_{i}$. Finally, the set of axioms $\mathbb T$ is a collection of sequents of the form $\varphi\vdash\psi$ which we take in addition to those sequents provided by regular fragment of logic.
\end{definition}

\begin{remark}
  Note that we have defined a regular theory to appear without ``function symbols''. In the companion \cite{grl2} we sketch a proof that this is an anodyne assumption inasmuch as the categories of models are concerned. The essence of this replacement is that function symbols may instead appear as relation symbols along with axioms asserting their functionality, and that models are sufficiently structured so as to prove that the internal relations thus formed are in fact the graphs of morphisms.
\end{remark}

Now that we have recalled the notion of regular theory we are prepared to give a construction assigning to each regular theory $(\Sigma,\mathbb T)$ a bare regular calculus $(\cc_{\mathbb T},F_{\mathbb T})$.

\begin{theorem}\label{con.rgcalc_rgtheory}
  Given a regular theory $(\Sigma,\mathbb T)$, we may define a regular calculus
  \[(\ccat C_{\mathbb T},F_{\mathbb T}) \quad\text{where}\quad \ccat C_{\mathbb T}\coloneqq\displaystyle{\bigsqcup_{\Sigma\text{-sort}}}\ww\ ,\] and $F_{\mathbb T}$ sends each object, viewed as a context, to the poset of equivalence classes of regular formulas in that context quotiented by $\alpha$-equivalence.
\end{theorem}
\begin{proof}
  Let us begin by more precisely stating the assignment of $F_{\mathbb T}$ on objects.

  An object of $\ccat C_{\mathbb T}$ is a list $\langle(\sigma_{i},n_{i})\rangle$ of pairs of sorts $\sigma_{i}\in\Sigma\text{-sort}$ and natural numbers $n_{i}\in\mathbb N$. From this list, let us form the context given by concatenating each $\sigma_{i}$ repeated $n_{i}$ many times, and write this as $\sigma_{1}^{n_{1}},\ldots,\sigma_{k}^{n_{k}}$. If $n_{i}=0$ then $\sigma_{i}$ does not appear in the context. We thus set $F_{\mathbb T}(\left< (\sigma_{i},n_{i})\right>)$ to be the set of all formulae $\varphi$ over $\Sigma$ for which the context formed by $\sigma_{1}^{n_{1}},\ldots,\sigma_{k}^{n_{k}}$ is suitable. We quotient this set by $\alpha$-equivalence, and impose the order associated to provability in the sense of \cite{Johnstone:2002a}. Finally, we perform the poset quotient of what is \emph{a priori} a pre-order.

  Next we must give the data of the functorial action $F_{\mathbb T}\colon\ccat C_{\mathbb T}\to \pposet$. An arbitrary morphism in $\ccat C_{\mathbb T}$ is decomposable as a list of morphisms $\left< f_{i}\right>$, and may be further decomposed as a composite of morphisms which are the identity in all but one index. Thus, to give the functorial action it suffices to specify the action of $F_{\mathbb T}$ on a morphism of a single index, and ensure that functoriality holds. Moreover, observe that the formulae which are suitable for a given context are precisely the same as the formulae which are suitable for any permutation of that context, so without loss of generality we may decompose an arbitrary morphism in $\ccat C_{\mathbb T}$ as a finite sequence of composites which are the identity in all but their last index. We may yet further simplify this task by recalling \eqref{eqn.generating_wires}, that in each index, $\ww$ is generated by four morphisms---$\eta$, $\epsilon$, $\mu$, $\delta$---and some relations governing their composites. Let us therefore address each one of these in turn.

  For a fixed sort $\sigma$, let us define $F_{\mathbb T}(\id+\ldots+\id+\eta_{\sigma})$ to be the function which weakens a formula $\varphi(\vec x)$ in context $\sigma_{1}^{n_{1}},\ldots,\sigma_{k}^{n_{k}}$ to a formula in context $\sigma_{1}^{n_{1}},\ldots,\sigma_{k}^{n_{k}},\sigma$ by adding a free variable $y:\sigma$. Then let us set $F_{\mathbb T}(\id+\ldots+\id+\epsilon_{\sigma})$ to be the assignment of formulae $\varphi(\vec x,y)$ in context $\sigma_{1}^{n_{1}},\ldots,\sigma_{k}^{n_{k}},\sigma$ to formulae $\exists_{y}\varphi(\vec x,y)$ in context $\sigma_{1}^{n_{1}},\ldots,\sigma_{k}^{n_{k}}$. Next, we set $F_{\mathbb T}(\id+\ldots+\id+\epsilon_{\mu})$ as the mapping taking a  formula $\varphi(\vec x, y,y')$ in context $\sigma_{1}^{n_{1}},\ldots,\sigma_{k}^{n_{k}},\sigma^{2}$ where $y,y':\sigma$ to the formula $\mu(\varphi)(\vec x,y)\coloneqq\varphi(\vec x,y,y)$ in context $\sigma_{1}^{n_{1}},\ldots,\sigma_{k}^{n_{k}},\sigma$. Finally, define $F_{\mathbb T}(\id+\ldots+\id+\delta_{\mu})$ to take a formula $\varphi(\vec x,y)$ in context $\sigma_{1}^{n_{1}},\ldots,\sigma_{k}^{n_{k}},\sigma$ to the formula $\delta(\varphi)(\vec x,y,y')\coloneqq [\varphi(\vec x,y)\wedge y=y']$ in context $\sigma_{1}^{n_{1}},\ldots,\sigma_{k}^{n_{k}},\sigma^{2}$. It is straightforward that all these assignments respect $\alpha$-equivalence, and respect the order of provability and the quotient thereby.

  Checking that $F_{\mathbb T}$ is po-functorial amounts to checking that the relations exhibited in \eqref{eqn.equations_wires} hold. For instance, in $\ww$ we have $\delta\cp\mu=\id$ and $\mu\cp\delta\leq\id_{\ord 2}$. Under $F_{\mathbb T}$ these become $\mu(\delta(\varphi))(\vec x,y)=\delta(\varphi)(\vec x,y,y)=[\varphi(\vec x,y)\wedge y=y]=\varphi(\vec x,y)$ and $\delta(\mu(\varphi))(\vec x,y,y')=[\mu(\varphi)(\vec x,y)\wedge y=y']=[\varphi(\vec x,y,y)\wedge y=y']\vdash \varphi(\vec x,y,y')$.

  Finally it remains to give the ajax structure of $F_{\mathbb T}$. First let us define laxator $1\to F_{\mathbb T}(\left<\right>)$ to pick out the formula $\top$ in the empty context, and this is evidently a right adjoint to $!\colon F_{\mathbb T}(\left<\right>)\to 1$ for $\varphi\vdash \top$. Next let us define the laxator \[\boxplus\colon F_{\mathbb T}(\left<(\sigma_{i},n_{i})\right>)\times F_{\mathbb T}(\left<(\sigma'_{i},n'_{i})\right>)\to F_{\mathbb T}(\left<(\sigma_{i},n_{i})\right>+\left<(\sigma'_{i},n'_{i})\right>)\] to be the function sending pairs of formulae $\varphi(\vec x),\psi(\vec y)$ to the formula $\varphi(\vec x)\wedge \psi(\vec y)$ where by construction the contexts for $\vec x$ and $\vec y$ are disjoint in the codomain. In this sense the laxators behave as a form of ``exterior conjunction''. We claim that this map is a right adjoint in $\pposet$. To see this, consider the opposed map $\lambda$ which sends a formula $\gamma(\vec x,\vec y)$ to the pair $\left<\exists_{\vec y}\gamma(\vec x,\vec y),\exists_{\vec x}\gamma(\vec x,\vec y) \right>$. First observe that \[\gamma(\vec x,\vec y)\vdash \exists_{\vec y}\gamma(\vec x,\vec y) \wedge \exists_{\vec x}\gamma(\vec x,\vec y) = \boxplus(\lambda (\gamma(\vec x,\vec y)))\ .\] In the other direction $\lambda(\boxplus(\left<\varphi(\vec x),\varphi(\vec y)\right>))$ has first component $\exists_{y}(\varphi(\vec x)\wedge\psi(\vec y))$ but this entails in particular $\varphi(\vec x)$, and similarly so for the second component. Thus $\boxplus$ is a right adjoint.

  To conclude the construction it remains to check that $\top$ and $\boxplus$ obey the symmetric lax-monoidal coherence conditions and are $2$-natural. The former conditions reduce to the associativity, commutativity and unitality of conjunction, while the latter may be checked on each generator and relation.
\end{proof}

Let us turn to some examples of this construction, beginning with the simplest. In \cref{sec.graph_terms_regul} later we shall find pleasing graphical renditions of some of the formulae we discuss below.

\begin{example}[The theory of an inhabited sort]\label{ex.rgth_inhabited}
  The theory of an inhabited sort is a regular theory. In this regular theory, $\Sigma\text{-sort}=\{S\}$ the unique sort, and $\Sigma\text{-rel}=\emptyset$ is empty. There is precisely one axiom in this theory, viz., $\mathbb T=\{\vdash\exists_{s}\top\}$---the sort $S$ is inhabited.

  The regular calculus $(\cc_{\mathbb T},F_{\mathbb T})$ has $\cc_{\mathbb T}=\ww$, and $F_{\mathbb T}(\ord n)$ is the set of regular formulae in $n$-many variables $s_{i}$. We wish to draw attention to the formula $[\exists_{s}\top]\in F_{\mathbb T}(\ord 0)$ which is asserted by our axiom. Note that this formula is obtained from $\top\in F_{\mathbb T}(\ord 0)$ by the functorial action $F_{\mathbb T}(\eta\cp\epsilon:\ord 0\to\ord0)(\top)$ using the unique morphism $\eta\cp\epsilon\ne\id_{\ord 0}\in\ww(\ord0,\ord 0)$. That is, without this extra morphism $\ord 0\to\ord1\from\ord0$ of \eqref{eqn.extra_law} we would not be able to speak of inhabitedness of a sort in a bare regular calculus.
\end{example}

It is also of interest to consider a regular theory whose corresponding regular calculus has a genuine non-trivial inequality.

\begin{example}[The theory of a pre-order]\label{ex.rgth_preorder}
  The theory of a pre-order is a regular theory. In this regular theory, $\Sigma\text{-sort}=\{X\}$ the unique sort, and $\Sigma\text{-rel}=\{P\}$ the pre-order relation symbol. The axiom set of this theory is small enough to be stated completely: $\mathbb T=\{x\vdash xPx,\, \exists_{y}[xPy\wedge yPz]\vdash xPz\}$, reflexivity and transitivity.

  The regular calculus $(\cc_{\mathbb T},F_{\mathbb T})$ has $\cc_{\mathbb{T}}=\ww$ and $F_{\mathbb T}(\ord n)$ is the set of regular formulae in $n$-many variables $x_{i}$ and the relation symbol $P$, ordered by provability on $\mathbb T$. In particular, observe that $F_{\mathbb T}(\ord 2)$ contains the sub-poset $\{[x=x'] \leq [xPx']\leq\top\}$ which displays two genuine inequalities imposed by the axioms of the theory.
\end{example}

Some more algebraic examples include the theories of monoids and categories.

\begin{example}[The theory of a monoid]\label{ex.rgth_monoid}
  The theory of a monoid is a regular theory. Here $\Sigma\text{-sort}=\{M\}$, $\Sigma\text{-rel}=\{\star,e\}$ where the $\star$ is a ternary relation symbol and $e$ is a unary relation symbol, both on $M$. Our axioms include the usual axioms of a monoid, say $m\star e=m'\vdash m=m'$, but phrased relationally instead, $\star(m,e(m'),m'')\vdash m=m'$. We also include axioms to ensure that what are \emph{a priori} the relations of $\star$ and $e$ are in fact functions, for example, $\vdash\exists_{m}e(m)$ and $e(m)\wedge e(m')\vdash m=m'$.

  The regular calculus $(\cc_{\mathbb T},F_{\mathbb T})$ associated to the theory $(\Sigma,\mathbb T)$ of monoids therefore has as its category of contexts $C_{\mathbb T}=\ww$, and $F_{\mathbb T}(\ord n)$ is the collection of formulae in $n$-many variables of this theory. For example, $F(\ord 2)$ has the formula for expressing commutativity: $\exists_{m,m'}[\star(x,y,m)\wedge\star(y,x,m')\wedge m=m']\in F(\ord 2)$.
\end{example}

\begin{example}[The theory of a category]\label{ex.rgth_cat}
  The theory of a category is also a regular theory. Here $\Sigma\text{-sort}=\{O,M\}$, $\Sigma\text{-rel}=\{\dom,\cod,\id,\comp\}$ where $\id,\dom,\cod$ are all binary relations on $O,M$, and $\comp$ is a binary relation on $M,M$.

  The axioms of the regular theory of categories include once more all the usual axioms, say $\vdash \dom\id(x)=x$, but phrased relationally instead, $\id(x,f)\wedge\dom(f,y)\vdash x=y$. We also include axioms to ensure that $\dom,\cod,\id$ are functions, as well as axioms to ensure that $\comp$ is a function whenever it is defined.

  The regular calculus $(\cc_{\mathbb T},F_{\mathbb T})$ associated to the theory $(\Sigma,\mathbb T)$ of categories therefore has as its category of contexts $C_{\mathbb T}=\ww\sqcup\ww$, and $F_{\mathbb T}([\ord n, \ord m])$ is the collection of formulae in $n$ object variables and $m$ morphism variables. For example, $F([\ord 0, \ord 2])$ contains the formula for isomorphisms (where $f,g:M$ are the isomorphism pair):
  \begin{align*}
    \exists_{h,k:M}\exists_{x,x',y,y':O}\big[&
                                              [\dom(f,x)\wedge\cod(f,y)\wedge\dom(g,y')\wedge\dom(g,x')\wedge(x=x')\wedge(y=y')]\\
                                            &\wedge [\comp(f,g,h)\wedge\comp(g,f,k)]\wedge[\id(x,h)\wedge\id(y,k)] \big]
  \end{align*}
\end{example}

Once we pass to the graphical regular logic of \cref{chap.graphical_reglog} we will see that all three example formulas above---the inhabitedness of a sort, commutativity in monoids, and isomorphisms in categories---admit simple and intuitive graphical depictions.

To conclude this section, let us pause to consider the question of models. Classically, one defines what it means for a regular category to be a model of a given regular theory $(\Sigma,\mathbb T)$. One then also defines a ``syntactic regular category'' $\cat{C}^{\mathrm{reg}}_{\mathbb T}$ and the theoretical edifice is considered satisfactory when regular functors $\cat{C}^{\mathrm{reg}}_{\mathbb T}\to\cat{R}$ are equivalent to models in $\cat{R}$.

In order to replicate this process for regular calculi we will need to develop two main tools in the coming section. First we must find a structure that plays the role of regular categories in the story of models; for us these are the ``relational po-categories'' of \cref{chap.relational_pocats} next. We must also provide the construction analogous to the syntactic regular category; for us this is the ``syntactic po-category'' construction of \cref{sec.syn} which takes a regular calculus to a relational po-category. With these tools we will, in the companion \cite{grl2}, prove a satisfactory model theorem---\cref{thm.model_equiv} below---for regular calculi which subsumes that sketched for regular categories above.

\chapter{Relational po-categories}\label{chap.relational_pocats}

The goal of this section is at first to axiomatise those po-categories whose objects, morphisms, and $2$-morphisms arise as the objects, relations, and inclusions of relations in a regular category. We call these po-categories \emph{relational} and introduce this notion and the $2$-category $\rrlpocat$ of such in \cref{sec.relational}. As we shall see there, relational po-categories admit tidy definition in terms of the notions of supply of \cref{sec.supply}, the po-category for wiring $\ww$, and an extra piece of structure due to Freyd and Scedrov \cite{freyd1990categories}. Moreover, as we note in \cref{thm.carboni_equivalence}, the $2$-category of relational po-categories is appropriately equivalent to the $2$-category of regular categories. The fundamental impetus for so characterising regular categories is the construction of the $2$-functor $\fprd$ of \cref{sec.prd}, from relational po-categories to regular calculi. In our main theorem, \cref{thm.main}, we show that this $2$-functor $\fprd$ is $2$-dimensionally fully-faithful and is a $2$-dimensional right adjoint, and thereby affirm that regular calculi and graphical logic may be leveraged as tools for studying relational po-categories and so regular categories.

\section{The \texorpdfstring{$2$}{2}-category \texorpdfstring{$\mathcal{P}\Cat{rlPoCat}$}{PrlPoCat} of prerelational po-categories}\label{sec.prerel}

The first step towards characterising relations in regular categories is the notion of a prerelational category. Whereas relational po-categories play host to a full regular structure, prerelational po-categories support, roughly speaking, only the underlying cartesian structure.

\begin{definition}[Prerelational po-category] \label{def.prerel_pocat}
  A symmetric monoidal po-category $(\rr,I,\otimes)$ is \define{prerelational} if it supplies wiring $\ww$, such that the induced supply of cocommutative comonoids (\cref{ex.wiring_comonoids}) is lax homomorphic.
\end{definition}

\begin{remark}\label{rem.lax_hom_prerel}
  The lax homomorphicity condition requires that for every $f\colon c\to d$ in $\rr$, both $f\cp\epsilon_d\leq\epsilon_c$ and $f\cp\delta_d\leq\delta_c\cp (f\otimes f)$. In graphical notation:
  \begin{equation}\label{eqn.lax_hom_prerel}
    \begin{tikzpicture}[inner WD,baseline=(1.-30)]

      \node(1){};
      \node(f)[syntax, oshellr,right=1.25 of 1] (f) {$f$};
      \node(2)[right=1.25 of f,link]{};
      \draw(1) to node[above]{$c$} (f);
      \draw(f) to node[above]{$d$} (2);
    \end{tikzpicture}\
    \leq\
    \begin{tikzpicture}[inner WD,baseline=(1.-30)]

      \node(1)[]{};
      \node(2)[right=1.25 of 1,link]{};
      \draw(1) to node[above]{$d$} (2);
    \end{tikzpicture}\hspace{0.75cm}\text{and}\hspace{0.75cm}
    \begin{tikzpicture}[inner WD,baseline=(1.-30)]

      \node(1){};
      \node(f)[syntax, oshellr,right=1.25 of 1] (f) {$f$};
      \node(2)[right=1 of f,link]{};
      \coordinate (d1) at ($(2)+(1,1.3)$);
      \coordinate (d2) at ($(2)+(1,-1.3)$);
      \draw(1) to node[above]{$c$} (f);
      \draw(f) to node[above]{$d$} (2);
      \draw(2) to [out=60, in=180] (d1) -- ++(1,0);
      \draw(2) to [out=-60, in=180] (d2) -- ++(1,0);
    \end{tikzpicture}\ \leq \
    \begin{tikzpicture}[inner WD,baseline=(1.-30)]

      \node(1){};
      \node(2)[right=1 of 1,link]{};
      \coordinate (d1) at ($(2)+(1,1.3)$);
      \coordinate (d2) at ($(2)+(1,-1.3)$);
      \draw(1) to node[above]{$c$} (2);
      \draw(2) to [out=60, in=180] (d1) to node[syntax,oshellr,pos=0.3]{$f$} ++(3,0) node[above]{$d$};
      \draw(2) to [out=-60, in=180] (d2) to node[syntax,oshellr,pos=0.3]{$f$} ++(3,0) node[below]{$d$};
    \end{tikzpicture}
  \end{equation}
\end{remark}

As follows from the mate calculus, or indeed through pleasing graphical manipulations, we may re-characterise prerelational po-categories in terms of the induced supply of commutative monoids instead.

\begin{lemma} \label{lemma.lax_comons_oplax_mons}
  Suppose $\cc$ supplies wiring $\ww$. Then the induced supply of comonoids is lax homomorphic iff the supply of monoids is oplax homomorphic.
\end{lemma}
\begin{proof}
  Note that given $l$ left adjoint to $r$, we always have that $a \le (l \cp b)$ iff $(r \cp a) \le b$ and that $(a \cp l)\le b$ iff $a \le (b \cp r)$. The result now follows from \cref{prop.left_adj_in_ww}.
\end{proof}

\begin{example}\label{ex.ww_prerelational}
  The po-category $\ww$ is prerelational. It supplies $\ww$ by \cref{prop.p_supplies_itself}. To see that the supply of comonoids is lax homomorphic, take a cospan $m\To{f}n\From{g}p$. In the case of the diagonal, the necessary composites are computed as pushouts (e.g.\ $n +_m n$), and the inequality from \eqref{eqn.lax_hom_prerel} is indicated by the dotted arrow, in the diagram
  
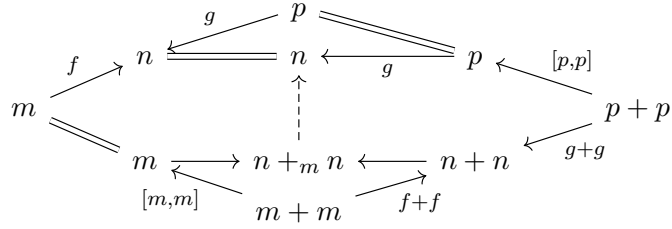
\[
    \begin{tikzcd}[row sep=3pt]
      &&
      p\ar[dl, "g"']\ar[dr, equal]\\&
      n\ar[r, equal]&
      n&
      p\ar[l, "g"]\\
      m\ar[ur, "f"]\ar[dr, equal]&&&&
      p+p\ar[ul, "\copair{p,p}"']\ar[dl, "g+g"]\\&
      m\ar[r]&
      n+_m n\ar[uu, dashed]&
      n+n\ar[l]\\&&
      m+m\ar[ul, "\copair{m,m}"]\ar[ur, "f+f"']
    \end{tikzcd}
  \]
  The case of co-units is similar.
\end{example}

\begin{example}
  The po-category of finite sets and corelations is also prerelational; see \cref{ex.corel}. More generally, given any category $\cat{C}$ with finite limits, the poset reflection of $\sspan(\cat{C})$ is prerelational. As a proof sketch, by Fox's theorem \cite{fox1976coalgebras}, any category $\cat C$ with finite limits homomorphically supplies $\finset\op$ and with some work this may transferred to a suitable supply of $\ww$ in $\sspan(\cat C)$.
\end{example}

\begin{example}
  If $\cat{C}$ is regular, then $\rrel(\cat{C})$ is prerelational, see \cite{fong2019regular}.
\end{example}

A class of examples for which we will have use later is given by the following lemma.

\begin{lemma}\label{lemma.irritating_prerelational_coprod}
  Let $J$ be a set and $\rr$ a prerelational po-category supplying $\ww$. Then with induced supply of $\ww$ in $\bigsqcup_{J}\rr$ of \cref{lemma.irritating_supply_coprod}, the symmetric monoidal po-category $\bigsqcup_{J}\rr$ is prerelational.
\end{lemma}

\begin{proof}
  Direct computation.
\end{proof}

By {\cite[Corollary 6.2]{fong2019regular}}, if $f\colon r\to s$ is a left adjoint in a prerelational po-category $(\rr,I,\otimes)$ then its right adjoint is its transpose $f\tp\colon r\to s$. With this fact we are justified in introducing the following notation.

\begin{notation}[Left adjoints]
  We shall denote \funcrinl{$f$} by \ladjinl{$f$} when $f\colon r\to s$ is known to be a left adjoint. In keeping with \cref{notation.transpose}, we shall denote its right adjoint and transpose $f\tp\colon s\to r$ as \radjinl{$f$}.
\end{notation}

\begin{remark}
  Left adjoints in a prerelational po-category are profitably understood as the true morphisms of the cartesian $1$-category whose structure has been elaborated. As such, we should expect all analogues of cartesian results for left adjoints, such as: two left adjoints $f,g\colon r\to s\otimes t$ are equal iff their ``projections'' under $\epsilon_{s}$ and $\epsilon_{t}$ agree; left adjoints are monoid homomorphisms strictly, for to them the structure is cartesian; a ``natural transformation'' between ``cartesian'' po-functors of prerelational po-categories is automatically ``monoidal'' if its components are left adjoints. Indeed we make the first and last results precise and prove them as \cref{lemma.pres_all_struct_pre,lemma.ladjlaxnat_m} later, while the middle appears as \cite[Corollary 6.2]{fong2019regular}. This serves as additional motivation to distinctly signify left adjoints graphically.
\end{remark}

\begin{definition}
  The \define{$2$-category $\pprlpocat$ of prerelational po-categories} has as objects the prerelational po-categories, as morphisms the strong symmetric monoidal po-functors, and as $2$-morphisms the left adjoint oplax-natural transformations $\alpha\colon F\tto G$; see \cref{def.ladj_lax_nt}.
\end{definition}

Although it is not immediate, the $1$ and $2$-morphisms we have chosen above are indeed appropriate for respecting \emph{all} the structure present in prerelational po-categories. To show this we appeal to a few results of \cite{fong2019regular}.

\begin{lemma}[{\cite[Proposition 6.22]{fong2019regular}}]\label{lemma.pres_all_struct_pre}
  If $\rr$ and $\rr'$ are prerelational po-categories and $F\colon\rr\to\rr'$ is any strong symmetric monoidal po-functor, then $F$ automatically preserves the supply of $\ww$.
\end{lemma}

\begin{lemma}\label{lemma.prerelational_epsilon_joint_mono}
  If $\rr$ is prerelational, and $f,g\colon r\to s\otimes s'$ are two left adjoints, then $f=g$ iff
  \begin{center}
    \begin{tikzpicture}[inner WD,baseline=(f.base)]
      \node(f)[syntax,funcr]{$f$};
      \draw (f.west) -- +(-1,0);
      \draw ($(f.east)+(0,0.3)$) -- +(1,0) node[link]{};
      \draw ($(f.east)-(0,0.3)$) -- +(1,0);
      \node(g)[right=2cm of f,syntax,funcr]{$g$};
      \draw (g.west) -- +(-1,0);
      \draw ($(g.east)+(0,0.3)$) -- +(1,0) node[link]{};
      \draw ($(g.east)-(0,0.3)$) -- +(1,0);
      \node at ($(f)!0.5!(g)$) {$=$};
    \end{tikzpicture}\hspace{1cm} and \hspace{1cm}
    \begin{tikzpicture}[inner WD,baseline=(f.base)]
      \node(f)[syntax,funcr]{$f$};
      \draw (f.west) -- +(-1,0);
      \draw ($(f.east)+(0,0.3)$) -- +(1,0);
      \draw ($(f.east)-(0,0.3)$) -- +(1,0) node[link]{};
      \node(g)[right=2cm of f,syntax,funcr]{$g$};
      \draw (g.west) -- +(-1,0);
      \draw ($(g.east)+(0,0.3)$) -- +(1,0);
      \draw ($(g.east)-(0,0.3)$) -- +(1,0) node[link]{};
      \node at ($(f)!0.5!(g)$) {$=$};
    \end{tikzpicture}\quad .
  \end{center}
\end{lemma}

\begin{proof}
  Certainly $f=g$ implies the given condition. For the converse recall that by \cite[Corollary 6.2]{fong2019regular} left adjoints are comonoid homomorphisms -- the inequalities of \eqref{eqn.lax_hom_prerel} are equalities --, so that we may argue as follows.
  \[
    \begin{tikzpicture}[inner WD]
      \node(f)[syntax,funcr]{$f$};
      \draw (f.west) -- +(-1,0);
      \draw ($(f.east)+(0,0.3)$) -- +(1,0);
      \draw ($(f.east)-(0,0.3)$) -- +(1,0);

      \node(f2)[syntax,funcr,right=1.5cm of f]{$f$};
      \draw (f2.west) -- +(-1,0);
      \draw($(f2.east)+(0,0.3)$) to[out=45,in=180] +(1,1) -- +(1,0) node[link](upper){};
      \draw($(f2.east)-(0,0.3)$) to[out=-45,in=180] +(1,-1) -- +(1,0) node[link](lower){};
      \draw(upper) to[out=45,in=180] +(1,1) -- +(3,0);
      \draw(upper) to[out=-45,in=180] ++(1,-1) -- ++(0.5,0) to[out=0,in=180] ++(2,-1) -- ++(0.5,0) node[link]{};
      \draw(lower) to[out=45,in=180] ++(1,1) -- ++(0.5,0) to[out=0,in=180] ++(2,1) -- ++(0.5,0) node[link]{};
      \draw(lower) to[out=-45,in=180] +(1,-1) -- +(3,0);
      \node[font=\normalsize] at ($(f)!0.5!(f2)$) {$=$};

      \node(dot)[link,right=2.75cm of f2]{};
      \draw (dot.west) -- +(-1,0);
      \draw(dot) to[out=45,in=180] +(1,1) -- +(1,0) node(upf)[syntax,funcr]{$f$};
      \draw(dot) to[out=-45,in=180] +(1,-1) -- +(1,0) node(dnf)[syntax,funcr]{$f$};
      \draw($(upf.east)+(0,0.3)$) -- +(1,0);
      \draw($(upf.east)-(0,0.3)$) -- +(1,0) node[link]{};
      \draw($(dnf.east)+(0,0.3)$) -- +(1,0) node[link]{};
      \draw($(dnf.east)-(0,0.3)$) -- +(1,0);
      \node[font=\normalsize] at ($(f2)!0.75!(dot)$) {$=$};

      \node(dot2)[link,right=2cm of dot]{};
      \draw (dot2.west) -- +(-1,0);
      \draw(dot2) to[out=45,in=180] +(1,1) -- +(1,0) node(upg)[syntax,funcr]{$g$};
      \draw(dot2) to[out=-45,in=180] +(1,-1) -- +(1,0) node(dng)[syntax,funcr]{$g$};
      \draw($(upg.east)+(0,0.3)$) -- +(1,0);
      \draw($(upg.east)-(0,0.3)$) -- +(1,0) node[link]{};
      \draw($(dng.east)+(0,0.3)$) -- +(1,0) node[link]{};
      \draw($(dng.east)-(0,0.3)$) -- +(1,0);
      \node[font=\normalsize] at ($(dot)!0.7!(dot2)$) {$=$};

      \node(g)[funcr,syntax,right=2.25cm of dot2]{$g$};
      \draw (g.west) -- +(-1,0);
      \draw ($(g.east)+(0,0.3)$) -- +(1,0);
      \draw ($(g.east)-(0,0.3)$) -- +(1,0);
      \node[font=\normalsize] at ($(dot2)!0.6!(g)$) {$=$};
    \end{tikzpicture}
    \qedhere
  \]
\end{proof}

\begin{lemma}\label{lemma.ladjlaxnat_m}
  Let $(F,\varphi),(G,\psi)\colon\rr\to\rr'$ be strong symmetric monoidal po-functors between prerelational po-categories, and let $\alpha\colon F\tto G$ be a left adjoint oplax-natural transformation. Then $\alpha$ is a monoidal oplax-natural transformation (\cref{def.symm_mon_po}).
\end{lemma}

\begin{proof}
  We must show that the diagrams
  \begin{center}
    \begin{tikzpicture}[node distance=1.25cm,baseline=(phi.base)]
      \node(1)[]{$I'$};
      \node(2)[below of= 1]{$FI$};
      \node(3)[right of= 2]{$GI$};
      \draw[a](1)to node(phi)[la,left]{$\varphi_{I}$}(2);
      \draw[a](1)to node[la,auto,squeeze]{$\psi_{I}$}(3);
      \draw[a](2)to node[la,below]{$\alpha_{I}$}(3);
    \end{tikzpicture}
    \hspace{1cm} and \hspace{1cm}
    \begin{tikzpicture}[node distance=1.25cm,baseline=(phi.base)]
      \node(1)[]{$Fr\otimes'Fs$};
      \node(2)[right=1.5cm of 1]{$Gr\otimes'Gs$};
      \node(3)[below of=1]{$F(r\otimes s)$};
      \node(4)[below of=2]{$G(r\otimes s)$};
      \draw[a](1)to node[la,above]{$\alpha_{r}\otimes\alpha_{s}$}(2);
      \draw[a](3)to node[la,below]{$\alpha_{r\otimes s}$}(4);
      \draw[a](1)to node(phi)[la,left]{$\varphi_{r,s}$}(3);
      \draw[a](2)to node[la,right]{$\psi_{r,s}$}(4);
    \end{tikzpicture}
  \end{center}
  are strictly commutative for all objects $r,s\in\ob\rr$. Observe that all the strongators are isomorphisms and so are, in particular, left adjoints. Thus these are diagrams of left adjoints and by \cite[Proposition 6.5]{fong2019regular} it suffices to show in the first case that $\varphi_{I}\cp \alpha_{I}\leq \psi_{I}$, equivalently $\alpha_{I}\cp\psi_{I}\inv\leq\varphi_{I}\inv$. In the second case we will use \cref{lemma.prerelational_epsilon_joint_mono} to show $\varphi_{r,s}\cp\alpha_{r\otimes s}\cp \psi\inv=\alpha_{r}\otimes \alpha_{s}$.

  In the first case, by \cref{lemma.pres_all_struct_pre} above, we have the equalities $\varphi_{I}\inv=\epsilon'_{FI}\colon FI\to I$ and $\psi_{I}\inv=\epsilon'_{GI}\colon GI\to I$ as both $(F,\varphi)$ and $(G,\psi)$ preserve the supply of $\ww$. As $\rr'$ is prerelational, $\alpha_{I}\colon FI\to GI$ must be lax homomorphic and so we have \[\alpha_{I}\cp\psi_{I}\inv=\alpha_{I}\cp\epsilon'_{GI}\leq\epsilon'_{FI}=\varphi_{I}\inv\ ,\] as desired.

  To show the equality of the left adjoints $\varphi_{r,s}\cp\alpha_{r\otimes s}\cp \psi\inv$ and $\alpha_{r}\otimes \alpha_{s}$ let us begin by writing $\pi'_{Gr}\coloneqq(\epsilon'_{Gr}\otimes\id_{Gs})\cp \lambda'_{Gs}\colon Gr\otimes'Gs\to Gs$ and likewise for $\pi'_{Gs}$,$\pi'_{Fr}$, and $\pi'_{Fs}$. \Cref{lemma.prerelational_epsilon_joint_mono} shows that it is enough to prove that
  \[
    \varphi_{r,s}\cp\alpha_{r\otimes s}\cp \psi\inv\cp\pi'_{Gr} =\alpha_{r}\otimes \alpha_{s}\cp \pi'_{Gr}\qand
    \varphi_{r,s}\cp\alpha_{r\otimes s}\cp \psi\inv\cp\pi'_{Gs} =\alpha_{r}\otimes \alpha_{s}\cp \pi'_{Gs}\ .
  \]
  The proofs that these equalities hold involve large diagrams which make use of the commutativity of the triangle established above, \cite[Proposition 6.5]{fong2019regular} to see that the oplax-naturality squares for $\alpha$ on left adjoints are actually strict, and \Cref{lemma.pres_all_struct_pre} for equalities like $G(\epsilon_{r})\cp\psi_{I}\inv=\epsilon'_{Gr}$. These diagrams are not especially illuminating and so we have not included them here. Nevertheless, with the aforementioned techniques, it is possible to show that both terms of the above-left claimed equality reduce to $\pi'_{Fs}\cp\alpha_{s}$, while for the above-right claimed equality the reduced form is $\pi'_{Fr}\cp\alpha_{r}$.
\end{proof}

\section{The \texorpdfstring{$2$}{2}-category \texorpdfstring{$\mathcal{R}\Cat{lPoCat}$}{RlPoCat} of relational po-categories}\label{sec.relational}

To make a prerelational po-category relational, we need tabulations. The following definition is due to Freyd and Scedrov \cite{freyd1990categories}.

\begin{definition}[Tabulation] \label{def.tabulation}
  Suppose $\cc$ supplies $\ww$ and let $f\colon r \to s$ be a morphism in $\cc$. A \define{tabulation $(f_R,f_L)$ of $f$} is a factorisation $r\xrightarrow{f_R}\tab f\xrightarrow{f_L}s$ of $f$ where
  \begin{enumerate}[label=(\roman*)]
  \item $f_R\colon r\to \tab f$ is a right adjoint in $\cc$;
  \item $f_L\colon \tab f\to s$ is a left adjoint in $\cc$; and
  \item $\hat{f}\cp\hat{f}\tp=\id_{\tab f}$, where $\hat{f}\coloneqq\delta_{\tab f}\cp(f_L\otimes f_R\tp)$; in pictures
    \begin{equation}\label{eqn.tabulation}
      \begin{tikzpicture}[baseline=(P1)]
        \node (P1) {
          \begin{tikzpicture}[inner WD, baseline=(dot1)]
            \node[funcr, syntax, minimum size=3ex] (fl) {$f_L$};
            \node[funcl, syntax, right=of fl, minimum size=3ex] (fl*) {$f_L$};
            \node[funcr, syntax, below=of fl, minimum size=3ex] (fr!) {$f_R\tp$};
            \node[funcl, syntax, minimum size=3ex] at (fl*|-fr!) (fr) {$f_R\tp$};
            \draw (fl.east) -- (fl*.west);
            \draw (fr!.east) -- (fr.west);
            \node[link] at ($(fl.west)!.5!(fr!.west)+(-1,0)$) (dot1) {};
            \node[link] at ($(fl*.east)!.5!(fr.east)+(1,0)$) (dot2) {};
            \draw (fl.west) to[out=180, in=60] (dot1);
            \draw (fr!.west) to[out=180, in=-60] (dot1);
            \draw (dot1) to node[above, font=\scriptsize] {$\tab f$} +(-2,0);
            \draw (fl*.east) to[out=0, in=120] (dot2);
            \draw (fr.east) to[out=0, in=-120] (dot2);
            \draw (dot2) to node[above, font=\scriptsize] {$\tab f$} +(2,0);
          \end{tikzpicture}
        };
        \node at ($(P1.east)+(.5,0)$) {$=$};
        \draw ($(P1.east)+(1,0)$) to node[above, font=\scriptsize] {$\tab f$} +(1,0);
      \end{tikzpicture}
    \end{equation}
  \end{enumerate}
\end{definition}

\begin{definition}[Relational po-category]\label{def.rl2cat}
  A \define{relational po-category} is a prerelational po-category $\rr$ in which additionally every morphism has a a chosen tabulation.

  The \define{$2$-category $\rrlpocat$ of relational po-categories} is the $2$-full sub-$2$-category $\pprlpocat$ whose objects are relational po-categories. That is, it has as objects relational po-categories, as morphisms strong symmetric monoidal po-functors, and as $2$-morphisms the left adjoint oplax-natural transformations.
\end{definition}

By \cref{lemma.pres_all_struct_pre} we thus know that strong symmetric monoidal po-functors between relational po-categories preserve the supply, but in fact more is true.

\begin{lemma}[{\cite[Proposition 6.22]{fong2019regular}}]\label{lemma.pres_all_struct}
  If $\rr$ and $\rr'$ are relational po-categories and $F\colon\rr\to\rr'$ is any strong symmetric monoidal po-functor, then $F$ automatically preserves the supply of $\ww$ and tabulations.
\end{lemma}

\begin{remark}
  A prerelational po-category is exactly what Carboni and Walters called a \emph{`bicategory of relations'} (quotation marks are an explicit part of the their terminology), and a relational po-category is exactly what they called a \emph{functionally complete `bicategory of relations'}. There are a few, ultimately immaterial differences in the definition; see \cite[Section 8.1]{fong2019regular} for details.
\end{remark}

\begin{example}\label{ex.ww_not_relational}
  As seen in \cref{ex.ww_prerelational}, the po-category $\ww$ is prerelational. However, it is not relational: the cospan $\ord 0\to \ord 1\from \ord 0$ in $\ww$ does not have a tabulation. On the other hand, the po-category of finite sets and corelations is relational.
\end{example}

Every regular category has a po-category of relations, and it is relational; this mapping extends to a $2$-functor $\rrel\colon\rrgcat\to\rrlpocat$ sending a regular category $\rr$ to the po-category with the same objects and with hom-posets given by
\begin{equation}\label{eqn.def_rel}
  \rrel(\rr)(r_1,r_2)\coloneqq\sub_{\rr}(r_1\times r_2).
\end{equation}
In the other direction, the category of left adjoints in any relational po-category is regular; this also extends to a $2$-functor $\fladj\colon\rrlpocat\to\rrgcat$. It turns out these functors form an equivalence of $2$-categories. Indeed, this is the Carboni-Walters idea, although they did not explicitly give the full $2$-categorical account. However, this equivalence was proven as the main theorem in \cite[Theorem 7.3]{fong2019regular}.

\begin{theorem}\label{thm.carboni_equivalence}
  The $2$-functors $\rrel\colon\rrgcat\leftrightarrows\rrlpocat\cocolon\fladj$ form an equivalence of $2$-categories. Their underlying 1-functors moreover form an equivalence between the underlying 1-categories.
\end{theorem}

\begin{example}
  Under the equivalence from \cref{thm.carboni_equivalence}, the relational po-category of finite sets and equivalence relations corresponds to the regular category $\finset\op$. This is \emph{almost} the free regular category on one object, but not quite: it is the free regular category in which every object $x$ is inhabited (the unique map $x\to 1$ is a regular epi).
\end{example}

\begin{remark}
  In fact, more is true of $\fladj$. As a consequence of the main result \cref{thm.main} and the work of companion paper \cite{grl2}, in \cref{cor.ladj_birep} we record a result of the companion: $\fladj$ is ``bi-represented'' by the relational po-category, $\fsyn\fprd\ww$, the syntactic category of the regular calculus of predicates in $\ww$. That is, understanding the regular category of left adjoints of a relational po-category is equivalent to mapping into it from $\fsyn\fprd\ww$.
\end{remark}

\section{The predicates \texorpdfstring{$2$}{2}-functor \texorpdfstring{$\mathbb{P}\Cat{rd}\colon\mathcal{P}\Cat{rlPoCat}\to\mathcal{R}\Cat{gCalc}$}{Prd}}\label{sec.prd}

We are now ready to address the second large source of regular calculi: in this section we will construct a $2$-functorial assignment $\rrgcat\to\rrgcalc$ of regular categories to regular calculi.

To begin, let us recall that \cref{thm.carboni_equivalence} gives a $2$-equivalence between the $2$-category of regular categories $\mathcal{R}\Cat{gCat}$ and that of relational po-categories $\mathcal{R}\Cat{lPoCat}$ (\cref{def.rl2cat}). With this in mind, our goal may be achieved by instead constructing a $2$-functor $\fprd\colon\rrlpocat\to\mathcal{R}\Cat{gCalc}$. However, it turns out that our construction of this to-be $2$-functor does not depend on the presence of tabulators, and so we further factor our promised assignment as follows:

\[
  \mathcal{R}\Cat{gCat}\overset{\ref{thm.carboni_equivalence}}\eqv\rrlpocat\rightarrowtail\pprlpocat\To{\fprd}\rrgcalc\ .
\]

With that let us turn our attention to the construction of the to-be $2$-functor $\fprd$.
Our first step will be to associate a right ajax po-functor to every prerelational po-category $\rr$. In the coming proposition we will show that $\rr(I,-)$ is right ajax. While we will have no use for the following, it is more generally true that $\rr(r,-)$ is right ajax for all objects $r\in\ob\rr$.

\begin{proposition}\label{prop.prerel_ajax}
  Let $(\rr,\otimes,I)$ be a prerelational po-category. The representable po-functor $\rr(I,-)\colon\rr\to\pposet$ has a lax monoidal structure whose laxators are right adjoints, and thus supports the structure of a regular calculus on $\rr$.
\end{proposition}

\begin{proof}
  For every $r\in\rr$, we have a poset $\rr(I,r)$, and this is clearly functorial in $r$. Consider the proposed adjunctions below:
  \begin{equation}\label{eqn.prd_right_ajax}
    \adjr{1}{\id_{I}}{!}{\rr(I,I)}
    \qqand
    \adjr{\rr(I,r)\times \rr(I,r')}{\otimes\cp (\delta_{I})^{*}}{\pi}{\rr(I,r\otimes r')}.
  \end{equation}
  The poset map labelled $\id_{I}$ sends 1 to the identity; that labelled $!$ is uniquely determined; that labelled $\otimes\cp(\delta_{I})^{*}$ sends $\pair{f,f'}$ to $f\otimes f'$ pre-composed with $\delta_{I}\colon I\to I\otimes I$; and that labelled $\pi$ sends $h\colon I\to r\otimes r'$ to the pair $\pi(h)\coloneqq\pair{h\cp(r\otimes\epsilon_{r'})\cp\rho_{r}, h\cp(\epsilon_r\otimes r')\cp\lambda_{r'}}$. The coming contents of the proof rely on the lax comonoid homomorphism condition in \cref{def.prerel_pocat}.

  To see that the first pair of maps form an adjunction, take any $s\in\rr(I,I)$. By definition of supply \eqref{eqn.supply_commute_tensors} we have $\epsilon_I=\id_I$, so lax comonoidality implies $s=s\cp\epsilon_I\leq\epsilon_I=\id_I$, as desired. Let us pause here to note that the same supply conditions imply the equality of morphisms $\delta_{I}={\lambda_{I}}\inv={\rho_{I}}\inv\colon I\to I\otimes I$.

  To see that the second pair of maps form an adjunction, take $\pair{f,f'}\in\rr(I,r)\times\rr(I,r')$ and $h\in\rr(I,r\otimes r')$ and consider the below arguments, where we have used the lax comonoidality properties of \eqref{eqn.lax_hom_prerel} and the facts $\epsilon_{I}=\id_{I}$ and $\delta_{I}={\lambda_{I}}\inv$ from above.
  \[(\otimes\cp(\delta_{I})^{*}\cp\pi)\left(\left<
        \begin{tikzpicture}[inner WD,baseline=(1.-20)]
          \node(1)[oshellr,syntax]{$f\vphantom{'}$};
          \draw(1.east) -- +(1,0);
        \end{tikzpicture},
        \begin{tikzpicture}[inner WD,baseline=(1.-20)]
          \node(1)[oshellr,syntax]{$f'$};
          \draw(1.east) -- +(1,0);
        \end{tikzpicture}\right>\right)=
    \left<
      \begin{tikzpicture}[inner WD,baseline=(2.north)]
        \node(1)[syntax,oshellr]{$f\vphantom{'}$};
        \node(2)[syntax,oshellr,below= 1 of 1]{$f'$};
        \node(dot)[link,right= 1 of 2]{};
        \draw(1.east) -- +(1,0);
        \draw(2.east) -- (dot);
      \end{tikzpicture}\ \ ,\ \
      \begin{tikzpicture}[inner WD,baseline=(2.north)]
        \node(1)[syntax,oshellr]{$f\vphantom{'}$};
        \node(2)[syntax,oshellr,below= 1 of 1]{$f'$};
        \node(dot)[link,right= 1 of 1]{};
        \draw(2.east) -- +(1,0);
        \draw(1.east) -- (dot);
      \end{tikzpicture}\right>\leq
    \left<
      \begin{tikzpicture}[inner WD,baseline=(1.-20)]
        \node(1)[oshellr,syntax]{$f\vphantom{'}$};
        \draw(1.east) -- +(1,0);
      \end{tikzpicture},
      \begin{tikzpicture}[inner WD,baseline=(1.-20)]
        \node(1)[oshellr,syntax]{$f'$};
        \draw(1.east) -- +(1,0);
      \end{tikzpicture}\right>
  \]

  \[
    \begin{tikzpicture}[inner WD,baseline=(1.-20)]
      \node(1)[oshellr,syntax]{$h$};
      \draw(1) to[out=45,in=180] +(1,1) -- +(1,0);
      \draw(1) to[out=-45,in=180] +(1,-1) -- +(1,0);
    \end{tikzpicture}\quad=\quad
    \begin{tikzpicture}[inner WD,baseline=(1.-20)]
      \node(1)[oshellr,syntax]{$h$};
      \draw(1) to[out=45,in=180] +(2,2) -- +(1,0) node[link](upper){};
      \draw(1) to[out=-45,in=180] +(2,-2) -- +(1,0) node[link](lower){};
      \draw(upper) to[out=45,in=180] +(1,1) -- +(1,0);
      \draw(upper) to[out=-45,in=180] +(1,-1) -- +(1,0) node[link]{};
      \draw(lower) to[out=45,in=180] +(1,1) -- +(1,0) node[link]{};
      \draw(lower) to[out=-45,in=180] +(1,-1) -- +(1,0);
    \end{tikzpicture}\quad=\quad
    \begin{tikzpicture}[inner WD,baseline=(1.-20)]
      \node(1)[oshellr,syntax]{$h$};
      \draw(1) to[out=45,in=180] +(2,2) -- +(1,0) node[link](upper){};
      \draw(1) to[out=-45,in=180] +(2,-2) -- +(1,0) node[link](lower){};
      \draw(upper) to[out=45,in=180] +(1,1) -- +(3,0);
      \draw(upper) to[out=-45,in=180] ++(1,-1) -- ++(0.5,0) to[out=0,in=180] ++(2,-2) -- ++(0.5,0) node[link]{};
      \draw(lower) to[out=45,in=180] ++(1,1) -- ++(0.5,0) to[out=0,in=180] ++(2,2) -- ++(0.5,0) node[link]{};
      \draw(lower) to[out=-45,in=180] +(1,-1) -- +(3,0);
    \end{tikzpicture}\quad\leq\quad
    \begin{tikzpicture}[inner WD,baseline=(c)]
      \node(1)[oshellr,syntax]{$h$};
      \draw(1) to[out=45,in=180] +(1,1) -- +(1,0);
      \draw(1) to[out=-45,in=180] +(1,-1) -- +(1,0) node[link]{};
      \node(2)[oshellr,syntax,below=1.5 of 1]{$h$};
      \draw(2) to[out=45,in=180] +(1,1) -- +(1,0) node[link]{};
      \draw(2) to[out=-45,in=180] +(1,-1) -- +(1,0);
      \coordinate (c) at ($(1)!0.5!(2)-(0,0.25)$);
    \end{tikzpicture}\quad=(\pi\cp\otimes\cp(\delta_{I})^{*})\
    \begin{tikzpicture}[inner WD,baseline=(1.-20)]
      \node(1)[oshellr,syntax]{$h$};
      \draw(1) to[out=45,in=180] +(1,1) -- +(1,0);
      \draw(1) to[out=-45,in=180] +(1,-1) -- +(1,0);
    \end{tikzpicture}
  \]

  It remains to establish that $(\rr(I,-),\id_{I},\otimes\cp(\delta_{I})^{*})$ assembles into a lax monoidal functor. The conditions on $!$ are all automatic by the terminality of $1\in\pposet$, and so we have reduced our claim to the assertion that the following diagram is commutative.
  \[\begin{tikzpicture}
      \node(1)[]{$\rr(I,(r\otimes r')\otimes r'')$};
      \node(2)[right=3 of 1]{$\rr(I,r\otimes(r'\otimes r''))$};
      \node(3)[below of= 1]{$\rr(I,r\otimes r')\times\rr(I,r'')$};
      \node(4)[below of= 2]{$\rr(I,r)\times\rr(I,r'\otimes r'')$};
      \node(5)[below of= 3]{$(\rr(I,r)\times\rr(I,r'))\times\rr(I,r'')$};
      \node(6)[below of= 4]{$\rr(I,r)\times(\rr(I,r')\times\rr(I,r''))$};
      \draw[a](1)to node[la,above]{$\alpha^{\rr}_{*}$}(2);
      \draw[a](1)to node[la,left]{$\pi$}(3);
      \draw[a](2)to node[la,right]{$\pi$}(4);
      \draw[a](3)to node[la,left]{$\pi\times\id$}(5);
      \draw[a](4)to node[la,right]{$\id\times\pi$}(6);
      \draw[a](5)to node[la,below]{$\alpha^{\pposet}$}(6);
    \end{tikzpicture}\]

  The commutativity of this diagram on $f\in\rr(I,(r\otimes r')\otimes r'')$ is equivalently three equalities between the corresponding components in $\rr(I,r)\times(\rr(I,r')\times\rr(I,r''))$. To show all three of these equalities it is advantageous to cast $\delta_{I}={\lambda_{I}}\inv$ which follows from \eqref{eqn.supply_commute_tensors} and the classical fact $\lambda_{I}=\rho_{I}$ in monoidal categories. In this way, the second of these equalities follows only from formal arguments on the axioms of a monoidal category, while the first and third follow from additional equations on $\epsilon$ assured by definition of supply \eqref{eqn.supply_commute_tensors}.
\end{proof}

Thus, if $(\rr,I,\otimes)$ is a prerelational po-category, we define
\begin{equation}\label{eqn.prd_obs}
  \fprd(\rr,I,\otimes)\coloneqq\left(\rr\To{\rr(I,-)}\pposet\right)
  \text{ with}\hspace{0.5cm}
  \begin{gathered}
    \adjr{1}{\id}{!}{\rr(I,I)}\\
    \adjr{\rr(I,r)\times \rr(I,r')}{\otimes\cp (\delta_{I})^{*}}{\pi}{\rr(I,r\otimes r')}
  \end{gathered}
\end{equation}
where the right ajax structure indicated on $\rr(I,-)$ was constructed in \eqref{eqn.prd_right_ajax} above.

Next let $(F,\varphi)\colon\rr\to\rr'$ be a morphism of prerelational po-categories, that is, a strong symmetric po-functor. Recall that by \cref{lemma.pres_all_struct_pre}, $F$ preserves the supply of $\ww$. We define $\fprd(F,\varphi)\coloneqq (F,F^\sharp)$, where the monoidal natural transformation $F^{\sharp}$ below-left is given in components $F^{\sharp}_{r}\colon\rr(I,r)\to\rr'(I',Fr)$ as below-right.
\begin{equation}\label{eqn.prd_mors}
  \fprd(F,\varphi)\coloneqq (F,F^\sharp)\text{ where}\hspace{1cm}
  \begin{gathered}
    \begin{tikzpicture}
      \node(1)[]{$\rr\phantom{'}$};
      \node(2)[below=0.50cm of 1]{$\rr'$};
      \node(3)[right=1.00cm of 2]{$\pposet$};
      \draw[a](1)to node(F)[la,left]{$F$}(2);
      \draw[a,rc](1) to (3|-1) to (3);
      \draw[a](2)to node(P')[la,below]{$\rr'(I,-)$}(3);
      \node[above,la] at (P'|-1) {$\rr(I,-)$};
      \cell[right,la]{(P'|-1)}{(P'.north)}{$F^{\sharp}$};
    \end{tikzpicture}\\
    F^{\sharp}_{r}\coloneqq\rr(I,r)\xrightarrow{F} \rr'(FI,Fr)\xrightarrow{(\varphi_{I})^{*}}\rr'(I',Fr)
  \end{gathered}
\end{equation}
It is straightforward to check that this definition is 1-functorial in $(F,\varphi)$.

Finally, let $\alpha\colon (F,\varphi)\tto (G,\psi)$ be a $2$-morphism of prerelational po-categories, that is, a left adjoint oplax-natural transformation. We wish to provide a $2$-morphism of regular calculi of the form $\fprd(\alpha)\coloneqq(\alpha,F^{\sharp}\leq G^{\sharp})\colon\fprd(F,\varphi)\tto\fprd(G,\psi)$.
To achieve this, let us note that \cref{lemma.ladjlaxnat_m} implies that $\alpha$ is, in fact, a monoidal oplax-natural transformation. As such, it remains to provide the data of a modification $F^{\sharp}\leq G^{\sharp}$ as in \eqref{eqn.two_morphism_modification}.

The required inequality $F^{\sharp}_{r}\cp P'\alpha_{r}\leq G^{\sharp}_{r}$ for each $c\in\ob\rr$ follows from the monoidal nature of $\alpha$ and the oplax-naturality of $\alpha$.
\begin{equation}\label{eqn.prd_twomors}
  \fprd(\alpha)\coloneqq(\alpha,F^{\sharp}\cp\rr'(I',\alpha)\leq G^{\sharp})\text{ where}\hspace{1cm}
  \begin{tikzpicture}[node distance=1.25cm,baseline=(3.base)]
    \node(1)[]{$I'$};
    \node(2)[right of= 1]{$I'$};
    \node(3)[below of= 1]{$FI$};
    \node(4)[below of= 2]{$GI$};
    \node(5)[below of= 3]{$Fr$};
    \node(6)[below of= 4]{$Gr$};
    \draw[d](1)to(2);
    \draw[a](1)to node[la,left]{$\varphi_{I}$}(3);
    \draw[a](2)to node[la,right]{$\psi_{I}$}(4);
    \draw[a](3)to node[la,above]{$\alpha_{I}$}(4);
    \draw[a](3)to node[la,left]{$Ff$}(5);
    \draw[a](4)to node[la,right]{$Gf$}(6);
    \draw[a](5)to node[la,below]{$\alpha_{r}$}(6);
    \path(5) to node[la,sloped]{$\leq$} (4);
    \path(3) to node[la,sloped]{$=$} (2);
  \end{tikzpicture}
\end{equation}
Again, it is straightforward to see that this is $2$-functorial by transitivity.

In summary, we have demonstrated the following.

\begin{proposition}
  The assignments of objects, morphisms and $2$-morphisms given by \eqref{eqn.prd_obs}, \eqref{eqn.prd_mors}, and \eqref{eqn.prd_twomors} assemble into a $2$-functor $\fprd\colon\pprlpocat\to\rrgcalc$.
\end{proposition}

We will henceforth freely confuse the $2$-functor $\fprd$ with its restriction to relational po-categories \[\rrlpocat\rightarrowtail\pprlpocat\To{\fprd}\rrgcalc\ .\]

\chapter{Graphical regular logic} \label{chap.graphical_reglog}

In this section we finally develop our graphical formalism for regular logic by defining the notion of graphical term, showing how these graphical terms represent predicates in contexts, and explaining how to reason with them. We sketch how the collection of graphical terms, together with our graphical reasoning, allows us to form the ``syntactic po-category'' of a regular calculus in a $2$-functorial fashion. In the companion paper \cite{grl2}, we make extensive use of this graphical regular logic to prove the that regular categories are ``pseudo-reflective'' in regular calculi by means of our syntactic po-category construction.

\section{Graphical terms}

Given a regular calculus $(\cont{P},P)$, graphical terms provide representations of its predicates, i.e.\ the elements in $P(\Gamma)$ for various contexts $\Gamma\in\cont{P}$.

We invite the reader to recall \cref{def.extended_wiring_diagram}, our definition of wiring diagrams in a po-category supplying $\ww$, as well as \cref{notation.wiring_diagrams_supplying_ww}, our graphical notation  therefor, before considering this next definition and its accompanying notation.

\begin{definition}(Graphical term)\label{def.graphical_term}
  Let $(\cont{P},P)$ be a regular calculus. A \define{graphical term} is the data of
  \begin{enumerate}[label=(\roman*)]
  \item a wiring diagram $(s,k,\{\Gamma_{i}\}_{i\in\ord s},\, \{(\kappa_i,f_i)\}_{i \in \ord k},\, \omega^\ast\colon \Gamma_1\otimes \dots \otimes \Gamma_s \otimes \kappa_1 \otimes \dots \otimes \kappa_k \to \Gamma_\out)$ in $\cont P$,
  \item for each $i\in\ord s$, a predicate $\theta_i \in P(\Gamma_i)$ in context $\Gamma_{i}=\bigotimes_{j\in\ord{n_{i}}} \tau_{i}(j)$.
  \end{enumerate}
  We will choose to suppress the details of the wiring diagram and notate such a graphical term by $(\theta_1,\dots,\theta_s;\omega)$, where $\omega$ is the morphism of $\cc$ represented by the wiring diagram, as given by \cref{def.morphism_represented_by_wd}. If $s=0$ then a graphical term $(;\omega)$ is simply a morphism $\omega\colon I\to\Gamma_\out$.

  We say that the graphical term $t = (\theta_1,\dots,\theta_s; \omega)$ \define{represents} the predicate
  \begin{equation}\label{eqn.graph_term_reps}
    \church{ t } \coloneqq P(\omega)(\theta_1\boxplus^{s}\cdots\boxplus^{s}\theta_s) \in P(\Gamma_\out)
  \end{equation}
  where $\boxplus^{s}$ is the $s$-ary laxator $\boxplus^{s}\colon \prod_{i\in\ord s} P(\Gamma_{i})\to P(\bigotimes_{\ord s} \Gamma_{i})$. In particular, if $s=0$ then a graphical term $(;\omega)$ represents the predicate $P(\omega)(\true)$. We extend the equality and implication of the poset $P(\Gamma_{\out})$ to graphical terms $t$, $t'$ via $\church{-}$ in the following sense: we say that \define{$t$ implies $t'$} and that \define{$t$ equals $t'$} when $\church{t}\vdash\church{t'}$ and $\church{t}=\church{t'}$ respectively. In a slight abuse of notation we will write $t\vdash t'$ and $t=t'$ for this implication and equality.\footnote{In this sense graphical terms inherit a pre-order as well as an equivalence relation, relative to which anti-symmetry of the pre-order holds.}
\end{definition}

\begin{example}
  When $s=1$ and $\omega=\id$ is the identity, then $\church{(-;\Gamma)}\colon P(\Gamma)\to P(\Gamma)$ is also the identity. More generally for any $s$, when $\omega=\id$, the map \[\church{(-,\,\ldots,\,-;\textstyle\bigotimes_{i\in\ord s} \Gamma_{i})}\colon\prod_{i\in\ord s} P(\Gamma_i)\to P(\textstyle\bigotimes_{i\in\ord s} \Gamma_i)\] is $\boxplus^{s}$, the $s$-ary laxator. We shall see other special cases in \cref{prop.diagrams_meet}.
\end{example}

\begin{notation}[Graphical terms]
  We draw a graphical term $(\theta_1,\dots,\theta_s; \omega)$ by drawing the morphism $\omega$ as in \cref{notation.wiring_diagrams_supplying_ww} and annotating the $i$\textsuperscript{th} inner IO shell with its corresponding predicate $\theta_i$.  In the case that $\omega$ is the identity morphism, we may simply draw the contexts annotated by their predicates. For instance,
  \begin{diagram*}[.][A]
    \node(A) {
      \begin{tikzpicture}[inner WD, baseline=(rho)]
        \node[pack,minimum size = 3ex] (rho) {$\theta$};
        \draw (rho.180) to[pos=1] node[left] {$\tau(1)$} +(180:2pt);
        \draw (rho.110) to[pos=1] node[above] {$\tau(2)$} +(110:2pt);
        \node at ($(rho.10)+(10:6pt)$) {$\vdots$};
        \draw (rho.-110) to[pos=1] node[below] {$\tau(s)$} +(-110:2pt);
      \end{tikzpicture}
    };
    \node(B)[left=3.5cm of A.west, anchor=east]{$\church{(\theta;\id_{\bigotimes\tau(i)})}$};
    \path(B) -- node[midway]{is represented by} (A);
  \end{diagram*}
\end{notation}

\begin{example}
  Let $\rc P$ be a regular calculus, and let $\theta_{i}\in P(\Gamma_{i})$ for $i\in\{1,2,3\}$ be predicates, let $f_{1}\colon\Gamma_{1}\to \Gamma'_{1}$ and $f_{2}\colon \Gamma_{1}\otimes \Gamma_{2}\to\Gamma'_{2}$ be morphisms of $\cont P$. Let us write $\sigma\colon \Gamma_{1}\otimes\Gamma_{1}\tpow2\otimes\Gamma_{2}\to\Gamma_{1}\tpow2\otimes(\Gamma_{1}\otimes\Gamma_{2})$ for the appropriate symmetry of $\cont P$, then the predicate  \[\left\llbracket\big(\theta_{1},\theta_{2},\theta_{3};(\Gamma_{1}\otimes\delta_{\Gamma_{1}}\otimes \Gamma_{2})\cp\sigma\cp((\mu_{\Gamma_{1}}\cp f_{1})\otimes f_{2})\big)\right\rrbracket\]
  is represented by the graphical term
  \begin{diagram*}[.][][inner WD]
    \node(t1)[pack]{$\theta_{1}$};
    \draw(t1.east) -- +(1, 0) node(l1)[link]{};
    \draw(l1) to[out=-90, in=180] +(1, -1.5) -- +(2, 0) node(f2)[oshellr,syntax,anchor=140]{$f_{2}$};
    \draw(f2.east) -- +(1, 0);
    \node(t2)[pack] at (t1|-f2.220) {$\theta_{2}$};
    \draw(t2.east) -- (t2-|f2.220);
    \draw(l1) to[out=+90, in=180] +(1, +1.5) -- +(1,0) node(l2)[link]{};
    \node(f1)[oshellr,syntax,anchor=west] at (l2-|f2.west) {$f_{1}$};
    \draw(f1.east) -- +(1, 0);
    \draw(l2) -- (f1.west);
    \node(t3)[pack, above of=l2, anchor=south]{$\theta_{3}$};
    \draw(t3.south) -- (l2);
  \end{diagram*}
\end{example}

\begin{example}
  We saw in \cref{ex.reg_calc_meet_sl} that right ajax po-functors $1\to\pposet$ are $\wedge$-semilattices. The corresponding diagrammatic language has no wires, since $1$ comprises only the monoidal unit. The semantics of an arbitrary graphical term $(\theta_1,\,\ldots,\,\theta_s;\id)$ is simply the meet $\theta_1\wedge\cdots\wedge\theta_s$.
\end{example}

\begin{remark}\label{rem.graphical_regular}
  Graphical terms can be used to depict formulae in regular logic. Fix a regular calculus $\rc{P}$, a context $\Gamma\in\cont P$ and suppose $\omega\colon \ord n_1+ \dots + \ord n_s \to \ord n_\out$ is a morphism of $\ww$, that is, by \cref{rem.canonical_decomp}, canonically a cospan of the form
  \[ \ord n_{1}+\dots+\ord n_{s} \xrightarrow{[\omega_{1},\dots,\omega_{s}]}\ord n_{\omega}\xleftarrow{\omega_{\out}}\ord n_{\out}\ .\]
  By the supply of $\ww$ this induces a morphism $\omega\colon\bigotimes\Gamma\tpow{n_{i}}\iso\Gamma^{\otimes(\sum n_{i})}\to\Gamma^{\otimes n_{\out}}$. While we will not dwell on the translation, a graphical term $(\theta_1,\dots,\theta_s; \omega)$ represents the following regular formula in free variables $x_{(\out,1)},\,\ldots,\,x_{(\out,{n}_{\out})}$.
  \[
    \bigexists_{\substack{i \in \{1,\dots,s,\omega\} \\ j \in \ord{n}_i}} x_{(i,j)} \left[
      \bigwedge_{i'\in\{
        1,\dots s\}} \theta_{i'}(x_{(i',1)},\, \ldots,\, x_{(i',n_{i'})}) \quad\wedge\quad \bigwedge_{\substack{i'\in\{1,\,\ldots,\,s,\out\} \\j' \in \ord{n}_{i'}
        }} \big(x_{(i',j')} = x_{(\omega,\omega_{i'}{(j')})}\big)\right]
  \]
  For example, given the supplied morphism $\omega\colon\Gamma\tpow{3}\otimes\Gamma\tpow{3}\otimes\Gamma\tpow{4}\to\Gamma\tpow{6}$ of \cref{ex.wiring_diagram} and predicates $\theta_{1}\in P(\Gamma\tpow{3})$, $\theta_{2}\in P(\Gamma\tpow{2})$, and $\theta_{3}\in P(\Gamma\tpow{4})$, the graphical term
  \begin{diagram*}[][][penetration=0, inner WD,  pack size=9pt, link size=2pt, scale=2, baseline=(out)]
    \node[packs] at (-1.5,-1) (f) {$\theta_{3}$};
    \node[packs] at (0,1.9) (g) {$\theta_{1}$};
    \node[packs] at (1.5,-1) (h) {$\theta_{2}$};
    \node[outer pack, inner sep=34pt] at (0,.2) (out) {};

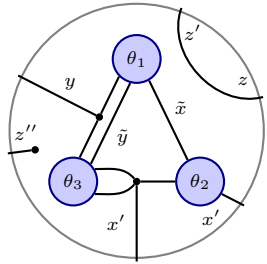
    \node[link] at ($(f)!.5!(h)$) (link1) {};
    \node[link] at (-2.4,-.25) (link2) {};
    \node[link] at ($(f.75)!.5!(g.-135)$) (link3) {};

    \draw[fr_out] (out.270) to node[auto]{$x'$} (link1);
    \draw[fr_out] (out.190) to node[near end, above]{$z''$} (link2);
    \draw[fr_out] (out.155) to node[auto, inner sep=2pt]{$y$} (link3);
    \draw[fr_out] (out.-35) to node[auto]{$x'$} (h.-30);
    \draw[to_out,fr_out] (out.15) to[out=-165,in=-110] node[auto,swap, near start]{$z$} node[auto,swap,near end,inner sep=1pt]{$z'$} (out.70);
    \draw (f.30) to[out=0,in=130] (link1);
    \draw (f.-30) to[out=0,in=-130] (link1);
    \draw (h.180) to (link1);
    \draw (g.-60) to node[auto,inner sep=2pt]{$\tilde x$} (h.120);
    \draw (f.45) to node[auto,swap, inner sep=2pt]{$\tilde y$} (g.-105);
    \draw (f.75) to (link3);
    \draw (g.-135) to (link3);
  \end{diagram*}

  would represent the following formula after simplification.
  \[
    \psi(y,z,z',x,x',z'') = \exists\, \tilde{x},\tilde{y}, \left[\theta_1(\tilde{x},\tilde{y},y) \wedge \theta_2(x',\tilde{x},x) \wedge \theta_3(y,\tilde{y},x',x') \wedge (z=z')\right]
  \]
\end{remark}

Before we address the \emph{calculus} portion of our graphical notation for regular calculi, let us turn our attention to a final class of examples of our graphical notation: bare regular calculi.

\begin{example}\label{ex.graphical_terms_in_bare}
  Recall that a regular calculus $\rc P$ is bare if its category of contexts $\cont P$ is of the form $\cont P=\bigsqcup_{J}\ww$. Given our notation of \cref{sec.depict_ww,sec.wds_for_ww_suppliers} for wiring diagrams in $\ww$ and the result of \cref{prop.barebare}, we see that a graphical term $(\theta_{1,i_{1}},\,\ldots\,\theta_{k,i_{k}};\omega)$ in the regular calculus $(\bigsqcup_{J}\ww,P)$ is precisely a collection of wiring diagrams for $\ww$ each of whose IO shells has been annotated by predicates $\theta_{j,i_{j}}\in P(\ord n_{j,i_{j}})$ where $j,i_{j}\in J$. A somewhat typical example might therefore be
  \begin{diagram*}
    \node(diagram) 
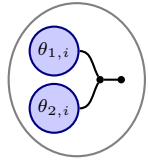
{
      \begin{tikzpicture}[inner WD]
        \node[pack] (theta1) {$\theta_{1,i}$};
        \node[pack, below=.3 of theta1] (theta2) {$\theta_{2,i}$};
        \coordinate (helper) at ($(theta1)!.5!(theta2)$);
        \node[link, right=2 of helper] (dot R) {};
        \draw (theta1.east) to[out=0, in=120] (dot R);
        \draw (theta2.east) to[out=0, in=-120] (dot R);
        \draw (dot R) -- +(1, 0) -- node[link](end){} +(1,0);
        \node[outer pack, fit=(theta1) (theta2) (end)] (outer) {};
      \end{tikzpicture}
    };
    \node(w)[right=2cm of diagram.east, anchor=west] {$\theta_{1,j}$, $\theta_{2,j}\in P(\ord 1_{j})$ for $j\in J$ and $\omega=\mu\cp\epsilon\colon \ord 1_{j}+\ord 1_{j}\to \ord 0_{j}$.};
    \coordinate (m) at ($(diagram.east)!0.5!(w.west)$);
    \node[anchor=base] at (m|-w.base) {where};
  \end{diagram*}
  As there are no labelled morphisms decorating its graphical terms and instead only wires, the regular calculus $(\bigsqcup_{J}\ww,P)$ is in this visual sense considered \emph{bare}.
\end{example}

\section{Reasoning with graphical terms}

Now that we have understood the graphical notation, it is time to attend to the calculus of manipulations it supports. Let $\rc{P}$ be a regular calculus. The following basic rules for reasoning with graphical terms express the $2$-functoriality and monoidal structure of the po-functor $P\colon \cont P\to\pposet$.

\begin{proposition} \label{prop.diagrams_basic}
  Let $(\theta_1,\dots,\theta_s;\omega)$ be a graphical term, where $\theta_i \in P(\Gamma_i)$.
  \begin{propenum}
  \item\label{prop.diagrams_basic_mono} (Monotonicity) Suppose $\theta_i \vdash \theta_i'$ for some $i$. Then
    \[
      \church{(\theta_1,\dots,\theta_i,\dots,\theta_s; \omega)} \vdash \church{(\theta_1,\dots,\theta_i',\dots,\theta_s; \omega)}.
    \]
  \item (Breaking) Suppose $\omega \leq \omega'$ in $\cont P$. Then
    \[
      \church{(\theta_1,\dots,\theta_s; \omega)} \vdash
      \church{(\theta_1,\dots,\theta_s; \omega')}.
    \]
  \item\label{prop.diagrams_basic_nesting} (Nesting) Suppose $\theta_i = \church{(\theta'_1,\dots,\theta'_\ell; \omega')}$ for some $1\leq i\leq s$. Then
    \begin{align*}
      &\church{(\theta_{1}\,,\ldots,\, \theta_{i-1}, \church{(\theta'_1,\dots,\theta'_\ell; \omega')}, \theta_{i+1}, \ldots, \theta_{s};\omega)}  =
        \church{(\theta_1,\dots, \theta_{i}\ldots, \theta_s; \omega)}\\
      =\, &\church{(\theta_1,\dots,\theta_{i-1},\theta'_1, \dots,
            \theta'_\ell,\theta_{i+1},\dots,\theta_s; (\textstyle\bigotimes_{1\leq j< i}\Gamma_{j} \otimes\,  {\omega'} \otimes \textstyle\bigotimes_{i< j\leq s}\Gamma_{j} )\cp \omega)}
    \end{align*}
  \end{propenum}
\end{proposition}
\begin{proof}
  By examining $\church{-}$ of \cref{def.graphical_term} we may reason as below.
  \begin{enumerate}[label=(\roman*)]
  \item This claim follows from the monotonicity of the map $\boxplus^{s}\cp P(\omega)$.
  \item This claim follows from the $2$-functoriality of $P$.
  \item This claim follows from the monoidal structure and $1$-functoriality of $P$. By using the symmetry of $\cont P$, without loss of generality we may assume that $i=s$. In this case, to prove the desired equality it is sufficient to demonstrate the commutativity of the following diagram.
    \begin{diagram*}[][][node distance=2cm]
      \node(1)[]{$\prod\limits_{j=1}^{s-1}P(\Gamma_j) \times \prod\limits_{j=1}^\ell P(\Gamma'_j)$};
      \node(2)[below of= 1]{$\prod\limits_{j=1}^{s-1}P(\Gamma_j) \times P\left(\bigotimes_{j=1}^\ell\Gamma'_j\right)$};
      \node(3)[below of= 2]{$\prod\limits^{s}_{j=1}P(\Gamma_{j})$};
      \node(4)[right=0.75cm of  2]{$P\left(\bigotimes^{s-1}_{j=1}\Gamma_{j}\otimes \bigotimes^{\ell}_{j=1}\Gamma'_{j}\right)$};
      \node(5)[below of= 4]{$P\left(\bigotimes_{j=1}^{s}\Gamma_{j}\right)$};

      \node(6) at ($(4)-(1)+(4)$){$P(\Gamma_{\out})$};
      \draw[a](1)to node[la,left]{$\id\times\boxplus^{\ell}$}(2);
      \draw[a](2)to node[la,left]{$\id\times P(\omega)$}(3);
      \draw[a](1)to node[la,auto]{$\boxplus^{(s-1)+\ell}$}(4);
      \draw[a](2)to node[la,above]{$\boxplus^{s}$}(4);
      \draw[a](4)to node[la,over,inner sep=3pt]{$P(\bigotimes^{s-1}_{j=1}\Gamma_{j}\otimes \omega')$}(5);
      \draw[a](3)to node[la,below]{$\boxplus^{s}$}(5);
      \draw[a](4)to node[la,auto]{$P((\bigotimes^{s-1}_{j=1}\Gamma_{j}\otimes \omega')\cp \omega)$}(6);
      \draw[a](5)to node[la,below]{$P(\omega)$}(6);
    \end{diagram*}
    In the above diagram, the upper triangle commutes by coherence laws for $\boxplus$, the square commutes by naturality of $\boxplus$, and the right hand triangle commutes by functoriality of $P$. \qedhere
  \end{enumerate}
\end{proof}

\begin{example}
  \cref{prop.diagrams_basic} is perhaps more quickly grasped through a graphical example of these facts in action. Suppose we have the entailment
  
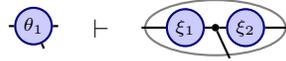
\begin{center}
    \begin{tikzpicture}[unoriented WD, font=\small, pack size=7pt, baseline=(P1.south)]
      \def\angle{-65};
      \node (P1) {\begin{tikzpicture}[inner WD,baseline=(theta.base)]
          \node[pack] (theta) {$\theta_1$};
          \draw (theta.180) -- +(180:2pt);
          \draw (theta.0) -- +(0:2pt);
          \draw (theta.\angle) -- +(\angle:2pt);
        \end{tikzpicture}};
      \node[right=3 of P1] (P2) {
        \begin{tikzpicture}[inner WD,baseline=(xi1.base)]
          \node[pack] (xi1) {$\xi_1$};
          \node[pack, right=1 of xi1] (xi2) {$\xi_2$};
          \node[link] at ($(xi1.east)!.5!(xi2.west)$) (dot) {};
          \node[outer pack, fit=(xi1) (xi2)] (outer) {};
          \draw[fr_out] (outer) -- (xi1.west);
          \draw (xi1.east) -- (dot);
          \draw (dot) -- (xi2);
          \draw[to_out] (xi2) -- (outer);
          \draw[to_out] (dot) -- (outer.\angle);
        \end{tikzpicture}};
      \node at ($(P1.east)!.5!(P2.west)$) {$\vdash$};
    \end{tikzpicture}
  \end{center}
  Then using monotonicity, nesting, and then breaking we may deduce the entailment
  
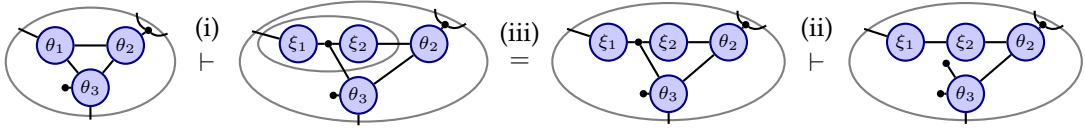
\begin{center}
    \begin{tikzpicture}[unoriented WD, font=\small, baseline=(P1.195)]
      \node (P1) {
        \begin{tikzpicture}[inner WD]
          \node[pack] (theta1) {$\theta_1$};
          \node[pack, right=1.5 of theta1] (theta2) {$\theta_2$};
          \node[pack] at ($(theta1)!.5!(theta2)+(0,-2)$) (theta3) {$\theta_3$};
          \node[outer pack, inner xsep=2pt, inner ysep=1pt, fit=(theta1) (theta2) (theta3.-30)] (outer) {};
          \node[link] at ($(theta2.30)!.5!(outer.30)$) (dot1) {};
          \node[link,  left=.1 of theta3.west] (dot2) {};
          \draw (theta2) -- (dot1);
          \draw[to_out] (dot1) to[bend right] (outer.20);
          \draw[to_out] (dot1) to[bend left] (outer.45);
          \draw (dot2) -- (theta3);
          \draw[to_out] (theta1) -- (outer);
          \draw[to_out] (theta3) -- (outer);
          \draw (theta1) -- (theta3);
          \draw (theta1) -- (theta2);
          \draw (theta2) -- (theta3);
        \end{tikzpicture}
      };
      \node[right=2 of P1] (P2) {
        \begin{tikzpicture}[inner WD]
          \node[pack] (xi1) {$\xi_1$};
          \node[pack, right=1 of xi1] (xi2) {$\xi_2$};
          \node[link] at ($(xi1.east)!.5!(xi2.west)$) (dot) {};
          \draw (xi1.east) -- (dot);
          \draw (dot) -- (xi2);
          \node[outer pack, inner xsep=0, inner ysep=0, fit=(xi1) (xi2)] (outerxi) {};

          \node[pack, right=1.5 of xi2] (theta2) {$\theta_2$};
          \node[pack, below=.6 of xi2] (theta3) {$\theta_3$};
          \node[outer pack, inner xsep=1pt, inner ysep=2pt, fit=(xi1.west) (theta2) (theta3.-10)] (outer) {};
          \node[link] at ($(theta2.30)!.5!(outer.30)$) (dot1) {};
          \node[link,  left=.1 of theta3.west] (dot2) {};
          \draw (theta2) -- (xi2.east);
          \draw (dot) -- (theta3);
          \draw[to_out] (dot1) to[bend right] (outer.20);
          \draw[to_out] (dot1) to[bend left] (outer.35);
          \draw (dot2) -- (theta3);
          \draw[to_out] (xi1) -- (outer);
          \draw[to_out] (theta3.south) -- (theta3|-outer.south);
          \draw (theta2) -- (theta3);
          \draw (theta2) -- (dot1);
        \end{tikzpicture}
      };
      \node[right=2 of P2] (P3) {
        \begin{tikzpicture}[inner WD]
          \node[pack] (xi1) {$\xi_1$};
          \node[pack, right=1 of xi1] (xi2) {$\xi_2$};
          \node[link] at ($(xi1.east)!.5!(xi2.west)$) (dot) {};
          \draw (xi1.east) -- (dot);
          \draw (dot) -- (xi2);

          \node[pack, right=1 of xi2] (theta2) {$\theta_2$};
          \node[pack, below=.6 of xi2] (theta3) {$\theta_3$};
          \node[outer pack, inner xsep=1pt, inner ysep=1pt, fit=(theta1) (theta2) (theta3.-20)] (outer) {};
          \node[link] at ($(theta2.30)!.5!(outer.30)$) (dot1) {};
          \node[link,  left=.1 of theta3.west] (dot2) {};
          \draw (theta2) -- (xi2.east);
          \draw (dot) -- (theta3);
          \draw[to_out] (dot1) to[bend right] (outer.20);
          \draw[to_out] (dot1) to[bend left] (outer.35);
          \draw (dot2) -- (theta3);
          \draw[to_out] (xi1) -- (outer);
          \draw[to_out] (theta3) -- (outer);
          \draw (theta2) -- (theta3);
          \draw (theta2) -- (dot1);
        \end{tikzpicture}
      };
      \node[right=2 of P3] (P4) {
        \begin{tikzpicture}[inner WD]
          \node[pack] (xi1) {$\xi_1$};
          \node[pack, right=1 of xi1] (xi2) {$\xi_2$};
          \node[link] at ($(xi1.east)!.5!(xi2.west)+(.5,-1)$) (dotc) {};

          \draw (xi1.east) -- (xi2.west);
          \node[pack, right=1 of xi2] (theta2) {$\theta_2$};
          \node[pack, below=.6 of xi2] (theta3) {$\theta_3$};
          \node[outer pack, inner xsep=1pt, inner ysep=1pt, fit=(theta1) (theta2) (theta3.-20)] (outer) {};
          \node[link] at ($(theta2.30)!.5!(outer.30)$) (dot1) {};
          \node[link,  left=.1 of theta3.west] (dot2) {};
          \draw (theta2) -- (xi2.east);
          \draw (dotc) -- (theta3);
          \draw[to_out] (dot1) to[bend right] (outer.20);
          \draw[to_out] (dot1) to[bend left] (outer.35);
          \draw (dot2) -- (theta3);
          \draw[to_out] (xi1) -- (outer);
          \draw[to_out] (theta3) -- (outer);
          \draw (theta2) -- (theta3);
          \draw (theta2) -- (dot1);
        \end{tikzpicture}
      };
      \node (imp1) at ($(P1.east)!.5!(P2.west)$) {$\vdash$};
      \node[above=-.5 of imp1] {(i)};
      \node (imp2) at ($(P2.east)!.5!(P3.west)$) {$=$};
      \node[above=-.5 of imp2] {(iii)};
      \node (imp3) at ($(P3.east)!.5!(P4.west)$) {$\vdash$};
      \node[above=-.5 of imp3] {(ii)};
    \end{tikzpicture}
  \end{center}
  We'll see many further examples in \cite{grl2}, where we prove that we can construct a regular category from a regular calculus.
\end{example}

The nesting rule of \cref{prop.diagrams_basic_nesting} has two particularly important cases.

\begin{example}[Wiring diagrams as predicates]\label{ex.wd_as_preds}
  Let $\omega\colon I\to \Gamma$ be a morphism in $\cont P$. Observe that we have the equalities
  \begin{diagram*}[,][6]
    \node(1)[anchor=east]{$\church{(;\omega)}=\big(1$};
    \node(2)[right=2cm of 1]{$P(I)$};
    \node(3)[right=2cm of 2]{$P(\Gamma)\big)$};
    \node(4)[anchor=east] at ($(1.east)-(0,0.75cm)$){$=\big(1$};
    \node(5)[] at ($(2)-(0,0.75cm)$){$P(\Gamma)$};
    \node(6)[anchor=west] at ($(3.west)-(0,0.75cm)$){$P(\Gamma)\big)=\church{(\church{(;w)};\Gamma)}$};
    \draw[a](1)to node[la,above]{$\true$}(2);
    \draw[a](2)to node[la,above]{$P(\omega)$}(3);
    \draw[a](4)to node[la,above]{$\true\cp P(\omega)$}(5);
    \draw[a](5)to node[la,above]{$P(\Gamma)$}(6);
  \end{diagram*}
  so that we are justified in equating the following two graphical terms
  \begin{diagram*}[.][pred]
    \node(wd)
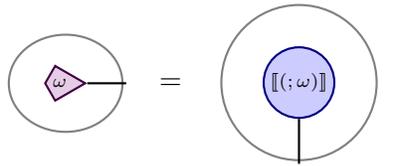
{\begin{tikzpicture}[inner WD]
        \node(w)[oshellr,syntax]{$\omega$};
        \node(outer)[outer pack,fit={(w)},inner sep=7pt]{};
        \draw[to_out](w.east) -- (outer);
      \end{tikzpicture}};
    \node(pred)[right=1cm of wd]{\begin{tikzpicture}[inner WD]
        \node(p)[pack]{$\church{(;\omega)}$};
        \node(outer)[outer pack,fit={(p)},inner sep=7pt]{};
        \draw[to_out](p.south) -- (outer);
      \end{tikzpicture}};
    \node at ($(wd)!0.45!(pred)$) {$=$};
  \end{diagram*}
  Moreover, every regular calculus has a rich stock of such morphisms $I\to \Gamma$. In \cref{lemma.name_unfolding_iso} we exhibited an isomorphism $\cont P(\bigotimes \Gamma_{i},\Gamma_{\out})\iso \cont P(I,\bigotimes\Gamma_{i}\otimes \Gamma_{\out})$ mediated by taking the name $\omega\mapsto \name\omega$ of a morphism. In this way we may view arbitrary wiring diagrams $\omega\colon \bigotimes\Gamma_{i}\to \Gamma_{\out}$ in $\cont P$ as graphical terms $\church{(\church{(;\name\omega)};\bigotimes \Gamma_{i}\otimes \Gamma_{\out})}$ of the above-right form.

  Note, however, that in general the above merely constitutes a map of wiring diagrams into predicates which preserves representation. It is not necessarily the case that all predicates $\theta\in P(\Gamma)$ may be realised as $\church{(\church{(;\omega)};\Gamma)}$ for some wiring diagram $\omega$. Nevertheless, in \cref{sec.prd_graphical_terms} we will see that there is a large class of regular calculi in which wiring diagrams and predicates do coincide.
\end{example}

\begin{example}[Exterior conjunction]
  Let $\Gamma_{1}$ and $\Gamma_{2}$ be contexts, and let $\theta_{1}\in P(\Gamma_{1})$ and $\theta_{2}\in P(\Gamma_{2})$ be predicates. Observe that we have the equalities
  \[ \theta_{1}\boxplus \theta_{2} = \church{(\theta_{1}\boxplus \theta_{2};\Gamma_{1}\otimes \Gamma_{2})} =\church{(\theta_{1},\theta_{2};\Gamma_{1}\otimes \Gamma_{2})}\]
  of elements of $P(\Gamma_{1}\otimes \Gamma_{2})$ so that we are justified in equating, for example, the following two graphical terms.
  \begin{diagram*}
    \node (P1) 
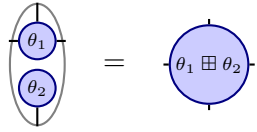
{
      \begin{tikzpicture}[inner WD]
        \node[pack] (theta1) {$\theta_1$};
        \node[pack, below=.4 of theta1] (theta2) {$\theta_2$};
        \node[outer pack, inner ysep=0pt, fit=(theta1) (theta2)] (outer) {};
        \draw (theta1.0) -- (theta1.0-|outer.east);
        \draw (theta1.90) -- (outer.north);
        \draw (theta1.180) -- (theta1.180-|outer.west);
        \draw (theta2.270) -- (outer.south);
      \end{tikzpicture}
    };
    \node[right=1cm of P1] (P2) {
      \begin{tikzpicture}[inner WD, pack size=6pt]
        \node[pack] (rho) {$\theta_1\boxplus\theta_2$};
        \draw (rho.0) -- +(0:2pt);
        \draw (rho.90) -- +(90:2pt);
        \draw (rho.180) -- +(180:2pt);
        \draw (rho.270) -- +(270:2pt);
      \end{tikzpicture}
    };
    \node at ($(P1.east)!.5!(P2.west)$) {$=$};

    \pgfresetboundingbox
    \useasboundingbox (P1.130) rectangle (P2.-40);
  \end{diagram*}
  Under the interpretation of graphical terms as formulae in regular logic suggested by \cref{rem.graphical_regular}, this process of vertical merging of graphical terms corresponds to the logical conjunction of the formulae they represent.
\end{example}

In \cref{prop.meet_sl} we saw how a regular calculus endows each poset $P(\Gamma)$ with the structure of a meet-semilattice. As we will now see, this structure permits an intuitive graphical interpretation. In the following proposition, the graphical terms on right are illustrative examples of the equalities stated on the left.

\begin{proposition} \label{prop.diagrams_meet}
  For all contexts $\Gamma\in\ob\cont P$ and predicates $\theta_{1}$, $\theta_{2}\in P(\Gamma)$, we have
  \begin{propenum}
  \item\label{prop.diagrams_meet_true} (True is removable) $\church{(\true_{\Gamma};\Gamma)} = \church{ (;\eta_\Gamma) }$, \hfill $
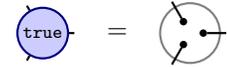
    \begin{tikzpicture}[unoriented WD, font=\small,baseline=(true)]
      \node (P1) {
        \begin{tikzpicture}[inner WD]
          \node[pack] (true) {$\true$};
          \draw (true.0) -- +(0:2pt);
          \draw (true.120) -- +(120:2pt);
          \draw (true.240) -- +(240:2pt);
        \end{tikzpicture}
      };
      \node[right=3 of P1] (P2) {
        \begin{tikzpicture}[inner WD, shorten >=-2pt]
          \coordinate (helper);
          \node[link] (dot0) at ($(helper)+(0:5pt)$) {};
          \node[link] (dot120) at ($(helper)+(120:5pt)$) {};
          \node[link] (dot240) at ($(helper)+(240:5pt)$) {};
          \node[outer pack, surround sep=8pt, fit=(helper)] (outer) {};
          \draw (dot0) -- (outer.0);
          \draw (dot120) -- (outer.120);
          \draw (dot240) -- (outer.240);
        \end{tikzpicture}
      };
      \node at ($(P1.east)!.5!(P2.west)$) {$=$};
    \end{tikzpicture}
    $

  \item (Meets-are-merges)
    $
    \church{ (\theta_1\wedge\theta_2;\Gamma)} = \church{
      (\theta_1,\theta_2;\mu_\Gamma) }.
    $
    \hfill
    $
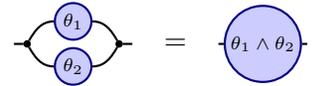
\begin{tikzpicture}[unoriented WD,baseline=(theta)]
      \node (P1) {
        \begin{tikzpicture}[inner WD]
          \node[pack] (theta1) {$\theta_1$};
          \node[pack, below=.3 of theta1] (theta2) {$\theta_2$};
          \coordinate (helper) at ($(theta1)!.5!(theta2)$);
          \node[link, right=2 of helper] (dot R) {};
          \node[link, left=2 of helper] (dot L) {};
          \draw (theta1.east) to[out=0, in=120] (dot R);
          \draw (theta2.east) to[out=0, in=-120] (dot R);
          \draw (theta1.west) to[out=180, in=60] (dot L);
          \draw (theta2.west) to[out=180, in=-60] (dot L);
          \draw (dot R) -- +(5pt, 0);
          \draw (dot L) -- +(-5pt, 0);
        \end{tikzpicture}
      };
      \node[right=3 of P1] (P2) {
        \begin{tikzpicture}[inner WD]
          \node[pack] (theta) {$\theta_1\wedge\theta_2$};
          \draw (theta.180) -- +(180:2pt);
          \draw (theta.0) -- +(0:2pt);
        \end{tikzpicture}
      };
      \node at ($(P1.east)!.5!(P2.west)$) {$=$};
    \end{tikzpicture}
    $
  \end{propenum}
\end{proposition}
\begin{proof}
  These equations are simply the definitions of $\true$
  and meet; see \eqref{eqn.meet_sl} and \eqref{eqn.graph_term_reps}.
\end{proof}

\begin{example}[Discarding]\label{lemma.dotting_off}
  Note that \cref{prop.diagrams_meet_true} and the monotonicity of diagrams
  (\cref{prop.diagrams_basic_mono}) further imply that for all $\theta \in P(\Gamma)$
  we have $\theta \vdash \church{
    (;\eta_\Gamma) }$:
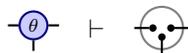
  \[
    \begin{tikzpicture}[unoriented WD, font=\small, baseline=(phi)]
      \node (P1) {
        \begin{tikzpicture}[inner WD, shorten >=-2pt]
          \node[pack] (phi) {$\theta$};
          \draw (phi.0) -- +(0:2pt);
          \draw (phi.180) -- +(180:2pt);
          \draw (phi.270) -- +(270:2pt);
        \end{tikzpicture}
      };
      \node[right=3 of P1] (P2) {
        \begin{tikzpicture}[inner WD, shorten >=-2pt]
          \node[pack, fill=white, white] (phi) {$\theta$};
          \node[outer pack, fit=(phi)] (outer) {};
          \node[link] at ($(outer.180) - (180:5pt)$) (dot180) {};
          \node[link] at ($(outer.0) - (0:5pt)$) (dot0) {};
          \node[link] at ($(outer.270) - (270:5pt)$) (dot270) {};
          \draw (dot0) -- (outer.0);
          \draw (dot180) -- (outer.180);
          \draw (dot270) -- (outer.270);
        \end{tikzpicture}
      };
      \node at ($(P1.east)!.5!(P2.west)$) {$\vdash$};
    \end{tikzpicture}
  \]
\end{example}

\section{Graphical terms in the regular calculus of a regular theory}
\label{sec.graph_terms_regul}

In this section we will revisit \cref{con.rgcalc_rgtheory} which associates to a regular theory $(\Sigma,\mathbb T)$ a regular calculus $(\cc_{\mathbb T},F_{\mathbb T})$. In particular we will illustrate graphically the highlighted formulae of \cref{ex.rgth_inhabited,ex.rgth_preorder,ex.rgth_monoid,ex.rgth_cat}.

\subparagraph{Theory of an inhabited sort} In \cref{ex.rgth_inhabited} we outlined the theory of an inhabited sort. In the associated regular calculus, the graphical terms are simply wiring diagrams whose IO shells are annotated by formulae whose arity is no greater than the number of wires. The single axiom of the theory, $\top\vdash\exists_{s}\top$, is rendered as the below entailment of graphical terms. Note that $\top$ is the blank graphical term in the empty context and we have made use of the equation $\exists_{s}\top=F_{\mathbb T}(\eta;\epsilon)(\top)$ for the consequent.
\begin{diagram*}
  \node(B)[]{
    \begin{tikzpicture}[inner WD]
      \node[outer pack, minimum size=20pt]{};
    \end{tikzpicture}
  };
  \node(A)[right=2cm of B]{
    \begin{tikzpicture}[inner WD]
      \node(1)[link]{};
      \node[outer pack, minimum size=20pt,anchor=center] at (1){};
    \end{tikzpicture}
  };
  \node(At)[below=0.5 of A]{$\exists_{s}\top$};
  \node(Bt)at (At-|B){$\top$};
  \path (B) -- node[midway]{$\leq$} (A);
\end{diagram*}
In fact this inequality is an equality, for we always have $\exists_{s}\top\leq\top$, and so the theory of an inhabited sort allows us to freely erase or introduce dots; compare \eqref{eqn.extra_law} of $\ww$. The outer circles were rendered above to emphasise that these are graphical terms, but from here we will suppress this detail in keeping with \cref{rem.suppress}.

\subparagraph{Theory of a pre-order} In \cref{ex.rgth_preorder} we outlined the theory of a pre-order. Here we wish to draw attention to the two axioms of the theory, reflexivity $x\vdash xPx$ and transitivity $\exists_{y}[xPy\wedge yPz]\vdash xPz$, and their renditions as the below-left and below-right graphical terms in the associated regular calculus, respectively. Note in particular that the former contributes to the regular calculus genuine, strict inequalities between graphical terms.

\begin{diagram*}
  \node(A)[]{
    \begin{tikzpicture}[inner WD]
      \node(1)[minimum size=0, inner sep=0]{};
      \draw(1) -- ++(-2,0) -- ++(2,0);
      \node(2)[pack] at ($(1)+(6,0)$) {$P$};
      \draw(2.west) -- ++(-0.5, 0);
      \draw(2.east) -- ++(+0.5, 0);
      \path(1) -- node(a)[midway]{$\leq$} ($(2)-(1,0)$);
    \end{tikzpicture}
  };
  \node(B)[right=3cm of A.east,anchor=west]{
    \begin{tikzpicture}[inner WD]
      \node(1)[pack]{$P$};
      \node(2)[below= of 1,pack]{$P$};
      \coordinate(l) at ($(1)!0.5!(2)$){};
      \draw(1) -- (2);
      \draw(1.north) -- ++(0, +0.5);
      \draw(2.south) -- ++(0, -0.5);

      \node(3)[right=4 of l,pack]{$P$};
      \draw(3.north) -- ++(0,+0.5);
      \draw(3.south) -- ++(0,-0.5);
      \path(l-|2.east) -- node[midway]{$\leq$} (3);
    \end{tikzpicture}
  };
\end{diagram*}

\subparagraph{Theory of a monoid} In \cref{ex.rgth_monoid} we outlined the theory of a monoid, and explored the corresponding regular calculus. Here in particular we wish to highlight the utility of graphical terms by drawing a graphical term below-right which represents the formula for commutativity of two variables below-left.

\begin{diagram*}[][][]
  \node(A)[]{$\exists_{m,m'}\big[\star(x,y,m)\wedge\star(y,x,m')\wedge m=m'\big]$};
  \node(C)[right= of A]{
    \begin{tikzpicture}[inner WD]
      \node(1)[pack]{$\star$};
      \node(2)[pack,right= of 1]{$\star$};
      \draw(1.east)--(2.west);
      \node(l1)[left= of 1,link]{};
      \draw(l1)-- ++(-1.0,0);
      \node(l2)[right= of 2,link]{};
      \draw(l2)-- ++(+1.0,0);
      \draw(l1) to[out=-45,in=270] (1);
      \draw(l2) to[out=-135,in=270] (2);
      \draw(l1) to[out=45,in=120] (2);
      \draw(l2) to[out=135,in=60] (1);
    \end{tikzpicture}
  };
\end{diagram*}

\subparagraph{Theory of a category} In \cref{ex.rgth_cat} we outlined the theory of regular a category, and produced a formula on pairs $f,g$ of morphisms corresponding to the property of the pair being mutually inverse $[f\circ g=\id]\wedge [g\circ f=\id]$. Rendered as a graphical term in the associated regular calculus this formula might look as follows.

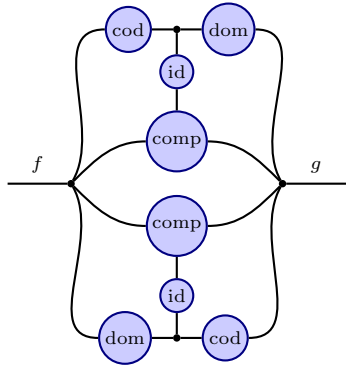
\begin{diagram*}[][][inner WD]
  \node(l1)[link]{};
  \node(c1)[pack]at ($(l1)+(5,+2)$) {$\comp$};
  \node(c2)[pack]at ($(l1)+(5,-2)$) {$\comp$};
  \node(l2) at ($(c1)+(5,-2)$)[link]{};
  \node(id1)[pack,above=of c1]{$\id$};
  \node(la1)[link,above= of id1]{};
  \node(cod1)[pack,left= of la1]{$\cod$};
  \node(dom1)[pack,right= of la1]{$\dom$};

  \node(id2)[pack,below=of c2]{$\id$};
  \node(la2)[link,below= of id2]{};
  \node(dom2)[pack,left= of la2]{$\dom$};
  \node(cod2)[pack,right= of la2]{$\cod$};

  \draw(c1.north) -- (id1.south);
  \draw(id1.north) -- (la1);
  \draw(cod1.east) -- (la1) -- (dom1.west);
  \draw(l1) to[out=60,in=180] (cod1);
  \draw(l1) to[out=45,in=180] (c1);
  \draw(l1) to[out=-45,in=180] (c2);
  \draw(l1) to[out=-60,in=180] (dom2);

  \draw(c2.south) -- (id2.north);
  \draw(id2.south) -- (la2);
  \draw(cod2.west) -- (la2) -- (dom2.east);
  \draw(l2) to[out=120,in=0] (dom1);
  \draw(l2) to[out=135,in=0] (c1);
  \draw(l2) to[out=-135,in=0] (c2);
  \draw(l2) to[out=-120,in=0] (cod2);

  \draw(l1) -- node[above]{$f$} ++(-3,0);
  \draw(l2) -- node[above]{$g$} ++(+3,0);
\end{diagram*}

\section{Graphical terms in the regular calculus of predicates}\label{sec.prd_graphical_terms}

We wish to highlight, as a special case, regular calculi of the form $\fprd\rr$ where $\rr$ is a prerelational po-category. Recall the definition of graphical terms (\cref{def.graphical_term}) and observe that a graphical term in such a regular calculus comprises the data of a wiring diagram $\omega\colon r_{1}\otimes\cdots\otimes r_{s}\to r_{\out}$ in $\rr$ as well as \emph{morphisms} $\{\theta_{i}\colon I\to r_{i}\}_{i\in\ord s}$ of $\rr$.

Note that the right ajax structure on $\rr(I,-)$ of \eqref{eqn.prd_right_ajax} has in particular $\true=\id_{I}$. As such, a graphical term $(\theta_{1},\,\ldots,\,\theta_{s};\omega)$ with $s>0$ represents the same predicate as the graphical term $(;\sigma\cp(\bigotimes \theta_{i})\cp\omega)$, namely $\church{(\theta_{1},\,\ldots,\,\theta_{s};\omega)}=\sigma\cp (\bigotimes \theta_{i})\cp\omega$, where $\sigma\colon I\to I\tpow{s}$ is the appropriate symmetry. As such, we have succeeded in obtaining an equality\footnote{Recall that graphical terms inherit an equality relation under taking of representations, $\church{-}$} of the graphical terms $(\theta_{1},\,\ldots,\,\theta_{s};\omega)=(;\sigma\cp(\bigotimes \theta_{i})\cp\omega)$, where the latter is merely the data of a wiring diagram $I\to r_{\out}$.

Of course we may read the equality $(\theta_{1},\,\ldots,\,\theta_{s};\omega)=(;\sigma\cp(\bigotimes \theta_{i})\cp\omega)$ ``the other way''. Given any wiring diagram $\omega\colon I\to r_{\out}$ -- that is, a graphical term of the form $(;\omega)$ -- we may construct the graphical term $(\omega;\id_{r_{\out}})$ such that $(;\omega)=(\omega;\id_{r_{\out}})$. Graphically this is the observation that any wiring diagram with empty domain $I$ may be equivalently re-drawn as a predicate.

Recall however that in the presence of the name-unfolding isomorphism of \cref{lemma.name_unfolding_iso}, the distinction between domain and codomain is fluid. With this final piece we are ready to observe that the data of graphical terms in $\fprd\rr$ is precisely the data of wiring diagrams in $\rr$, in the following sense.

\begin{lemma}\label{lemma.graphical_terms_in_prd}
  Let $\rr$ be a prerelational po-category and fix $s\in\nn$ and objects $r_{1},\,\ldots,\,r_{s},\, r_{\out}$ of $\rr$. Then, with graphical terms considered in $\fprd\rr$,
  \begin{lemenum}
  \item  the set of wiring diagrams $\omega\colon\bigotimes r_{i}\to r_{\out}$ in $\rr$ is in bijection with the set of graphical terms of the form $(;\omega'\colon I\to (\bigotimes r_{i})\otimes r_{\out})$, mediated by the assignment $\omega\mapsto(;\name{\omega})$,
  \item\label{lemma.graphical_terms_in_prd_ii} the set of graphical terms $(\theta_{1},\,\ldots,\,\theta_{s};\omega\colon\bigotimes r_{i}\to r_{\out})$ and the set of graphical terms of the form $(;\omega'\colon I\to r_{\out})$ admit the following opposed functions between them which preserve the represented predicate
    \begin{align*}  (\theta_{1},\,\ldots,\,\theta_{s};\omega)&\mapsto(;\sigma\cp(\textstyle\bigotimes\theta_{i})\cp\omega)\\
      (\omega';\id_{r_{\out}})&\mapsfrom(;\omega')
    \end{align*}
    where $\sigma\colon I\to I\tpow{s}$ is the appropriate symmetry of $\rr$.\hfill$\qed$
  \end{lemenum}\vspace{-\topsep}
\end{lemma}

The first bijection above is a strengthening of the graphical phenomenon we identified in \cref{ex.wd_as_preds}. The second correspondence, although not strictly a bijection of sets\footnote{but rather an isomorphism of setoids}, lends itself readily to graphical understanding.

\begin{example}[Graphical terms as wiring diagrams]
  For a given graphical term, the combination of the correspondence of \cref{lemma.graphical_terms_in_prd_ii} above and nesting (\cref{prop.diagrams_basic_nesting}) establishes that any IO shell containing a predicate may be re-drawn as a morphism of $\rr$ with empty domain $I$ without altering the represented predicate. For example, the following two graphical terms represent the same predicate.
  \begin{diagram*}[][][baseline=(A.base)]
    \node(A)
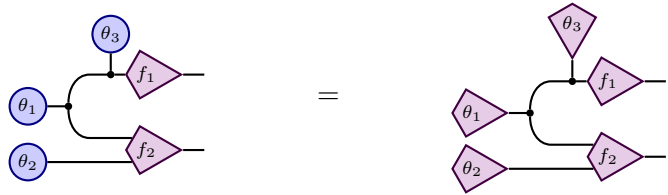
{
      \begin{tikzpicture}[inner WD]
        \node(t1)[pack]{$\theta_{1}$};
        \draw(t1.east) -- +(1, 0) node(l1)[link]{};
        \draw(l1) to[out=-90, in=180] +(1, -1.5) -- +(2, 0) node(f2)[oshellr,syntax,anchor=140]{$f_{2}$};
        \draw(f2.east) -- +(1, 0);
        \node(t2)[pack] at (t1|-f2.220) {$\theta_{2}$};
        \draw(t2.east) -- (t2-|f2.220);
        \draw(l1) to[out=+90, in=180] +(1, +1.5) -- +(1,0) node(l2)[link]{};
        \node(f1)[oshellr,syntax,anchor=west] at (l2-|f2.west) {$f_{1}$};
        \draw(f1.east) -- +(1, 0);
        \draw(l2) -- (f1.west);
        \node(t3)[pack, above of=l2, anchor=south]{$\theta_{3}$};
        \draw(t3.south) -- (l2);
      \end{tikzpicture}
    };
    \node(B)[right=3cm of A.east, anchor=west]{
      \begin{tikzpicture}[inner WD]
        \node(t1)[oshellr,syntax]{$\theta_{1}$};
        \draw(t1.east) -- +(1, 0) node(l1)[link]{};
        \draw(l1) to[out=-90, in=180] +(1, -1.5) -- +(2, 0) node(f2)[oshellr,syntax,anchor=140]{$f_{2}$};
        \draw(f2.east) -- +(1, 0);
        \node(t2)[oshellr,syntax] at (t1|-f2.220) {$\theta_{2}$};
        \draw(t2.east) -- (t2-|f2.220);
        \draw(l1) to[out=+90, in=180] +(1, +1.5) -- +(1,0) node(l2)[link]{};
        \node(f1)[oshellr,syntax,anchor=west] at (l2-|f2.west) {$f_{1}$};
        \draw(f1.east) -- +(1, 0);
        \draw(l2) -- (f1.west);
        \node(t3)[oshelld,syntax, above of=l2, anchor=south]{$\theta_{3}$};
        \draw(t3.south) -- (l2);
      \end{tikzpicture}
    };
    \path (A) --node[midway]{$=$} (B);
  \end{diagram*}
\end{example}

\section{The syntactic po-category of a regular calculus: a sketch}\label{sec.syn}

As we have seen, graphical terms provide an effective way to reason in regular calculi. It is thus of interest to consider forming, from a given regular calculus $\rc P$, a po-category $\fsyn\rc P$ whose \emph{objects} are the graphical terms of $\rc P$. In this way we may study the collection of representations for predicates at once -- that is, we may study the \emph{syntax} of the regular calculus.

In what follows we will sketch the ``syntactic po-category'' construction, but we will choose here to defer the full details to the companion paper. Consider then the following collections of data which together form the objects and morphisms of the \define{syntactic po-category} $\fsyn\rc P$ of the regular calculus $\rc P$.
\begin{equation}\label{eqn.syntactic_def}
  \begin{cases}
    \ob\fsyn\rc{P} &\coloneqq\{(\Gamma,p)\mid \Gamma\in\ob\cont{P}, p\in P(\Gamma)\}\\
    \fsyn\rc{P}\big( (\Gamma_1,p_1), (\Gamma_2,p_2)\big) &\coloneqq\{\theta_{12}\in P(\Gamma_1\otimes \Gamma_2)\mid\theta_{12}\leq p_1\boxplus p_2\}
  \end{cases}
\end{equation}

Of course to claim that these data form a po-category we must provide various additional structures and prove properties thereof. While it is possible to continue our sketch in the language of supplied morphisms and right ajax po-functors -- that is, semantically -- we will instead make use of the tools of graphical regular logic. Thus, in pictures, the to-be po-category $\fsyn\rc P$ has:
\begin{itemize}
\item objects $(\Gamma,p)$ represented by graphical terms
  \begin{tikzpicture}[inner WD, baseline=(g)]
    \node[pack] (rho) {$p$};
    \draw (rho.south) to node(g)[anchor=west]{$\Gamma$} +(0,-0.5);
  \end{tikzpicture}
\item morphisms $\theta_{12}\colon(\Gamma_1,p_1)\to(\Gamma_2,p_2)$ represented by graphical terms
  \begin{tikzpicture}[inner WD, pack size=4pt, baseline=(g.base)]
    \node[pack] (theta) {$\theta_{12}$};
    \draw (theta.west) to[pos=1] node(g)[left=-3pt] {$\Gamma_1$} +(-2pt, 0);
    \draw (theta.east) to[pos=1] node[right=-2pt] {$\Gamma_2$} +(2pt, 0);
  \end{tikzpicture}
  together with an entailment
  \begin{tikzpicture}[inner WD, baseline=(g.base)]
    \node (P1) {\begin{tikzpicture}[inner WD]
        \node[pack, pack size=6pt] (theta) {$\theta_{12}$};
        \draw (theta.west) to[pos=1] node(g)[left=-3pt] {$\Gamma_1$} +(-2pt, 0);
        \draw (theta.east) to[pos=1] node[right=-2pt] {$\Gamma_2$} +(2pt, 0);
      \end{tikzpicture}};
    \node (P2) [right=3 of theta] {
      \begin{tikzpicture}[inner WD]
        \node[pack] (p1) {$p_1$};
        \node[pack, right=of p1] (p2) {$p_2$};
        \node[outer pack, fit=(p1) (p2)] (outer) {};
        \draw[shorten >=-2pt] (p1.west) to[pos=1] node[left] {$\Gamma_1$} (outer.west);
        \draw[shorten >=-2pt] (p2.east) to[pos=1] node[right] {$\Gamma_2$} (outer.east);
      \end{tikzpicture}
    };
    \node[font=\normalsize] at ($(P1.east)!.5!(P2.west)$) {$\vdash$};
  \end{tikzpicture}

  \vspace{-.2in}
\item the identity on $(\Gamma,p)$ represented by the graphical term
  \begin{tikzpicture}[inner WD, baseline=(g.base)]
    \node[pack] (phi) {$p$};
    \node[link, below=3pt of phi] (dot) {};
    \node[outer pack, fit=(phi) (dot)] (outer) {};
    \draw (phi) -- (dot);
    \draw (dot) to[pos=1] node(g)[left] {$\Gamma$} (dot-|outer.west);
    \draw (dot) to[pos=1] node[right] {$\Gamma$} (dot-|outer.east);
  \end{tikzpicture}
\item the composite $\theta_{12}\cp\theta_{23}$ represented by the graphical term
  \begin{tikzpicture}[inner WD, pack size=4pt, baseline=(theta.base)]
    \node[pack] (theta) {$\theta_{12}$};
    \node[pack, right=2 of theta] (theta') {$\theta_{23}$};
    \node[outer pack, fit=(theta) (theta')] (outer) {};
    \draw (outer.east) to (theta'.east) node[right=1] {$\scriptstyle \Gamma_3$};
    \draw (theta) to node[above=-.3] {$\scriptstyle \Gamma_2$} (theta');
    \draw (theta.west) node[left=1] {$\scriptstyle \Gamma_1$} to (outer.west);
  \end{tikzpicture}
\end{itemize}

With the induced poset structure on the homs, in the companion paper \cite{grl2} we prove that the collections $\fsyn\rc P$ with the composition and identities above indeed form a po-category.

\begin{example}
  In \cref{ex.reg_calc_meet_sl} we established that $\wedge$-semilattices $L$ are equivalently regular calculi $(1,L)$. By unwinding the syntactic po-category construction for such a regular calculus we see that $\fsyn(1,L)$ is the po-category whose objects are the elements of $L$, whose hom posets $\fsyn(1,L)(l,l')$ are the down-sets $\mathop\downarrow\{l\wedge l'\}$, whose composition is meet, and whose identities are given by the top element of each hom poset.
\end{example}

The syntactic po-category moreover inherits from $\cont P$ a canonical symmetric monoidal structure where the monoidal product of objects, the monoidal product of morphisms, and the braiding correspond respectively to the following graphical terms -- for details see \cite{grl2}.
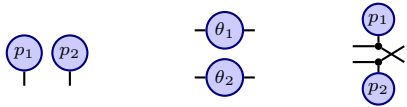
\begin{diagram*}
  \node(prod_obj){\begin{tikzpicture}[inner WD, baseline=(theta2.north)]
      \node[pack] (theta1) {$p_1$};
      \node[pack, right=.4 of theta1] (theta2) {$p_2$};
      \node[inner ysep=5pt, fit=(theta1) (theta2)] (outer) {};
      \draw (theta1.south) -- (theta1.south|-outer.south);
      \draw (theta2.south) -- (theta2.south|-outer.south);
    \end{tikzpicture}};
  \node(prod_mor)[right=1 of prod_obj]{    \begin{tikzpicture}[inner WD, baseline=(theta2.north)]
      \node[pack] (theta1) {$\theta_1$};
      \node[pack, below=.4 of theta1] (theta2) {$\theta_2$};
      \node[inner ysep=0pt, fit=(theta1) (theta2)] (outer) {};
      \draw (theta1.0) -- (theta1.0-|outer.east);
      \draw (theta1.180) -- (theta1.180-|outer.west);
      \draw (theta2.0) -- (theta2.0-|outer.east);
      \draw (theta2.180) -- (theta2.180-|outer.west);
    \end{tikzpicture}};
  \node(braiding)[right=1 of prod_mor]{\begin{tikzpicture}[inner WD, baseline=(dot2.north)]
      \node[link] (dot1) {};
      \node[link,below=.4 of dot1] (dot2) {};
      \node[pack,above=.3 of dot1,inner sep=1pt] (phi1) {$p_1$};
      \node[pack,below=.3 of dot2,inner sep=1pt] (phi2) {$p_2$};
      \node[inner xsep=3pt, inner ysep=-3pt, fit=(phi1) (phi2)] (outer) {};
      \draw (dot1) -- (phi1.south);
      \draw (dot2) -- (phi2.north);
      \draw (dot1) -- (dot1-|outer.west);
      \draw (dot1) -- (dot2-|outer.east);
      \draw (dot2) -- (dot2-|outer.west);
      \draw (dot2) -- (dot1-|outer.east);
    \end{tikzpicture}};
\end{diagram*}

With this symmetric monoidal structure we may induce, from the supply of $\ww$ in $\cont P$, a supply of $\ww$ in $\fsyn\rc P$. This work appears in \cite{grl2}, but we summarise the results here. Recall that $\epsilon$, $\delta$, $\eta$, and $\mu$ are the generating morphisms of $\ww$; see \eqref{eqn.generating_wires}. For an object $(\Gamma,p)\in\ob\fsyn\rc P$, the supplied morphisms corresponding to these generators are the following graphical terms.
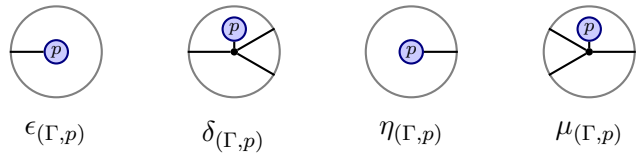
\begin{equation*}
  \begin{tikzpicture}[unoriented WD]
    \node (P1) {
      \begin{tikzpicture}[inner WD]
        \node[pack,inner sep=1pt] (phi1) {$p$};
        \node[outer pack, minimum size=7ex, fit=(phi1)] (outer) {};
        \draw (phi1.west) -- (outer.180);
      \end{tikzpicture}
    };
    \node[below=.1 of P1] {$\epsilon_{(\Gamma,p)}$};
    \node[right=3 of P1] (P2) {
      \begin{tikzpicture}[inner WD]
        \node[link] (dot1) {};
        \node[pack,above=.3 of dot1,inner sep=1pt] (phi1) {$p$};
        \node[outer pack, minimum size=7ex, fit=(dot1)] (outer) {};
        \draw (dot1) -- (phi1.south);
        \draw (dot1) -- (outer.180);
        \draw (dot1) -- (outer.30);
        \draw (dot1) -- (outer.-30);
      \end{tikzpicture}
    };
    \node[below=.1 of P2] {$\delta_{(\Gamma,p)}$};
    \node[right=3 of P2] (P3) {
      \begin{tikzpicture}[inner WD]
        \node[pack,inner sep=1pt] (phi1) {$p$};
        \node[outer pack, minimum size=7ex, fit=(phi1)] (outer) {};
        \draw (phi1.east) -- (outer.0);
      \end{tikzpicture}
    };
    \node[below=.1 of P3] {$\eta_{(\Gamma,p)}$};
    \node[right=3 of P3] (P4) {
      \begin{tikzpicture}[inner WD]
        \node[link] (dot1) {};
        \node[pack,above=.3 of dot1,inner sep=1pt] (phi1) {$p$};
        \node[outer pack, circle, minimum size=7ex, fit=(dot1)] (outer) {};
        \draw (dot1) -- (phi1.south);
        \draw (dot1) -- (outer.0);
        \draw (dot1) -- (outer.150);
        \draw (dot1) -- (outer.210);
      \end{tikzpicture}
    };
    \node[below=.1 of P4] {$\mu_{(\Gamma,p)}$};
  \end{tikzpicture}
\end{equation*}

In fact there is even more structure present, $\fsyn\rc P$ is a relational po-category. As a corollary of our main theorem we prove \cref{cor.vague_eqv} in the companion: there is a suitably natural equivalence of relational po-categories $\rr\eqv\fsyn\fprd\rr\eqv\fsyn(\bigsqcup_{J}\ww,F)$ for some ajax po-functor $F$ derived from $\rr$. By using the equivalence $\rrlpocat\eqv\rrgcat$ between relational po-categories and regular categories of \cref{thm.carboni_equivalence} we may therefore observe that we have obtained equivalences $\cat R\eqv\fladj\fsyn\rc P$ for some regular calculus $\rc P$ derived from $\cat R$.

Of interest here to the classical categorical logician is that we may work in regular categories using our graphical syntax for regular calculi. As this equivalence is one of regular categories, we may in particular compute limits graphically. For example, given two parallel left adjoints $h,k$ in $\fladj\fsyn\rc P$ we may compute the equaliser $\operatorname{eq}(h,k)\to\dom(h)$ and the image sub-object $\operatorname{im}(h)\rightarrowtail\cod(h)$ by the below-left and below-right graphical terms respectively.
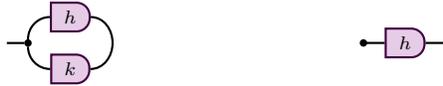
\begin{diagram*}
  \node(A){
    \begin{tikzpicture}[inner WD]
      \node(1)[funcr,syntax]{$h$};
      \node(2)[funcr,syntax,below= of 1]{$k$};
      \coordinate(3) at ($(1)!0.5!(2)+(2,0)$);
      \draw[out=0,in=90] (1) to (3) to[out=270,in=0] (2);
      \node(l)[link] at ($(1)!0.5!(2)-(2,0)$){};
      \draw[out=180,in=90] (1) to (l) to[out=270,in=180] (2);
      \draw(l) -- ++(-1,0);
    \end{tikzpicture}
  };
  \node(B)[right=3cm of A.east,anchor=west]{
    \begin{tikzpicture}[inner WD]
      \node(1)[funcr,syntax]{$h$};
      \draw(1.east) -- ++(1,0);
      \draw(1.west) -- ++(-1,0) node[link]{};
    \end{tikzpicture}
  };
\end{diagram*}

We conclude this section with some remarks about the nature of our syntactic po-category construction.

\begin{remark}
  Observe that $\pposet$ is a sub-$2$-category of $\CCat{Cat}$. Instead of our bespoke construction of the syntactic po-category above, it is tempting to consider some appropriate po-categorical variant of a monoidal Grothendieck construction -- perhaps as developed in \cite{moeller2018monoidal} or \cite{Buckley2013}. Indeed, at the level of objects it would seem that there is a coincidence between $\fsyn\rc P$ and the total space of a Grothendieck-type construction $\int\rc P$.

  However, it presently appears to the authors that any so-attempted recasting of $\fsyn$ is doomed to failure. In a putative Grothendieck construction, consider the pair of objects $(I,\true)$ and $(\Gamma,p)$. In order for the total space to supply $\ww$, we would require the presence of morphisms $\widehat\epsilon\colon(\Gamma,p)\to(I,\true)$ and $\widehat\eta\colon(I,\true)\to(\Gamma,p)$. However, in general we have only $P(\epsilon)(p)\vdash\true$ and $p\vdash P(\eta)(\true)$, and thus there appears to be no uniform way to select the direction of the inequalities for morphisms in $\int\rc P$.

  In this way, some form of `symmetrisation' of domain and codomain becomes necessary, considerations of which result in our $\fsyn\rc P$.
\end{remark}

Despite the fact that in later sections and the companion paper we shall realise $\fsyn$ as a highly-structured $2$-functor, it fails to mediate any form of categorical equivalence. That is, the following example stands to establish that a regular calculus is \emph{more} than the data of its graphical terms. Of course this is completely expected from the syntax-semantics point of view, and stands in directly analogy with other means of describing theories: different logical axioms can result in the same theory or have the same semantics. For an example of the former, for the theory of a pre-order we see that the formulations $x P y\wedge y P z\vdash_{x,y,z} x Pz$ and $\exists_{y}[x P y\wedge y P z]\vdash_{x,z} x Pz$ of transitivity imply the same theory. For an example of the latter, a finite product sketch with $A=B\times C$ is semantically equivalent with one in which $A=B\times C$ and $A=C\times B$. In this way, the appropriate notion of ``same-ness'' for regular calculi is a form of ``Morita equivalence'' as detected by $\fsyn$; see \cref{thm.model_equiv} for a theorem from this perspective.

\begin{example}[$\fsyn$ identifies distinct regular calculi]
  Recall \cref{ex.reg_calc_meet_sl}, that is, that $\wedge$-semilattices $L$ are equivalently regular calculi $L\colon1\to\pposet$. Consider then that we may form the degenerate regular calculus $(\cc,\cc\To{!}1\To{L}\pposet)$ for any $\cc$ supplying $\ww$. The syntactic po-categories of these degenerate regular calculi are equally degenerate: in $\fsyn(\cc,\cc\To{!}1\To{L}\pposet)$ there is a (natural) isomorphism $(c,p)\iso(c',p)$ for all elements $p\in L$ and objects $c,c'\in\ob\cc$, viz, $p$ itself.

  It may be checked that given $\cc$ and $\dd$ inequivalent symmetric monoidal po-categories supplying $\ww$, the regular calculi $(\cc,!_{\cc}\cp L)$ and $(\dd,!_{\dd}\cp L)$ are inequivalent. However, the po-functor $\fsyn(\cc,!_{\cc}\cp L)\to\fsyn(\dd,!_{\dd}\cp L)$ which sends $(c,p)\mapsto (I^{\dd},p)$ and $\theta\mapsto\theta$ mediates an equivalence (with evident inverse) of symmetric monoidal po-categories supplying $\ww$. Thus $\fsyn$ cannot mediate a $2$-dimensional equivalence of $2$-categories.
\end{example}

\begin{remark}
  In fact this example is part of a more general class. In \cref{prop.syn_is_relational} below we record a result of the companion paper: whenever there is a morphism of regular calculi $(F,F^{\sharp})\colon\cont P\to\cont Q$ such that $F$ is essentially surjective (\cref{def.po_equiv}) and $F^{\sharp}$ is an isomorphism then $\fsyn(F,F^{\sharp})\colon\fsyn\cont P\to\fsyn\cont Q$ mediates an equivalence. This result corresponds to our understanding above: the essential surjectivity of $F$ says roughly that $\cont P$ has enough contexts to model those of $\cont Q$, and the invertibility of $F^{\sharp}$ says that we may bijectively translate relations across contexts from $\cont P$ to $\cont Q$ without altering semantics. At worst then, $\cont P$ contains somehow redundant copies of contexts of $\cont Q$, but without a difference in semantics.

  From this the above example may be derived. Observe that for any po-category $\cc$ supplying $\ww$, the canonical morphism $(!_{\cc},\id_{L})\colon(\cc,!_{\cc}\cp L)\to(1,L)$ of regular calculi satisfies these conditions. Thus we may conclude that $\fsyn(\cc,!_{\cc}\cp L)\eqv\fsyn(1,L)$.
\end{remark}

\section{Morphisms of regular calculi \texorpdfstring{\&}{and} graphical terms}\label{sec.graphical_morphisms}

Our notions of $1$- and $2$-morphisms of regular calculi, \cref{def.reg_calc_morphisms}, interact well with graphical terms and indeed preserve all of the desired structure. We have just seen the sense in which graphical terms in a regular calculus are meaningfully the objects of a syntactic po-category \eqref{eqn.syntactic_def}, so we now elucidate the manner in which morphisms of regular calculi $\rc P\to \rc Q$ act on graphical terms.

Given a morphism $(F,F^{\sharp})\colon \rc P\to\rc Q$ of regular calculi and a collection of contexts $\{\Gamma_{i}\}_{i\in\{1,\,\ldots,\,s,\,\out\}}$ in $\cont P$, the monoidal $2$-naturality of $F^{\sharp}$ renders commutative the following diagram.
\begin{diagram*}
  \node(1)[]{$\prod P(\Gamma_{i})$};
  \node(s)[right=2cm of 1.east]{};
  \node(2)[right=2cm of s,anchor=west]{$P(\bigotimes \Gamma_{i})$};
  \node(3)[right=2cm of 2,anchor=west]{$P(\Gamma_{\out})$};
  \node(4)[below= of 1]{$\prod P'F(\Gamma_{i})$};
  \node(5)[anchor=base] at (s|-4.base) {$P'(\bigotimes' F(\Gamma_{i}))$};
  \node(6)[below= of 2]{$P'F(\bigotimes \Gamma_{i})$};
  \node(7)[below= of 3]{$P'F(\Gamma_{\out})$};
  \draw[a](1)to node[la,above]{$\boxplus$}(2);
  \draw[a](2)to node[la,above]{$P(\omega)$}(3);
  \draw[a](1)to node[la,left]{$\prod F^{\sharp}_{\Gamma_{i}}$}(4);
  \draw[a](3)to node[la,right]{$F^{\sharp}_{\Gamma_{\out}}$}(7);
  \draw[a](4)to node[la,below]{$\boxplus'$}(5);
  \draw[a](5)to node[la,below]{$P'\varphi$}(6);
  \draw[a](6)to node[la,below]{$P'F(\omega)$}(7);
\end{diagram*}

Thus, given a graphical term $(\theta_{1},\,\ldots,\,\theta_{s};\omega)$ of $\rc P$ where $\theta_{i}\in P(\Gamma_{i})$, we see that we obtain the graphical term $(F^{\sharp}_{\Gamma_{1}}(\theta_{1}),\,\ldots,\,F^{\sharp}_{\Gamma_{s}}(\theta_{s});\varphi\cp F(\omega))$ of $\rc Q$ with the property \[\church{(F^{\sharp}_{\Gamma_{1}}(\theta_{1}),\,\ldots,\,F^{\sharp}_{\Gamma_{s}}(\theta_{s});\varphi\cp F(\omega))}=F^{\sharp}_{\Gamma_{\out}}(\church{(\theta_{1},\,\ldots,\,\theta_{s};\omega)})\ .\]
The fact that $\varphi$ is the strongator of the supply-preserving strong symmetric monoidal po-functor $F$ affords us an easy graphical understanding of this action. First we replace all the predicates $\theta_{i}\in P(\Gamma_{i})$ in IO shells with the predicates $F^{\sharp}_{\Gamma_{i}}(\theta_{i})$, and then when $\omega$ is composed of tensors of morphisms $\omega=\bigotimes\omega_{i}$, we may ``pull $\varphi$ through the tensors'' in $F(\omega)$ and preserve wiring as we go. These principles are illustrated by the following example.
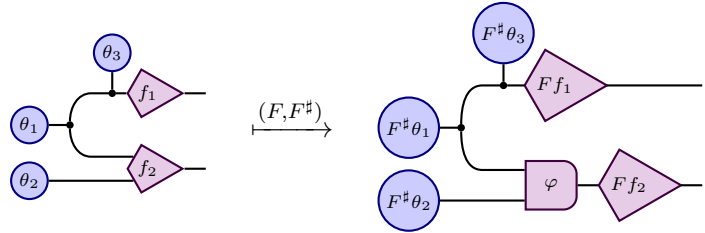
\begin{diagram*}
  \node(G1) {
    \begin{tikzpicture}[inner WD]
      \node(t1)[pack]{$\theta_{1}$};
      \draw(t1.east) -- +(1, 0) node(l1)[link]{};
      \draw(l1) to[out=-90, in=180] +(1, -1.5) -- +(2, 0) node(f2)[oshellr,syntax,anchor=140]{$f_{2}$};
      \draw(f2.east) -- +(1, 0);
      \node(t2)[pack] at (t1|-f2.220) {$\theta_{2}$};
      \draw(t2.east) -- (t2-|f2.220);
      \draw(l1) to[out=+90, in=180] +(1, +1.5) -- +(1,0) node(l2)[link]{};
      \node(f1)[oshellr,syntax,anchor=west] at (l2-|f2.west) {$f_{1}$};
      \draw(f1.east) -- +(1, 0);
      \draw(l2) -- (f1.west);
      \node(t3)[pack, above of=l2, anchor=south]{$\theta_{3}$};
      \draw(t3.south) -- (l2);
    \end{tikzpicture}
  };
  \node(F1)[right=2cm of G1.east,anchor=west] {
    \begin{tikzpicture}[inner WD]
      \node(t1)[pack]{$F^{\sharp}\theta_{1}$};
      \draw(t1.east) -- +(1, 0) node(l1)[link]{};
      \draw(l1) to[out=-90, in=180] +(1, -2) -- +(2, 0) node(f2)[funcr,syntax,anchor=150,inner sep=7pt]{$\varphi$};
      \draw(f2.east) -- +(1, 0) node(f2p)[oshellr,syntax,anchor=west, shrink]{$Ff_{2}$};
      \draw(f2p.east) -- +(1, 0);
      \node(t2)[pack] at (t1|-f2.210) {$F^{\sharp}\theta_{2}$};
      \draw(t2.east) -- (t2-|f2.210);
      \draw(l1) to[out=+90, in=180] +(1, +2) -- +(1,0) node(l2)[link]{};
      \node(f1)[oshellr,syntax,anchor=west, shrink] at (l2-|f2.west) {$Ff_{1}$};
      \draw(f1.east) -- (f1.east-|f2p.east) -- +(1, 0);
      \draw(l2) -- (f1.west);
      \node(t3)[pack, above of=l2, anchor=south]{$F^{\sharp}\theta_{3}$};
      \draw(t3.south) -- (l2);
    \end{tikzpicture}
  };
  \path (G1) --node[midway]{$\xmapsto{(F,F^{\sharp})}$} (F1);
\end{diagram*}

This description suggests that morphisms of regular calculi preserve the connectivity, wiring, and compositionality of graphical terms, and so all the structure present in our syntactic po-categories. Indeed, as we prove in the companion paper \cite{grl2}, such a morphism $(F,F^{\sharp})\colon\rc P\to\rc Q$ of regular calculi induces a symmetric monoidal supply preserving po-functor $\fsyn(F,F^{\sharp})\colon\fsyn\rc P\to\fsyn\rc Q$. In this fashion we may prove that $\fsyn$ forms a $2$-functor $\fsyn\colon\rrgcalc\to\rrlpocat$, and observe moreover that it sends certain morphisms of regular calculi to equivalences.

\begin{proposition}[{\cite{grl2}}]\label{prop.syn_is_relational}
  The syntactic po-category construction is the on-objects component of a $2$-functor $\fsyn\colon\rrgcalc\to\rrlpocat$. Moreover, if $(F,F^{\sharp})\colon\cont P\to\cont Q$ is a morphism of regular calculi such that $F$ is essentially surjective and $F^{\sharp}$ is an isomorphism then $\fsyn(F,F^{\sharp})$ is an equivalence.
\end{proposition}

\chapter{Preview of the next paper}\label{chap.preview}

Now that we have seen all of the major players in the theory of regular calculi, let us give an overview of the results and details developed in the companion paper \cite{grl2}.

To begin with, in the companion we give the full account of the construction of the syntactic po-category $\fsyn\rc P$ of a regular calculus $\rc P$ and prove that this extends to a $2$-functor to relational po-categories, subsuming the outlines of \cref{sec.syn,sec.graphical_morphisms} here.

Using this and the $2$-functor $\fprd$ of \cref{sec.prd} here, we turn our attention to rigorously stating and proving our main theorem which connects regular calculi to relational po-categories, and so to regular categories.

\begin{theorem}\label{thm.main}
  The $2$-functors $\fsyn\colon\rrgcalc\to\rrlpocat$ and $\fprd\colon\rrlpocat\to\rrgcalc$ are involved in a bi-adjunction $\fsyn\biadj\fprd$. Moreover, the co-unit of the bi-adjunction is part of an adjoint equivalence so that this bi-adjunction is pseudo-reflection of $\rrlpocat$ into $\rrgcalc$.
\end{theorem}

This main theorem allows us to deduce almost all of the rest of our results. For instance, we use it to give a $2$-dimensional account of the regular completion of a category with finite limits.

\begin{theorem}\label{thm.reg_lex}
  The forgetful $2$-functor $U\colon\rrgcat\to\fflcat$ is right bi-adjoint to the composite
  \[ \fflcat\To{\fspanpo}\pprlpocat\To{\fprd}\rrgcalc\To{\fsyn}\rrlpocat\To{\fladj}\rrgcat\]
  where $\fflcat$ is the category of categories with finite limits.
\end{theorem}

As a corollary of this result, we are able to show that taking the regular category of left adjoints in a relational po-category is in fact bi-represented.

\begin{corollary}\label{cor.ladj_birep}
  The $2$-functor $\fladj\colon\rrlpocat\to\rrgcat$ is bi-represented by the relational po-category $\fsyn\fprd\ww$, that is, there is a pseudo-natural adjoint equivalence of the $2$-functors
  \[ \rrlpocat(\fsyn\fprd\ww,-)\aeqv\fladj\colon\rrlpocat\to\rrgcat\ .\]
\end{corollary}

Our main theorem also affords us an important ``bare-ification'' result for regular calculi: every relational po-category is equivalent to the syntactic po-category of a bare regular calculus. As such, we are entitled to reason in an arbitrary relational po-category using graphical terms with \emph{no} kites.

\begin{corollary}\label{cor.vague_eqv}
  Let $\rr$ be a relational po-category, then there are pseudo-natural equivalences in $\rrlpocat$
  \[
    \rr \eqv \fsyn\fprd\rr\eqv \fsyn\big(\textstyle\bigsqcup_{\ob\rr}\ww,\ s^{\sqcup}\cp \rr(I,-)\big)\ .
  \]
\end{corollary}

With these tools, in the companion we then return to the topic of regular theories and their connection to and realisations as regular calculi. We connect our work on regular calculi to the classical devices of categorical regular logic by first proving the following theorem.

\begin{theorem}\label{thm.syn_equiv}
  Given a regular theory $(\Sigma,\mathbb T)$, there is an equivalence of categories between the syntactic regular category $\cat C^{\mathrm{reg}}_{\mathbb T}$ associated to the regular theory and the regular category of left adjoints $\fladj\fsyn(\cc_{\mathbb T},F_{\mathbb T})$ in the syntactic po-category of the regular calculus associated to the regular theory (\cref{con.rgcalc_rgtheory}).
\end{theorem}

To address the question of models of a regular calculus, we first make the following definition.

\begin{definition}\label{def.models}
  The $2$-category of models of a regular calculus $\rc P$ in a relational po-category $\rr$ is defined to be $\rc P\model(\rr)\coloneqq\rrlpocat(\fsyn\rc P,\rr)$. In this way, $\rc P\model(\cdot)$ extends to a $2$-functor $\rc P\model(\cdot)\colon\rrlpocat\to\CCat{Cat}$.
\end{definition}

Using this, our main theorem, \cref{thm.main}, and \cref{thm.syn_equiv} we are able to prove that models in our sense suitably agree with and therefore generalise the classical notion of models of a regular theory.

\begin{theorem}\label{thm.model_equiv}
  Let $(\Sigma, \mathbb T)$ be a regular theory. There is an equivalence of categories \[(\cc_{\mathbb T},F_{\mathbb T})\text-\operatorname{Mod}(\frel\cat R)\eqv \mathbb T\text-\operatorname{Mod}(\cat R)\] pseudo-natural in regular categories $\cat R$.
\end{theorem}

\printbibliography

\end{document}